\documentclass[12pt,a4paper,intlimits,sumlimits]{amsart}

\usepackage{amsmath, amsthm, mathtools}
\usepackage{amsfonts}
\usepackage[alphabetic,nobysame]{amsrefs}
\usepackage{graphicx}
\usepackage{framed}
\usepackage{amssymb}
\usepackage{esvect}
\usepackage{xcolor}
\usepackage{eucal}

\usepackage{hyperref}
\usepackage[english]{babel}

\def \N {\mathbb{N}}
\def \R {\mathbb{R}}

\newcommand{\tr}{\operatorname{tr}}
\newcommand{\supp}{\operatorname{supp}}

\newcommand{\sech}{\operatorname{sech}}
\newcommand{\PV}{\operatorname{PV}}

\newtheorem{theorem}{Theorem}[section]
\newtheorem{proposition}[theorem]{Proposition}
\newtheorem{lemma}[theorem]{Lemma}
\newtheorem{corollary}[theorem]{Corollary}
\newtheorem{openpb}[theorem]{Open Problem}

\theoremstyle{definition}

\numberwithin{equation}{section}

\usepackage{geometry}
\renewcommand{\epsilon}{\varepsilon}
\newcommand{\e}{\varepsilon}

\renewcommand{\leq}{\leqslant}
\renewcommand{\le}{\leqslant}

\renewcommand{\ge}{\geqslant}

\allowdisplaybreaks

\title[Kernel representations with variable coefficients]{Kernel representations \\ of fractional gradient bilinear forms.\\
Part II. Variable coefficients}

\author{Serena Dipierro, Caterina Sportelli, and Enrico Valdinoci}

\address{Serena Dipierro: Department of Mathematics and Statistics, The University of Western Australia, 35 Stirling Highway, Crawley, Perth, WA 6009, Australia}
\email{serena.dipierro@uwa.edu.au}

\address{Caterina Sportelli: Departamento de Análisis Matemático, Universidad de Granada, 18071 Granada, Spain  \& Department of Mathematics and Statistics, The University of Western Australia, 35 Stirling Highway, Crawley, Perth, WA 6009, Australia}
\email{caterina.sportelli@uwa.edu.au, caterina.sp@ugr.es}

\address{Enrico Valdinoci: Department of Mathematics and Statistics, The University of Western Australia, 35 Stirling Highway, Crawley, Perth, WA 6009, Australia}
\email{enrico.valdinoci@uwa.edu.au}

 \date{}

\begin{document}

\begin{abstract}
This paper investigates the existence and properties of interaction kernels representing nonlocal energies associated with the Riesz fractional gradient and variable, uniformly elliptic matrix fields $A(x)$. Building on our previous analysis of constant matrices in \cite{DSVPART1}, we establish necessary conditions on the singularity and decay of such kernels, derive an explicit representation formula under suitable regularity assumptions, and prove uniqueness
of the kernel representation.

We also obtain sharp asymptotics and positivity criteria near the diagonal, averaged decay estimates at infinity, and characterize the divergence free property of $A(x)$ through the decay of the kernel.

Finally, we show that the associated energy may fail the truncation property and that arbitrarily small smooth perturbations of constant matrices can destroy the nonnegativity of the representing kernel.

These results provide negative answers to both~\cite[Open Problems~1.9 and~1.10]{MR3819528}.
\end{abstract}
 
\maketitle

\tableofcontents

\section{Introduction}

In the classical theory of elliptic equations, divergence form operators of the type
\[
-\mbox{div}(A(x)\nabla u),
\]
where $A:\R^N\to\R^{N\times N}$ is a matrix (or, more precisely, a matrix-valued function), provide a natural framework for describing anisotropy and spatial inhomogeneity. A fractional counterpart of this construction can be obtained by replacing the classical gradient with the Riesz fractional gradient. For $s\in(0,1)$ and $u\in C_c^\infty(\R^N)$, the latter is defined by
\[
D^s u(z):=d_{N,s}\int_{\R^N}(u(z)-u(x))\frac{z-x}{|z-x|^{N+s+1}} dx,
\]
where
\begin{equation}\label{DNSDEFN}
d_{N,s}:=\frac{2^s\Gamma\left(\frac{N+s+1}{2}\right)}{\pi^{N/2}\Gamma\left(\frac{1-s}{2}\right)}
\end{equation}
(we refer to \cite[Section 1.5]{GENTLE} and the references therein for further background on fractional gradients).

The Riesz fractional gradient has been extensively studied as a vector-valued counterpart of the classical gradient, particularly in connection with fractional divergence operators and weak formulations. This framework has proved useful in the study of functions of bounded variation, the development of fractional Sobolev, Hardy, Trudinger and Morrey inequalities, and the analysis of $\Gamma$-convergence (see e.g. \cite{MR3403430, MR3420498, MR3615452, MR3714833, MR3819528, MR4001075, MR4048032, MR4104831, MR4179861}). 
In a related direction, \cite{MR5007887} constructs nonlocal operators through suitable interaction kernels, recovering classical elliptic operators in divergence form in the limit $s\nearrow1$. This provides another connection between nonlocal kernels and elliptic operators with matrix-valued coefficients.

Despite these developments, the representation of energies involving variable matrix fields through suitable interaction kernels remains a delicate question. Indeed, this construction leads naturally to energies of the form
\[
\int_{\R^N}A(x)D^su(x)\cdot D^su(x) dx,
\]
whose positivity and coercivity follow from the uniform ellipticity of $A$.  However, the uniform ellipticity alone does not guarantee, in general, that this energy can be represented in terms of pairwise differences $u(x)-u(y)$ through a nonnegative symmetric kernel.

For the fractional Laplacian, such interactions are described by a positive symmetric kernel proportional to $|x-y|^{-N-2s}$ through the constant
\begin{equation}\label{,mknjbvcxzdxcfvbn}
c_{N,s}=\frac{2^{2s}s\Gamma\left(\frac{N+2s}{2}\right)}{\pi^{N/2}\Gamma(1-s)}.
\end{equation}
This representation plays a fundamental role in the analysis of nonlocal equations, as it makes the interactions explicit and allows the use of techniques that rely on the positivity and decay of the kernel. For general matrices with varible coefficients, it is not clear whether an analog representation exists. Indeed, the positivity of the energy does not, by itself, guarantee positivity of the underlying interaction kernel.

The present paper is the second part of our investigation of this question. In \cite{DSVPART1}, we considered the case of constant matrices, deriving an explicit formula for the associated kernel and studying how its sign depends on the constant matrix $A$. Here we turn to the variable coefficients case, which introduces new difficulties, both in identifying the kernel and in understanding its behavior near the diagonal and at infinity. Our starting point is the following question, raised in \cite[Open Problem 1.10]{MR3819528}:

\begin{openpb}\label{opeSPJMDPKM}
Let $s\in(0,1)$. Let $A:\R^N\to\R^{N\times N}$ be bounded and measurable. Assume that~$A$ is uniformly elliptic, namely that there exist constants $0<\lambda_A\le\Lambda_A<+\infty$ such that, for every $x, \xi\in\R^N$,
\begin{equation}\label{uniformlyellipticdefn}
\lambda_A |\xi|^2\le A(x)\xi\cdot\xi\le\Lambda_A |\xi|^2.
\end{equation}
Can one find a nonnegative kernel $k_A$ such that
\begin{equation}\label{lòkgyftdresdetfyguhijok}
k_A(x,y)=k_A(y,x),
\end{equation}
and, for some $\alpha\in(0,2)$, $\eta>0$, $0<\lambda\le\Lambda<+\infty$, and~$M>0$,
\begin{alignat}{2}\label{srdtfyguioplèòlkoijhugyftr1}
&\lambda\le k_A(x,y)|x-y|^{N+\alpha}\le\Lambda,\quad&&\mbox{ for } |x-y|\le 1,\\ \label{srdtfyguioplèòlkoijhugyftr2}
\mbox{ and }\quad &k_A(x,y)\le M|x-y|^{-N-\eta},\quad&&\mbox{ for } |x-y|>1,
\end{alignat}
such that
\begin{equation}\label{mnbhgvcfdxszxdfcgvhb}
\int_{\R^N} A(x)D^s u(x)\cdot D^s v(x)\,dx = \iint_{\R^{2N}}(u(x)-u(y))(v(x)-v(y))k_A(x,y)\,dx\,dy
\end{equation}
for every $u,v\in C_c^\infty(\R^N)$?
\end{openpb}

We show that the answer to this question is negative in general. Our analysis also provides a detailed description of the properties of representing kernels,  starting with the restrictions on their singularity near the diagonal and their decay at infinity.

\subsection{Singularity and decay of the kernel}
Here we provide necessary conditions that can be obtained without an explicit representation of the kernel $k_A(x, y)$. In particular, we determine the only possible order of its singularity at the diagonal and we derive restrictions on its decay at infinity. These results show that the order $2s$ plays a distinguished role, both in the local behavior of the kernel and in its decay at infinity.

Throughout the paper we will use the notation
\[
\Delta:=\{(x,x):x\in\R^N\}.
\]

To begin with,  we show that the singularity exponent in \eqref{srdtfyguioplèòlkoijhugyftr1} is necessarily determined by the order of the fractional gradient. More precisely,  we have the following:

\begin{theorem}\label{0-pkr32pvb8mn7Jhb398fklsOdDTH}
Let $s\in(0,1)$ and let $A:\R^N\to\R^{N\times N}$ be bounded and measurable and satisfy~\eqref{uniformlyellipticdefn}.

Let $k:\R^{2N}\setminus\Delta\to\R$ be measurable and satisfy~\eqref{srdtfyguioplèòlkoijhugyftr1} and
\begin{equation}\label{kmljnbhgvcfdxcfvgbh}
|k(x,y)|\le M|x-y|^{-N-\eta}, \quad\mbox{ for } |x-y|>1,
\end{equation}
for some $\alpha\in(0,2)$, $\eta>0$, $0<\lambda\le\Lambda<+\infty$ and $M>0$.

In addition, assume that, for all $u\in C^\infty_c(\R^N)$,
\begin{equation}\label{OJS02irkt-3}
\int_{\R^N} A(x)D^s u(x)\cdot D^s u(x) \,dx = \iint_{\R^{2N}}(u(x)-u(y))^2\,k(x,y)\,dx\,dy.
\end{equation}

Then, $\alpha=2s$.
\end{theorem}

We stress that the conclusion $\alpha=2s$ in Theorem~\ref{0-pkr32pvb8mn7Jhb398fklsOdDTH} follows under assumptions weaker than those required in Open problem~\ref{opeSPJMDPKM}. Indeed, our proof neither uses the symmetry condition~\eqref{lòkgyftdresdetfyguhijok}, nor requires the kernel to be nonnegative on the whole of~$\R^{2N}\setminus\Delta$.
The positivity of the kernel is only needed near the diagonal, and it is already contained in the lower bound in~\eqref{srdtfyguioplèòlkoijhugyftr1}.

We next prove that the exponent $\eta$ appearing in~\eqref{srdtfyguioplèòlkoijhugyftr2} cannot be taken strictly larger than~$2s$.

\begin{theorem}\label{sxdfcghvbjknlmò,à.,lmknjbhvgcfdxszxfcgvbhn}
Let $s\in(0,1)$ and let $A:\R^N\to\R^{N\times N}$ be bounded and measurable and satisfy~\eqref{uniformlyellipticdefn}.

Let $k:\R^{2N}\setminus\Delta\to\R$ be measurable and satisfy
\begin{equation}\label{mlnjkbhgvcfdxszexdcfgvhbjnkml,}
k(x,y)\le \Lambda |x-y|^{-N-2s}, \quad \mbox{ for } 0<|x-y|\le1,
\end{equation}
and \begin{equation}\label{kmljnbhgvcfdxcfvgbhBIS}
k(x,y)\le M|x-y|^{-N-\eta}, \quad\mbox{ for } |x-y|>1,
\end{equation}
for some $\eta>0$, $\Lambda>0$ and $M>0$. 

In addition, assume~\eqref{OJS02irkt-3}. 

Then, $0<\eta\le2s$.
\end{theorem}

We point out that, as in Theorem \ref{0-pkr32pvb8mn7Jhb398fklsOdDTH},
also in Theorem~\ref{sxdfcghvbjknlmò,à.,lmknjbhvgcfdxszxfcgvbhn} neither the symmetry assumption~\eqref{lòkgyftdresdetfyguhijok} nor the global nonnegativity of the kernel is needed.
Moreover, the assumptions required in Theorem~\ref{sxdfcghvbjknlmò,à.,lmknjbhvgcfdxszxfcgvbhn} are weaker than
those used in Theorem \ref{0-pkr32pvb8mn7Jhb398fklsOdDTH}. Indeed, in Theorem \ref{0-pkr32pvb8mn7Jhb398fklsOdDTH}, the lower bound in \eqref{srdtfyguioplèòlkoijhugyftr1} is essential, since it provides a positive contribution which prevents cancellations. In contrast,
Theorem \ref{sxdfcghvbjknlmò,à.,lmknjbhvgcfdxszxfcgvbhn} only requires an upper bound on the kernel, since the lower bound of the relevant energy is provided by the uniform ellipticity of the matrix $A(x)$.

Theorem \ref{sxdfcghvbjknlmò,à.,lmknjbhvgcfdxszxfcgvbhn} shows that, if the kernel at infinity is controlled from above by a quantity proportional to $|x-y|^{-N-\eta}$, then necessarily $\eta\le 2s$. We will also prove the complementary statement, namely if, for $|x-y|>1$, the kernel is bounded from below by a positive multiple of $|x-y|^{-N-\gamma}$, then necessarily $\gamma\ge 2s$ (see the forthcoming Theorem \ref{NONNETH}).
Accordingly, if the kernel is bounded from above and below by positive multiples of the same power of $|x-y|$, then that exponent must be exactly~$2s$ (see the forthcoming Corollaries~\ref{cor1tails} and~\ref{cor2tails}).

\subsection{An explicit representation formula}
The results discussed above provide necessary information on the singularity and the decay of any kernel representing the quadratic form \eqref{OJS02irkt-3}, but do not give an explicit expression for this kernel.  Thus, we next turn to the identification of the kernel in terms of the matrix $A(x)$. 

We first work only assuming that $A$ is bounded, measurable and symmetric. Thus, given~$\varepsilon>0$, we define the truncated kernel
\begin{equation}\label{kerneltruncated}
k_{A,\varepsilon}(x,y):=-\frac{d_{N,s}^2}{2}\int_{\R^N\setminus(B_\varepsilon(x)\cup B_\varepsilon(y)) } \frac{(z-x)\cdot A(z)(z-y)}{|z-x|^{N+s+1}|z-y|^{N+s+1}} \,dz.
\end{equation}
We will show that,  if a symmetric kernel $k_A(x, y)$ represents the bilinear form associated with $A(x)$, then the truncated kernels $k_{A,\varepsilon}(x, y)$ recover $k_A(x, y)$ at the level of the corresponding bilinear forms (see the forthcoming Theorem \ref{thmA(x)identification}).

We stress that, without imposing any continuity assumption on $A(x)$, the kernel is identified in a weak sense through the family $k_{A,\varepsilon}(x, y)$. However, this result does not yet provide a pointwise formula for the kernel $k_A(x, y)$. 
To obtain such a representation, we impose an additional local regularity assumption on $A(x)$. For every compact set $K\subset\R^N$, we
set
\begin{equation}\label{omegaAkdefn}
\omega_{A,K}(r):=\sup_{\substack{z,w\in K\\ |z-w|\le r}}|A(z)-A(w)|
\end{equation}
and we assume that $A$ satisfies the uniform fractional Dini condition
\begin{equation}\label{mkojihguyftdreswxrdcftvgybhjn}
\int_0^1 \frac{\omega_{A, K}(r)}{r^{1+s}} \,dr <+\infty
\end{equation}
for every compact set $K\subset\R^N$. We note that this condition implies continuity.

Under assumption \eqref{mkojihguyftdreswxrdcftvgybhjn},  we prove that the singularities at $z=x$ and $z=y$ can be controlled separately.  More precisely,  in the forthcoming Lemma~\ref{lemmareg1}, we show that there exists the limit
\begin{equation}\label{INT2PVdefn}
I_A(x,y):=\lim_{\varepsilon\searrow0}\int_{\R^N\setminus\left(B_\varepsilon(x)\cup B_\varepsilon(y)\right)}
\frac{(z-x)\cdot A(z)(z-y)}{|z-x|^{N+s+1}|z-y|^{N+s+1}} \,dz
\end{equation}
and it is finite. This  allows us to pass from the identification at the level of bilinear forms to an almost everywhere pointwise identification of the kernel. More precisely, we have the following:

\begin{theorem}\label{thmreg2}
Let~$A:\R^N\to \R^{N\times N}$ satisfy \eqref{mkojihguyftdreswxrdcftvgybhjn} for every compact set~$K\subset\R^N$.
Assume that there exists $k_A\in L^1_{loc}(\R^{2N}\setminus\Delta)$ satisfying \eqref{lòkgyftdresdetfyguhijok}, and such that~\eqref{mnbhgvcfdxszxdfcgvhb} holds.

Then, for almost every $x\ne y$,
\begin{equation}\label{CHARKE}
k_A(x,y)=-\frac{d_{N,s}^2}{2}\lim_{\varepsilon\searrow0}\int_{\R^N\setminus\left(B_\varepsilon(x)\cup B_\varepsilon(y)\right)}\frac{(z-x)\cdot A(z)(z-y)}{|z-x|^{N+s+1}|z-y|^{N+s+1}}\, dz.
\end{equation}
\end{theorem}

In particular, the kernel representing the bilinear form is uniquely determined outside the diagonal by the matrix $A(x)$, and is given by \eqref{CHARKE}.  

A formally similar integral representation related to \eqref{CHARKE} can be found in \cite[Theorem~2.6]{LoRodBOH} for bounded and measurable matrices $A(x)$. However, these assumptions do not guarantee the existence of the principal value, and the convergence of the truncated bilinear forms does not imply pointwise convergence of the kernel. Our approach addresses this gap by establishing first a weak identification for which no regularity assumption is needed, and then rigorously deriving the pointwise formula under the uniform
fractional Dini regularity condition~\eqref{mkojihguyftdreswxrdcftvgybhjn}. Moreover, unlike the nonuniqueness phenomenon discussed in \cite[Remark~2.8]{LoRodBOH} for test functions supported in a fixed domain, our formulation on the whole space ensures the uniqueness of any symmetric representing kernel (see~\cite[Lemma~2.5]{DSVPART1}.) 

We also point out that 
if~$A$ is a constant matrix, then \eqref{CHARKE} boils down to the formula derived in~\cite[Theorem~1.2]{DSVPART1}.

\begin{corollary}\label{COINCIDONOCOST}
Let~$A\in \R^{N\times N}$ be a constant and symmetric matrix.
Assume that there exists~$k_A\in L^1_{loc}(\R^{2N}\setminus\Delta)$ satisfying \eqref{lòkgyftdresdetfyguhijok}, and such that~\eqref{mnbhgvcfdxszxdfcgvhb} holds.

Then,
\begin{equation}\label{PARTE1COR1.5}
k_A(x,y) = \frac{c_{N,s}}{4s|x-y|^{N+2s}}\left((N+2s)\frac{(x-y)\cdot A(x-y)}{|x-y|^2}-\tr(A)\right).
\end{equation}
\end{corollary}

We stress that
formula \eqref{CHARKE} will play a central role in what follows. First, it will allow us to prove that suitable spatial averages of $k_A(x,x+h)$ satisfy, for large $|h|$, an upper bound of order~$|h|^{-N-2s}$. Then,  a precise analysis of \eqref{CHARKE} will allow us to identify the contribution of order~$|x-y|^{-N-s}$ in the pointwise behavior of $k_A(x,y)$ at infinity and to show that this contribution vanishes if and only if the matrix $A(x)$ is divergence free.

\subsection{Averaged decay at infinity}\label{sectiondecayinfinity}
Theorem  \ref{sxdfcghvbjknlmò,à.,lmknjbhvgcfdxszxfcgvbhn} identifies $2s$ as a ``critical" exponent for pointwise bounds on the kernel when $|x-y|>1$.
However, it does not show that $k_A(x, y)$ itself satisfies \eqref{srdtfyguioplèòlkoijhugyftr2} with $\eta=2s$. 

We now prove that, for the kernel $k_A(x, y)$ explicitly given by \eqref{CHARKE}, the exponent $2s$ appears naturally after spatial averaging. More prcisely, we show that for every fixed $h>0$, the average of $k_A(x,x+h)$ over cubes of increasing size is controlled by $|h|^{-N-2s}$. 

\begin{theorem}\label{ESPONENTS}
Let $A\in L^\infty(\R^N,\R^{N\times N})$ satisfy~\eqref{mkojihguyftdreswxrdcftvgybhjn}
for every compact set~$K\subset\R^N$.
Let~$k_A$ be given by~\eqref{CHARKE}.

Then, there exists $C>0$, depending only on $N$, $s$ and $\|A\|_{L^\infty(\R^N, \R^{N\times N})}$, such that, for every~$x_\star\in\R^N$ and~$h\in\R^N\setminus B_1$,
\[
\limsup_{L\to+\infty}\frac1{(2L)^N}\left|\;\int_{x_\star+[-L,L]^N}k_A(x,x+h)\,dx\right|\le\frac{C}{|h|^{N+2s}}.
\]
\end{theorem}

\subsection{Asymptotics and positivity near the diagonal}
We next use the representation formula~\eqref{CHARKE} to study $k_A(x,y)$
near the diagonal.  
If $A$ is constant, then the positivity of the kernel is determined by the condition $(N+2s)\lambda_{\min}(A)>\tr(A)$, where~$\lambda_{\min}(A)$ is the smallest eigenvalue of the matrix~$A$ and~$\tr(A)$ is its trace,
as showed in~\cite[Theorem~1.2]{DSVPART1}.

For a variable matrix $A(x)$, we cannot expect that the corresponding pointwise condition controls the kernel at every scale. Nevertheless, the same quantity governs the behavior of the kernel near the diagonal. More precisely, we show that, after rescaling around a point $x$, the kernel converges to the constant coefficient kernel associated
with the frozen matrix $A(x)$. Thus, for every $e\in\mathbb S^{N-1}$,
\[
\lim_{r\searrow0} r^{N+2s}k_A(x,x+re) = \frac{c_{N,s}}{4s}\left((N+2s)e\cdot A(x)e-\tr(A(x))\right),
\]
(see the forthcoming Proposition \ref{FREEZING}). We notice that, since
\[
\min_{e\in\mathbb S^{N-1}}\left((N+2s)e\cdot A(x)e-\tr(A(x))\right) = (N+2s)\lambda_{\min}(A(x))-\tr(A(x)),
\]
the quantity $(N+2s)\lambda_{\min}(A(x))-\tr(A(x))$ determines whether the above limit is positive uniformly with respect to the direction $e$. 

In the next result we show that, on compact sets, a uniform positive lower bound for the above quantity is equivalent to a positive estimate from above and below for $k_A(x,y)$ when $x$ and $y$ are sufficiently close. More precisely, we have the following: 

\begin{theorem}\label{LOCALETRACE}
Let $A\in L^\infty(\R^N,\R^{N\times N})$ be symmetric and satisfy~\eqref{mkojihguyftdreswxrdcftvgybhjn} for every compact set~$K\subset\R^N$. Let $k_A$ be given by~\eqref{CHARKE}.

Then, the following statements are equivalent:
\begin{enumerate}
\item[(i)] We have that
\[
\inf_{x\in K}\left((N+2s)\lambda_{\min}(A(x))-\tr(A(x))\right)>0.
\]
\item[(ii)] For every compact set $K\subset\R^N$, there exist $r_K>0$ and $0<\lambda_K\le\Lambda_K<+\infty$ such that
\[
\frac{\lambda_K}{|x-y|^{N+2s}}\le k_A(x,y)\le\frac{\Lambda_K}{|x-y|^{N+2s}}
\]
for every $x\in K$ and every $y\in\R^N$ satisfying $0<|x-y|\le r_K$.
\end{enumerate}
\end{theorem}

We stress that the restriction to small distances is essential. Indeed, when $A$ is a constant matrix,  the associated kernel is homogeneous, so that the same condition controls both small and large scales. For variable coefficients matrices this homogeneity is lost, and the behavior at infinity has to be studied separately (see Sections \ref{sectiondecayinfinity} and \ref{sectiondivergencefree}).

\subsection{Failure of the truncation inequality}

In this section, we investigate the truncation property associated with the bilinear form
\begin{equation}\label{MATHCALBDEFN}
\mathcal B_A(u,v):=\int_{\R^N} A(x)D^su(x)\cdot D^sv(x)\, dx,
\end{equation} in connection
with~\cite[Open Problem~1.9]{MR3819528}.  

For symmetric, bounded, and uniformly elliptic matrices $A(x)$, the properties required for~$\mathcal B_A(u, v)$ to define a Dirichlet form are already known, except for the behavior under truncation.

More precisely, the question posed in \cite[Open Problem 1.9]{MR3819528} is whether, defining
\begin{equation}\label{TAGLIODEFN}
T(t):=\min\{\max\{t,0\},1\},
\end{equation}
the inequality
\[
\mathcal B_A(T(u),T(u))\le \mathcal B_A(u,u)
\]
holds for every admissible function $u$. 

A positive answer would imply that $\mathcal B_A(u, v)$ is a Dirichlet form. The results below show that this
inequality fails in general.  Indeed Proposition \ref{FREEZING}, together with a failure of the trace condition in Theorem \ref{LOCALETRACE} at a single point, is enough to produce a function whose energy increases under the truncation \eqref{TAGLIODEFN}. Thus, the local sign of the kernel has a direct variational consequence for the associated bilinear fom.

More precisely, we prove the following:

\begin{theorem}\label{thmaltroopenproblem}
Let~$A:\R^N\to \R^{N\times N}$ be a symmetric matrix and satisfy~\eqref{uniformlyellipticdefn} and~\eqref{mkojihguyftdreswxrdcftvgybhjn} for every compact set~$K\subset\R^N$.
Let $k_A$ be as in \eqref{CHARKE}.

In addition, assume that there exists $x_0\in\R^N$ such that
\begin{equation}\label{x0conditionbhgvfctdxrsdtfygh}
(N+2s)\lambda_{\min}(A(x_0)) < \tr(A(x_0)).
\end{equation}

Then, there exists $u\in C_c^\infty(\R^N)$ such that
\[
\mathcal B_A(T(u),T(u)) > \mathcal B_A(u,u).
\]
\end{theorem}

As observed in \cite{MR3819528}, a negative answer to the truncation problem is particularly interesting in view of the regularity theory available for equations with matrices $A\in VMO$. Our result shows that the associated bilinear form may fail to be a Dirichlet form even for constant, uniformly elliptic matrices, which in particular satisfy the uniform fractional Dini condition introduced above and belong to $VMO(\R^N)$.

\subsection{Improved decay and divergence free matrices}\label{sectiondivergencefree}
We now investigate what structural information on $A(x)$ can be deduced from the decay of the kernel at infinity. More precisely, we show that the decay condition
\begin{equation}\label{ mnbvcxdszaxdcfvbnm,.poijuhgf}
\lim_{|x|\to+\infty}|x|^{N+s}k_A(x,y)=0\quad\mbox{ for every }y\in\R^N
\end{equation}
is equivalent to the divergence free structure of the matrix $A(x)$.

\begin{theorem}\label{CARASS}
Let $s\in(0,1)$ and let $A\in L^\infty(\R^N,\R^{N\times N})$ be symmetric and satisfy~\eqref{mkojihguyftdreswxrdcftvgybhjn}
for every compact set~$K\subset\R^N$. 
Let $k_A$ be the kernel defined in \eqref{CHARKE}.

Then, the following statements are equivalent:
\begin{enumerate}
\item[(i)] The kernel $k_A$ satisfies \eqref{ mnbvcxdszaxdcfvbnm,.poijuhgf}.
\item[(ii)] The columns of $A$ are divergence free in the sense of distributions, namely,
\[
\sum_{i=1}^N\partial_i A_{ij}=0\quad\mbox{ in }\mathcal D'(\R^N), \quad\mbox{ for } j=1,\ldots,N.
\]
\end{enumerate}

Moreover, if either~$(i)$ or~$(ii)$ holds true,
there exists~$C>0$, depending only on~$N$, $s$, and~$\|A\|_{L^\infty(\R^N,\R^{N\times N})}$, such that
\[
|k_A(x,y)|\le\frac{C}{|x-y|^{N+2s}}\quad\mbox{ for every }x,y\in\R^N, x\ne y.
\]

Finally, if $N=1$, then either~$(i)$ or~$(ii)$ implies that $A$ is constant.
\end{theorem}

We point out that the two equivalent statements in Theorem~\ref{CARASS} are of a different nature. 
Somehow the divergence free condition kills the contribution of order $|x-y|^{-N-s}$ and yields the stronger global estimate
\[
|k_A(x,y)|\le\frac{C}{|x-y|^{N+2s}}.
\]
Thus, for kernels of the form \eqref{CHARKE}, the divergence free property of $A(x)$ characterizes the vanishing of the term of order $|x-y|^{-N-s}$.

We emphasize that no nonnegativity assumption on $k_A(x,y)$ is required for this characterization. This is in contrast with Theorem~\ref{NONNETH} below, where the global nonnegativity of the kernel is used to obtain a lower bound for the quadratic form. Here, instead, the characterization relies on the specific representation \eqref{CHARKE}.

\subsection{A negative answer and instability under perturbations}
We finally investigate whether an affirmative answer to Open Problem~\ref{opeSPJMDPKM} is stable under perturbations of the matrix~$A(x)$.
One might expect that, if the question has an affirmative answer for a given matrix~$A(x)$, then the same remains true for matrices sufficiently close to $A(x)$. The result below shows that this is not the case, even under perturbations converging uniformly, together with all their derivatives.

We point out that ths instability is closely related to the previous analysis of the kernel at infinity.  In dimension $N=1$, the divergence free condition is
equivalent to $A$ being a constant matrix. For a constant matrix, the kernel has decay of order $|x-y|^{-1-2s}$. In the example below, an arbitrarily small nonconstant perturbation produces instead a contribution of order~$|x-y|^{-1-s}$. More precisely, along the pairs $(x,0)$ with $x\to+\infty$, this contribution is negative, and hence the corresponding kernel cannot be nonnegative.

Thus, although this result concerns a different question from \cite[{Open Problem 1.10}]{MR3819528}, it shows that arbitrarily small smooth perturbations of a constant coefficient matrix
can change the behavior of the leading term of the kernel at infinity and ``destroy" its nonnegativity. More precisely, we have the following:
\begin{theorem}\label{SORP}
Let $N=1$. Then, there exist a smooth, bounded and uniformly elliptic matrix~$A(x)$ for which Open Problem~\ref{opeSPJMDPKM} has
an affirmative answer, and a sequence of smooth, bounded and uniformly elliptic matrices $A_n$ for which Open Problem~\ref{opeSPJMDPKM} has a negative answer, such that
\[
\lim_{n\to+\infty}d(A_n,A)=0.
\]
\end{theorem}

\subsection{Organization of the paper}
The paper is organized as follows. In Section \ref{0-pkr32pvb8mn7Jhb398fklsOdDTHSE} we provide the proof of Theorems~\ref{0-pkr32pvb8mn7Jhb398fklsOdDTH} and \ref{sxdfcghvbjknlmò,à.,lmknjbhvgcfdxszxfcgvbhn}. In Section \ref{kojhygftrdeswaq	QAWSEDRFTGYUHJ} we introduce some preliminary results which lead to the proof of Theorem \ref{thmreg2} and we provide the example of an explicit class of admissible matrices.
The proofs of Theorems~\ref{ESPONENTS}, \ref{LOCALETRACE}, \ref{thmaltroopenproblem}, \ref{CARASS} and \ref{SORP} are presented in Sections~\ref{è+òplkojihugy87ft6r58ed4ws34aszxrdcft}, \ref{èòkmnjibhuyctxdsza<azsxdcfv}, \ref{poi98u76t5r4e3jhgfrde}, \ref{èòplokijhugyftrdeswaqASDFG} and \ref{l,kmojnbhugvycfdxsza<zsexdcv}, respectively.
Appendix~\ref{mkjnbhgvfcdxrsezaesxrdctfvyguhoplòè'0o98765}
contains the proof of Corollary~\ref{COINCIDONOCOST}.

\section{Proof of Theorems~\ref{0-pkr32pvb8mn7Jhb398fklsOdDTH} and \ref{sxdfcghvbjknlmò,à.,lmknjbhvgcfdxszxfcgvbhn}}\label{0-pkr32pvb8mn7Jhb398fklsOdDTHSE}
In this section we combine arguments from potential theory and Fourier analysis to detect the
singularity exponent of the interaction kernel and establish Theorems~\ref{0-pkr32pvb8mn7Jhb398fklsOdDTH} and \ref{sxdfcghvbjknlmò,à.,lmknjbhvgcfdxszxfcgvbhn}. We further show that a lower bound of the form $c|x-y|^{-N-\gamma}$, with $c>0$, for $|x-y|>1$,  implies~$\gamma\ge 2s$ (see Theorem \ref{NONNETH}).

We first present a useful auxiliary estimate on a suitable reminder:

\begin{lemma}\label{PHILA}
Let~$\omega\subset\R^N$ be a unit vector and~$\Gamma\in(1,+\infty)$. Given~$\phi\in C^\infty_c(\R^N)$, let
$$\phi_\Gamma(x) := \phi(x) \cos(\Gamma \omega \cdot x)$$
and
$$ R_\Gamma(x):=D^s \phi_\Gamma(x)+\Gamma^s \omega\phi_\Gamma(x)
=D^s \phi_\Gamma(x)+\Gamma^s \omega\phi(x) \cos(\Gamma \omega \cdot x)
.$$

Then,
$$\|R_\Gamma\|_{L^2(\R^N)}\le \frac{C_{N,\phi}}{\Gamma^{1-s}},$$
for some~$C_{N,\phi}>0$.
\end{lemma}

\begin{proof} We observe that
$$2\phi_\Gamma(x)=\phi_{\Gamma,+}(x)+\phi_{\Gamma,-}(x),$$
where~$\phi_{\Gamma,\pm}(x):=\phi(x)\, e^{\pm i\Gamma \omega \cdot x}$.

Therefore, since\footnote{Throughout the paper, given~$u\in\mathcal{ S}(\R^N )$, we denote its Fourier transform by either~$\widehat{u}$ or~$\mathcal{F}(u)$ and adopt the following normalization: $$
\widehat{u}(\xi) := (2\pi)^{-\frac{N}2}\int_{\R^N}
u(x) e^{-ix\cdot\xi}\, dx. $$
With this convention, Plancherel's identity reads as
$$\int_{\R^N}u(x)\overline{v(x)}\,dx=\int_{\R^N}\widehat{u}(\xi)
\overline{\widehat{v}(\xi)}\,d\xi.$$
Moreover,
$$\widehat{D^su}(\xi)=i\xi|\xi|^{s-1}\widehat{u}(\xi).$$}
$$\widehat{\phi_{\Gamma,\pm}}(\xi)=
(2\pi)^{-\frac{N}{2}}\int_{\R^N}
\phi(x) \,e^{-ix\cdot(\xi\mp\Gamma \omega)} dx=
\widehat{\phi_\Gamma}(\xi\mp\Gamma \omega),$$
we conclude that
\begin{equation}\label{KLASSolkr}
2\widehat{\phi_\Gamma}(x)=\widehat{\phi_{\Gamma,+}}(\xi)+\widehat{\phi_{\Gamma,-}}(\xi)=
\widehat{\phi_\Gamma}(\xi-\Gamma \omega)+
\widehat{\phi_\Gamma}(\xi+\Gamma \omega).
\end{equation}

On a similar note, setting
$$\psi_\Gamma:=\phi(x) \sin(\Gamma \omega \cdot x),$$
we have that~$2i\psi_\Gamma(x)=\phi_{\Gamma,+}-\phi_{\Gamma,-}$
and consequently
$$2i\widehat{\psi_\Gamma}(\xi)=\widehat{\phi_{\Gamma,+}}(\xi)-\widehat{\phi_{\Gamma,-}}(\xi)
=\widehat{\phi_\Gamma}(\xi-\Gamma \omega)-\widehat{\phi_\Gamma}(\xi+\Gamma \omega).$$

This, in tandem with~\eqref{KLASSolkr}, returns that
\begin{eqnarray*}
\frac{2\widehat{R_\Gamma}(\xi)}i&=& 
\frac{2}{i}\,{\mathcal{F}}\big( D^s \phi_\Gamma 
+\Gamma^s \omega \psi_\Gamma\big)(\xi)\\&=&
2\xi|\xi|^{s-1} \widehat{\phi_\Gamma}(\xi) 
-2i\Gamma^s \omega \widehat{\psi_\Gamma}(\xi)\\&=&
\xi|\xi|^{s-1}\widehat{\phi_\Gamma}(\xi+\Gamma \omega)+
\xi|\xi|^{s-1}\widehat{\phi_\Gamma}(\xi-\Gamma \omega)
+\Gamma^s \omega 
\widehat{\phi_\Gamma}(\xi+\Gamma \omega)
-\Gamma^s \omega \widehat{\phi_\Gamma}(\xi-\Gamma \omega)
\\&=&\Big(\xi|\xi|^{s-1}+\Gamma^s \omega\Big)\widehat{\phi_\Gamma}(\xi+\Gamma \omega)+
\Big(\xi|\xi|^{s-1}-\Gamma^s \omega\Big)\widehat{\phi_\Gamma}(\xi-\Gamma \omega).
\end{eqnarray*}
As a result,
\begin{equation}\label{45POJLS}\begin{split}
2\|\widehat{R_\Gamma}\|_{L^2(\R^N)}\le&
\sqrt{\;\int_{\R^N}
\Big|\xi|\xi|^{s-1}+\Gamma^s \omega\Big|^2|\widehat{\phi_\Gamma}(\xi+\Gamma \omega)|^2\,d\xi}
\\&+\sqrt{\;\int_{\R^N}
\Big|\xi|\xi|^{s-1}-\Gamma^s \omega\Big|^2|\widehat{\phi_\Gamma}(\xi-\Gamma \omega)|^2\,d\xi}.
\end{split}\end{equation}

Now we point out that
\begin{equation}\label{0ujo32r04386jf3nw23}\begin{split}&
\int_{\R^N}
\Big|\xi|\xi|^{s-1}\pm\Gamma^s \omega\Big|^2|\widehat{\phi_\Gamma}(\xi\pm\Gamma \omega)|^2\,d\xi=
\int_{\R^N}
\Big|(\eta\mp\Gamma \omega)|\eta\mp\Gamma \omega|^{s-1}\pm\Gamma^s \omega\Big|^2|\widehat{\phi_\Gamma}(\eta)|^2\,d\eta\\&\qquad\le
C\int_{\R^N}
\Big|(\eta\mp\Gamma \omega)|\eta\mp\Gamma \omega|^{s-1}\pm\Gamma^s \omega\Big|^2
\frac{d\eta}{1+|\eta|^{N+4}},\end{split}\end{equation}
for some~$C>0$, that we will take the liberty of renaming from line to line.

It is also helpful to remark that, when~$\eta\in B_{\Gamma/2}$ we have that, for all~$t\in[0,1]$,
$$ |t\eta+\Gamma \omega|\in\left[\Gamma-|\eta|,\,|\eta|+\Gamma\right]\subset
\left[\frac\Gamma2,\,\frac{3\Gamma}2\right]$$
and therefore
\begin{eqnarray*}
&&\Big|(\eta+\Gamma \omega)|\eta+\Gamma \omega|^{s-1}-\Gamma^s \omega\Big|=
\left| \int_0^1 \frac{d}{dt}\Big[ (t\eta+\Gamma \omega)|t\eta+\Gamma \omega|^{s-1}\Big]\,dt\right|\\
&&\qquad=\left| \int_0^1 \eta|t\eta+\Gamma \omega|^{s-1}
+(s-1)(t\eta+\Gamma \omega)|t\eta+\Gamma \omega|^{s-3}(t\eta+\Gamma \omega)\cdot\eta
\,dt\right|\le C\Gamma^{s-1}|\eta|.
\end{eqnarray*}

On this account,
\begin{equation*}
\int_{B_{\Gamma/2}}
\Big|(\eta+\Gamma \omega)|\eta+\Gamma \omega|^{s-1}-\Gamma^s \omega\Big|^2
\frac{d\eta}{1+|\eta|^{N+4}}\le
C\Gamma^{2s-2}
\int_{B_{\Gamma/2}}
\frac{|\eta|^2\,d\eta}{1+|\eta|^{N+4}}\le C\Gamma^{2s-2}.
\end{equation*}
Analogously, by swapping~$\omega$ with~$-\omega$,
\begin{equation*}
\int_{B_{\Gamma/2}}
\Big|(\eta-\Gamma \omega)|\eta-\Gamma \omega|^{s-1}+\Gamma^s \omega\Big|^2
\frac{d\eta}{1+|\eta|^{N+4}}\le C\Gamma^{2s-2}.
\end{equation*}
We can thereby summarize the above information by writing
\begin{equation}\label{0ujo32r04386jf3nw2}
\int_{B_{\Gamma/2}}
\Big|(\eta\mp\Gamma \omega)|\eta\mp\Gamma \omega|^{s-1}\pm\Gamma^s \omega\Big|^2
\frac{d\eta}{1+|\eta|^{N+4}}\le C\Gamma^{2s-2}.
\end{equation}

Furthermore,
\begin{eqnarray*}
\Big|(\eta\mp\Gamma \omega)|\eta\mp\Gamma \omega|^{s-1}\pm\Gamma^s \omega\Big|\le
|\eta\mp\Gamma \omega|^s+\Gamma^s\le
(|\eta|+\Gamma )^s+\Gamma^s\le2
(|\eta|+\Gamma )^s,\end{eqnarray*}leading to
\begin{eqnarray*}
&&\int_{\R^N\setminus B_{\Gamma/2}}
\Big|(\eta\mp\Gamma \omega)|\eta\mp\Gamma \omega|^{s-1}\pm\Gamma^s \omega\Big|^2
\frac{d\eta}{1+|\eta|^{N+4}}\le 4\int_{\R^N\setminus B_{\Gamma/2}}
(|\eta|+\Gamma )^{2s}
\frac{d\eta}{1+|\eta|^{N+4}}\\&&\qquad\le 4\int_{\R^N\setminus B_{\Gamma/2}}
(3|\eta| )^{2s}
\frac{d\eta}{|\eta|^{N+4}}\le C\Gamma^{2s-4}.
\end{eqnarray*}

We now plug this and~\eqref{0ujo32r04386jf3nw2} into~\eqref{0ujo32r04386jf3nw23} and we gather that
\begin{equation*}
\int_{\R^N}
\Big|\xi|\xi|^{s-1}\pm\Gamma^s \omega\Big|^2|\widehat{\phi_\Gamma}(\xi\pm\Gamma \omega)|^2\,d\xi\le C\Gamma^{2s-2}.
\end{equation*}
This and~\eqref{45POJLS} yield that~$\|\widehat{R_\Gamma}\|_{L^2(\R^N)}\le C\Gamma^{s-1}$.
The desired result now follows from Plancherel's identity.
\end{proof}

With the preliminary work presented above, we can now proceed with the proof of Theorem~\ref{0-pkr32pvb8mn7Jhb398fklsOdDTH}.

\begin{proof}[Proof of Theorem~\ref{0-pkr32pvb8mn7Jhb398fklsOdDTH}] 
Let $\omega\in\mathbb S^{N-1}$ be a unit vector. Let $\phi\in C^\infty_c(\R^N)$ with~$\phi\not\equiv0$ and $\Gamma>1$. We set
\begin{equation}\label{klmjnhbgvfctdxrszeaw<ESRDTFYGHUJ}
\phi_\Gamma(x):=\phi(x)\cos(\Gamma\omega\cdot x).
\end{equation}

We first prove that there exists $\Gamma_1>1$ such that, for every~$\Gamma\ge\Gamma_1$,
\begin{equation}\label{,mnbhjgvcfdxszeasexdfcgvhb}
\frac{\lambda_A}{16}\|\phi\|_{L^2(\R^N)}^2\Gamma^{2s}\le\int_{\R^N} A(x)D^s\phi_\Gamma(x)\cdot D^s\phi_\Gamma(x)\,dx\le \Lambda_A
\left(\frac{\sqrt3}{2}+\frac14\right)^2 \|\phi\|_{L^2(\R^N)}^2 \Gamma^{2s}.
\end{equation}
Indeed, recalling Lemma \ref{PHILA}, we write
\begin{equation}\label{,lkmjnhbgvfctdxrszea<sxdrcfvgbhnjkml,ò.}
D^s\phi_\Gamma(x) = -\Gamma^s\omega\phi(x)\sin(\Gamma\omega\cdot x) + R_\Gamma(x),
\end{equation}
where
\begin{equation}\label{òèplkoijhugyftdres5dtygh0odei9wfu8gy}
\|R_\Gamma\|_{L^2(\R^N)}\le\frac{C_{N,\phi}}{\Gamma^{1-s}}.
\end{equation}
Since $\phi^2\in L^1(\R^N)$, by the Riemann--Lebesgue theorem there exists $\Gamma_0>1$ such that, for every~$\Gamma\ge\Gamma_0$,
\begin{equation}\label{RLUytrewertyuikm xzsdfghj}
\left|\;\int_{\R^N}\phi^2(x)e^{2i\Gamma\omega\cdot x} dx\right|\le\frac12\|\phi\|_{L^2(\R^N)}^2.
\end{equation}
Thus, since
\[
\begin{split}
\int_{\R^N}\phi^2(x)\sin^2(\Gamma\omega\cdot x)\,dx&=\frac12\int_{\R^N}\phi^2(x) \,dx-\frac12\int_{\R^N}\phi^2(x)\cos(2\Gamma\omega\cdot x)\,dx\\
&=\frac12 \|\phi\|^2_{L^2(\R^N)}-\frac12\Re\left(\;\int_{\R^N}\phi^2(x)e^{2i\Gamma\omega\cdot x}\, dx\right)
\end{split}
\]
we infer that, for every $\Gamma\ge\Gamma_0$,
\begin{equation}\label{mkjnihgufderxswzsxdrcfvgbhnj}
\frac14\|\phi\|_{L^2(\R^N)}^2 \le\int_{\R^N}\phi^2(x)\sin^2(\Gamma\omega\cdot x)\,dx \le \frac34\|\phi\|_{L^2(\R^N)}^2.
\end{equation}
Accordingly, for every $\Gamma\ge\Gamma_0$,
\[
\frac{\|\phi\|_{L^2(\R^N)}}{2}\Gamma^s\le
\left(\;\int_{\R^N }\Gamma^{2s}|\omega|^2\phi^2(x)\sin^2(\Gamma\omega\cdot x)\,dx\right)^{\frac12}
\le\frac{\sqrt{3}\,\|\phi\|_{L^2(\R^N)}}{2}\Gamma^s.
\]

We now set
\[
\Gamma_1:=\max\left\{\Gamma_0,\frac{4C_{N,\phi}}{\|\phi\|_{L^2(\R^N)}}\right\}.
\]
Then, by \eqref{òèplkoijhugyftdres5dtygh0odei9wfu8gy} we infer that, for every $\Gamma\ge\Gamma_1$,
\[
\|R_\Gamma\|_{L^2(\R^N)}\le \frac{C_{N,\phi}\Gamma^{s}}{\Gamma_1}\le\frac14\|\phi\|_{L^2(\R^N)}\Gamma^s.
\]
Accordingly, recalling \eqref{,lkmjnhbgvfctdxrszea<sxdrcfvgbhnjkml,ò.}, we deduce that, for every~$\Gamma\ge\Gamma_1$,
\[
\|D^s\phi_\Gamma\|_{L^2(\R^N)}\le\left(\;\int_{\R^N }\Gamma^{2s}|\omega|^2\phi^2(x)\sin^2(\Gamma\omega\cdot x)\,dx\right)^{\frac12} + \|R_\Gamma\|_{L^2(\R^N)}\le
\left(\frac{\sqrt3}{2}+\frac14\right)\|\phi\|_{L^2(\R^N)}\Gamma^s\]
and
\[ \|D^s\phi_\Gamma\|_{L^2(\R^N)} \ge \left(\;\int_{\R^N }\Gamma^{2s}|\omega|^2\phi^2(x)\sin^2(\Gamma\omega\cdot x)\,dx\right)^{\frac12} - \|R_\Gamma\|_{L^2(\R^N)}\ge\frac14\|\phi\|_{L^2(\R^N)}\Gamma^s.
\]
By using these inequalities together with \eqref{uniformlyellipticdefn}, we obtain~\eqref{,mnbhjgvcfdxszeasexdfcgvhb} for every~$\Gamma\ge\Gamma_1$.

We now write
\begin{equation}\label{,lòkmjngvycftdxrszeadxfcvghbjnkm,}
\iint_{\R^{2N}}(\phi_\Gamma(x)-\phi_\Gamma(y))^2k(x,y) \,dx\, dy =
\mathcal I_{1, \Gamma} + \mathcal I_{2, \Gamma},
\end{equation} where
\begin{eqnarray*} \mathcal I_{1, \Gamma}&:=&
 \iint_{\{|x-y|<1\}}(\phi_\Gamma(x)-\phi_\Gamma(y))^2k(x,y) \,dx\, dy\\
{\mbox{and }}\qquad  \mathcal I_{2, \Gamma}&:=&
 \iint_{\{|x-y|\ge1\}}(\phi_\Gamma(x)-\phi_\Gamma(y))^2k(x,y) \,dx \,dy
\end{eqnarray*}
By \eqref{srdtfyguioplèòlkoijhugyftr1}, and setting $h:=x-y$, we have
\begin{equation}\label{nearineq}
\lambda\iint_{\R^N\times B_1}\frac{|\phi_\Gamma(x)-\phi_\Gamma(x+h)|^2}{|h|^{N+\alpha}}\, dh\, dx \le \mathcal I_{1, \Gamma} \le\Lambda\iint_{\R^N\times B_1}\frac{|\phi_\Gamma(x)-\phi_\Gamma(x+h)|^2}{|h|^{N+\alpha}} \,dh\, dx.
\end{equation}

We recall \eqref{klmjnhbgvfctdxrszeaw<ESRDTFYGHUJ} and write
\[
\phi_\Gamma(x)-\phi_\Gamma(x+h) = I_\Gamma(x,h)+J_\Gamma(x,h),
\]
where
\[
I_\Gamma(x,h):=\big(\phi(x)-\phi(x+h)\big)\cos(\Gamma\omega\cdot(x+h))
\]
and
\begin{eqnarray*}
J_\Gamma(x,h)&:=&\phi(x)\Big(\cos(\Gamma\omega\cdot x) - \cos(\Gamma\omega\cdot(x+h))\Big) \\&=& 2\phi(x)\sin\left(\frac{\Gamma\omega\cdot h}{2}\right)
\sin\left(\Gamma\omega\cdot x+ \frac{\Gamma\omega\cdot h}{2}\right).
\end{eqnarray*}

By the mean value theorem, we have
\[
|\phi(x)-\phi(x+h)|
\le
\|\nabla\phi\|_{L^\infty(\R^N)}|h|.
\]
Moreover, if $h\in B_1$, then $\phi(x)-\phi(x+h)$ can be nonzero only if $x\in\mbox{supp}(\phi)+B_1$. Therefore, 
\begin{equation}\label{jnhbgvfcdxrszeaxdcfvgb}
\iint_{\R^N\times B_1}\frac{|I_\Gamma(x,h)|^2}{|h|^{N+\alpha}} \,dh\, dx\le\|\nabla\phi\|_{L^\infty(\R^N)}^2|\mbox{supp}(\phi)+B_1|\int_{B_1}|h|^{2-N-\alpha}\, dh=:C_I.
\end{equation}

Furthermore, for every $h\in B_1$,
\[
\int_{\R^N}\phi^2(x)\sin^2\left(\Gamma\omega\cdot x+\frac{\Gamma\omega\cdot h}{2}\right) dx = \frac12\|\phi\|_{L^2(\R^N)}^2 - \frac12\Re\left(e^{i\Gamma\omega\cdot h}\int_{\R^N}\phi^2(x)e^{2i\Gamma\omega\cdot x} dx\right).
\]
Hence, thanks to~\eqref{RLUytrewertyuikm xzsdfghj},
for every~$\Gamma\ge\Gamma_1$ we see that
\begin{equation*}
\frac14\|\phi\|_{L^2(\R^N)}^2\le\int_{\R^N}\phi^2(x)\sin^2\left(\Gamma\omega\cdot x+\frac{\Gamma\omega\cdot h}{2}\right) dx\le\frac34\|\phi\|_{L^2(\R^N)}^2.
\end{equation*}
Accordingly, setting \[
\varpi_\Gamma:=\int_{B_1}\sin^2\left(\frac{\Gamma\omega\cdot h}{2}\right)\frac{dh}{|h|^{N+\alpha}},
\]
we find that
\begin{equation}\label{hgfdreswaqAWESRDTFYGUHIJOKPLò}
\|\phi\|_{L^2(\R^N)}^2\varpi_\Gamma\le\iint_{\R^N\times B_1}\frac{|J_\Gamma(x,h)|^2}{|h|^{N+\alpha}} dh dx\le 3\|\phi\|_{L^2(\R^N)}^2\varpi_\Gamma.
\end{equation}

Now, by the change of variable $z:=\Gamma h$,
\[
\varpi_\Gamma = \Gamma^\alpha\int_{B_\Gamma}\sin^2\left(\frac{\omega\cdot z}{2}\right)\frac{dz}{|z|^{N+\alpha}}.
\]
Thus, by defining
\[
a_{\alpha,\omega}:=\int_{B_1}\sin^2\left(\frac{\omega\cdot z}{2}\right)\frac{dz}{|z|^{N+\alpha}}\qquad\mbox{ and }\qquad b_{\alpha,\omega}:=\int_{\R^N}
\sin^2\left(\frac{\omega\cdot z}{2}\right)\frac{dz}{|z|^{N+\alpha}},
\]
since $\alpha\in (0, 2)$, we get 
$$0<a_{\alpha,\omega}\Gamma^\alpha\le \varpi_\Gamma\le
 b_{\alpha,\omega}\Gamma^\alpha<+\infty.$$
 Combining this information with \eqref{hgfdreswaqAWESRDTFYGUHIJOKPLò}, we deduce that, for every $\Gamma\ge\Gamma_1$,
\begin{equation*}
\|\phi\|_{L^2(\R^N)}^2a_{\alpha,\omega}\,\Gamma^\alpha\le\iint_{\R^N\times B_1}\frac{|J_\Gamma(x,h)|^2}{|h|^{N+\alpha}} \,dh \,dx\le 3\|\phi\|_{L^2(\R^N)}^2b_{\alpha,\omega}\Gamma^\alpha.
\end{equation*}

From this and~\eqref{jnhbgvfcdxrszeaxdcfvgb}, we obtain that
\begin{equation}\label{uppernjkbhgvcfdxsrzr}
\begin{split}
&\iint_{\R^N\times B_1}\frac{|\phi_\Gamma(x)-\phi_\Gamma(x+h)|^2}{|h|^{N+\alpha}} \,dh \,dx=\iint_{\R^N\times B_1}\frac{|I_\Gamma(x,h)+J_\Gamma(x,h)|^2}{|h|^{N+\alpha}}\, dh\, dx\\
&\qquad\le 2\iint_{\R^N\times B_1}\frac{|I_\Gamma(x,h)|^2}{|h|^{N+\alpha}} \,dh \,dx + 2\iint_{\R^N\times B_1}\frac{|J_\Gamma(x,h)|^2}{|h|^{N+\alpha}}\, dh\, dx\\
&\qquad\le 2C_I+6\|\phi\|_{L^2(\R^N)}^2b_{\alpha,\omega}\Gamma^\alpha\\
&\qquad\le 2\left(C_I+3\|\phi\|_{L^2(\R^N)}^2b_{\alpha,\omega}\right)\Gamma^\alpha
\end{split}
\end{equation} and that
\begin{equation}\label{jnbugvyftcdrxsewxrdcftvgybhjnkm}
\begin{split}
\iint_{\R^N\times B_1}\frac{|\phi_\Gamma(x)-\phi_\Gamma(x+h)|^2}{|h|^{N+\alpha}} \,dh\, dx&\ge\frac12\int_{\R^N}\int_{B_1}\frac{|J_\Gamma(x,h)|^2}{|h|^{N+\alpha}} \,dh\, dx - \int_{\R^N}\int_{B_1}\frac{|I_\Gamma(x,h)|^2}{|h|^{N+\alpha}}\, dh\, dx\\
&\ge \frac12\|\phi\|_{L^2(\R^N)}^2a_{\alpha,\omega}\,\Gamma^\alpha-C_I,
\end{split}
\end{equation}
where the inequality~$|a+b|^2\ge\frac12|b|^2-|a|^2$, for every $a, b\in\R$, has been used.

We now set
\[
\Gamma_2:=\max\left\{\Gamma_1,\,\left(\frac{4C_I}{\|\phi\|_{L^2(\R^N)}^2a_{\alpha,\omega}}\right)^{1/\alpha}\right\}.
\] In this way,
\[
C_I\le\frac14 \|\phi\|_{L^2(\R^N)}^2a_{\alpha,\omega} \Gamma^\alpha,
\] and therefore,
combining this with \eqref{jnbugvyftcdrxsewxrdcftvgybhjnkm}, we get
\begin{equation}\label{lowerjnkbhgvcfhvbjknlm,}
\iint_{\R^N\times B_1}\frac{|\phi_\Gamma(x)-\phi_\Gamma(x+h)|^2}{|h|^{N+\alpha}} \,dh\, dx\ge\frac14
\|\phi\|_{L^2(\R^N)}^2a_{\alpha,\omega} \Gamma^\alpha.
\end{equation}

From~\eqref{uppernjkbhgvcfdxsrzr} and~\eqref{lowerjnkbhgvcfhvbjknlm,},
and recalling \eqref{nearineq}, it follows that, for every $\Gamma\ge\Gamma_2$,
\begin{equation}\label{kmjnhbgvfcdtxrszea<ezsxrdctfvgbh}
\frac{1}4 \lambda\|\phi\|_{L^2(\R^N)}^2a_{\alpha,\omega}\Gamma^\alpha\le \mathcal I_{1, \Gamma} \le2\Lambda\left(C_I+3\|\phi\|_{L^2(\R^N)}^2b_{\alpha,\omega}\right)\Gamma^\alpha.
\end{equation}

We now estimate $\mathcal I_{2,\Gamma}$.  By \eqref{kmljnbhgvcfdxcfvgbh}, we get
\[
|\mathcal I_{2,\Gamma}|\le M \iint_{\{|x-y|\ge1\}}\frac{(\phi_\Gamma(x)-\phi_\Gamma(y))^2}{|x-y|^{N+\eta}} \,dx\, dy.
\]
Thus, setting $h:=y-x$,
\begin{equation}\label{hgfctdrsewaesdrtfgybhujnkml,}
\begin{split}&
|\mathcal I_{2,\Gamma}|\le 2M\int_{\{|h|\ge1\}}\frac1{|h|^{N+\eta}}\left(\;\int_{\R^N}\left(\phi_\Gamma^2(x)+\phi_\Gamma^2(x+h)\right)\, dx\right)\, dh \\&\qquad=
4M\|\phi_\Gamma\|_{L^2(\R^N)}^2 \int_{\{|h|\ge1\}}\frac{dh}{|h|^{N+\eta}}
\le 4M\|\phi\|_{L^2(\R^N)}^2\int_{\{|h|\ge1\}}\frac{dh}{|h|^{N+\eta}}=:
\widetilde C\in(0,+\infty).
\end{split}
\end{equation}

From \eqref{,lòkmjngvycftdxrszeadxfcvghbjnkm,},\eqref{kmjnhbgvfcdtxrszea<ezsxrdctfvgbh} and \eqref{hgfctdrsewaesdrtfgybhujnkml,}, we obtain that
\begin{equation}\label{vm8oewuti43i9876543w-y65}\begin{split}&
\frac{1}4
\lambda\|\phi\|_{L^2(\R^N)}^2 a_{\alpha,\omega} \Gamma^\alpha-\widetilde C\le\iint_{\R^{2N}}(\phi_\Gamma(x)-\phi_\Gamma(y))^2k(x,y)\,dx\,dy\\&\qquad\qquad\le
2\Lambda\left(C_I+3\|\phi\|_{L^2(\R^N)}^2b_{\alpha,\omega}\right)\Gamma^\alpha+\widetilde C .
\end{split}\end{equation}

We now set
\[
\Gamma_3:=\max\left\{\Gamma_2,\,\left(\frac{8\widetilde C}{\lambda\|\phi\|_{L^2(\R^N)}^2a_{\alpha,\omega}}\right)^{1/\alpha}\right\}.
\] In this way,
for every $\Gamma\ge\Gamma_3$,
\[
\widetilde C\le\frac{1}8
\lambda\|\phi\|_{L^2(\R^N)}^2a_{\alpha,\omega} \Gamma^\alpha.
\]
Accordingly, we have that
\[
\frac{\lambda\|\phi\|_{L^2(\R^N)}^2a_{\alpha,\omega}}{8}\Gamma^\alpha\le\iint_{\R^{2N}}(\phi_\Gamma(x)-\phi_\Gamma(y))^2k(x,y)\,dx\,dy.
\]

From this and~\eqref{vm8oewuti43i9876543w-y65}, we conclude that,
for every $\Gamma\ge\Gamma_3$,
\begin{equation}\label{kjnhbgfctdxrszea<zesxdrcftvgbhn}
\begin{split}&
\frac{1}8\lambda\|\phi\|_{L^2(\R^N)}^2a_{\alpha,\omega}\Gamma^\alpha\le\iint_{\R^{2N}}(\phi_\Gamma(x)-\phi_\Gamma(y))^2k(x,y) \,dx\, dy\\
&\qquad\qquad\le \left(2\Lambda(C_I+3\|\phi\|_{L^2(\R^N)}^2b_{\alpha,\omega})+\widetilde C\right)\Gamma^\alpha
.\end{split}
\end{equation}

Then, for every $\Gamma\ge\Gamma_3$, identity~\eqref{OJS02irkt-3}, together with~\eqref{,mnbhjgvcfdxszeasexdfcgvhb} and \eqref{kjnhbgfctdxrszea<zesxdrcftvgbhn}, yields
\begin{equation}\label{mlkjnhbgftdrxsezazsxrdcf1}
\frac{\lambda_A}{16}\|\phi\|_{L^2(\R^N)}^2\Gamma^{2s-\alpha}\le 2\Lambda\left(C_I+3\|\phi\|_{L^2(\R^N)}^2b_{\alpha,\omega}\right)+\widetilde C
\end{equation}
and
\begin{equation}\label{mlkjnhbgftdrxsezazsxrdcf2}
\frac18\lambda\|\phi\|_{L^2(\R^N)}^2a_{\alpha,\omega}\Gamma^{\alpha-2s}\le\Lambda_A\left(\frac{\sqrt3}{2}+\frac14\right)^2 \|\phi\|_{L^2(\R^N)}^2.
\end{equation}
Now, if $\alpha<2s$, then
\[
\lim_{\Gamma\to+\infty} \frac{\lambda_A}{16}\|\phi\|_{L^2(\R^N)}^2\Gamma^{2s-\alpha}=+\infty,
\]
which contradicts \eqref{mlkjnhbgftdrxsezazsxrdcf1}. If instead~$\alpha>2s$, then
\[
\lim_{\Gamma\to+\infty} \frac{\lambda\|\phi\|_{L^2(\R^N)}^2a_{\alpha,\omega}}{8}\Gamma^{\alpha-2s}=+\infty,
\]
which contradicts \eqref{mlkjnhbgftdrxsezazsxrdcf2}. Hence, necessarily~$\alpha=2s$, as desired.
\end{proof}

We now provide the proof of Theorem \ref{sxdfcghvbjknlmò,à.,lmknjbhvgcfdxszxfcgvbhn}.

\begin{proof}[Proof of Theorem \ref{sxdfcghvbjknlmò,à.,lmknjbhvgcfdxszxfcgvbhn}]
Let $\phi\in C_c^\infty(\R^N)$ with~$\phi\not\equiv0$ and~$R\ge 1$,
and define
\[
u_R(x):=\phi\left(\frac{x}{R}\right).
\]
Since the fractional gradient is homogeneous of degree $s$, we see that
\[
D^s u_R(x) = R^{-s} (D^s\phi)\left(\frac{x}{R}\right).
\]
Accordingly, by using the change of variables $X:=x/R$, we obtain that
\[
\|D^s u_R\|_{L^2(\R^N)}^2 = R^{-2s} \int_{\R^N}\left|(D^s\phi)\left(\frac{x}{R}\right)\right|^2\,dx = R^{N-2s}\|D^s\phi\|_{L^2(\R^N)}^2.
\]
Since $\phi\not\equiv 0$, we get $\|D^s\phi\|_{L^2(\R^N)}>0$. Hence, by \eqref{uniformlyellipticdefn}, 
\begin{equation}\label{swderftgyhujikopijhugyftdre}
\int_{\R^N} A(x)D^s u_R(x)\cdot D^s u_R(x)\,dx\ge\lambda_A\|D^s u_R\|_{L^2(\R^N)}^2 = \lambda_A R^{N-2s}\|D^s\phi\|_{L^2(\R^N)}^2>0.
\end{equation}

For any~$z\in\R^N$, we define
\[
G(z):=\int_{\R^N}|\phi(x+z)-\phi(x)|^2\,dx
\]
and we claim that there exists a constant $C_\phi>0$, depending only on $\phi$, such that
\begin{equation}\label{hbgftcdrxsazwsexdcfvbhjnm}
G(z)\le C_\phi \min\{|z|^2,1\}\quad\mbox{ for every } z\in\R^N.
\end{equation}
Indeed,
\[
G(z) \leq 2
\int_{\R^N}\left(|\phi(x+z)|^2+|\phi(x)|^2\right)\,dx
= 4\|\phi\|_{L^2(\R^N)}^2.
\]
Moreover, by the fundamental theorem of calculus,
\[
\phi(x+z)-\phi(x) = \int_0^1 z\cdot\nabla\phi(x+tz)\,dt
\] and
accordingly, by Minkowski's integral inequality,
\[
G(z)\le\int_0^1 |z|^2 \|\nabla\phi(\cdot+tz)\|^2_{L^2(\R^N)}\,dt = |z|^2\|\nabla\phi\|_{L^2(\R^N)}^2.
\]
From these observations \eqref{hbgftcdrxsazwsexdcfvbhjnm} follows.

Furthermore, for every~$h\in\R^N$, using the change of variable~$X:=x/R$,
\begin{equation}\label{mkjnhbugvfcderxswzxrdctfvgbhn}
\int_{\R^N}
|u_R(x+h)-u_R(x)|^2\,dx =R^N \int_{\R^N}\left|\phi\left(X+\frac{h}{R}\right)-\phi(X)\right|^2 dX = R^N G\left(\frac{h}{R}\right).
\end{equation}

We now set
\[
E_{R, 1}:=\iint_{\{|x-y|\le 1\}} (u_R(x)-u_R(y))^2 k(x,y)\,dx\,dy.
\]
By \eqref{mlnjkbhgvcfdxszexdcfgvhbjnkml,} and~\eqref{mkjnhbugvfcderxswzxrdctfvgbhn}, we get
\begin{eqnarray*}
E_{R, 1}&\le&\Lambda\iint_{\{|x-y|\le 1\}}\frac{|u_R(x)-u_R(y)|^2}{|x-y|^{N+2s}}\, dx\, dy\\
&\le&\Lambda\int_{\{|h|\le 1\}}\frac{1}{|h|^{N+2s}}\left(\;\int_{\R^N}|u_R(x+h)-u_R(x)|^2\,dx\right)\,dh\\
&=&\Lambda R^N\int_{\{|h|\le 1\}} G\left(\frac{h}{R}\right)\frac{dh}{|h|^{N+2s}}.
\end{eqnarray*}
Now, since $|h|/R\le 1$, from \eqref{hbgftcdrxsazwsexdcfvbhjnm} we infer that
\[
G\left(\frac{h}{R}\right)\le C_\phi \frac{|h|^2}{R^2},
\] therefore
\begin{equation}\label{jhgvcfvguhbijnlkòlàojpihugyguhpij}
E_{R, 1}\le C_\phi\Lambda R^{N-2}\int_{\{|h|\le 1\}}|h|^{2-N-2s}\,dh = C_\phi\Lambda R^{N-2} \omega_{N-1}\int_0^1 r^{1-2s}\,dr
\le C R^{N-2},
\end{equation} for some~$C=C(\phi, \Lambda, N, s)>0$.

We also set
\[
E_{R, 2}:=\iint_{\{|x-y|\ge 1\}} (u_R(x)-u_R(y))^2 k(x,y)\,dx\,dy.
\]
By \eqref{kmljnbhgvcfdxcfvgbhBIS} and~\eqref{mkjnhbugvfcderxswzxrdctfvgbhn},
\begin{eqnarray*}
E_{R, 2}&\le& M \iint_{\{|x-y|\ge 1\}}\frac{|u_R(x)-u_R(y)|^2}{|x-y|^{N+\eta}}\,dx\,dy\\&=&
M \int_{\{|h|>1\}} \frac{1}{|h|^{N+\eta}}\left(\;\int_{\R^N} |u_R(x+h)-u_R(x)|^2\,dx\right)\,dh\\&=& M R^N\int_{\{|h|>1\}} G\left(\frac{h}{R}\right)
\frac{dh}{|h|^{N+\eta}}
.
\end{eqnarray*}
By using the change of variable~$z:=h/R$ and recalling~\eqref{hbgftcdrxsazwsexdcfvbhjnm}, we thus have that
\begin{equation}\label{dfguhijokljihugyfderctfvgbhnjmk}
\begin{split}
E_{R, 2}&\le M R^{N-\eta} \int_{\{|z|>R^{-1}\}} G(z)|z|^{-N-\eta} \,dz\\
&= M R^{N-\eta}\left(\;\int_{\{R^{-1}<|z|\le 1\}} G(z)|z|^{-N-\eta}\, dz + \int_{\{|z|>1\}}G(z)|z|^{-N-\eta}\, dz \right)\\
&\le C_\phi M R^{N-\eta}\left(\;\int_{\{R^{-1}<|z|\le 1\}} |z|^{2-N-\eta} \,dz + \int_{\{|z|>1\}} |z|^{-N-\eta}\, dz \right)
\end{split}
\end{equation}

We observe that
\[
\int_{R^{-1}<|z|\le 1}|z|^{2-N-\eta}\,dz =\omega_{N-1} \int_{1/R}^1 r^{1-\eta}\,dr\le 
\begin{cases}
\dfrac{\omega_{N-1}}{2-\eta} &\quad\mbox{ if } 0<\eta<2,\\
\omega_{N-1}\log R &\quad\mbox{ if } \eta=2\\
\dfrac{\omega_{N-1} R^{\eta-2}}{\eta-2} &\quad\mbox{ if } \eta>2.
\end{cases}
\]
Hence, using this information into~\eqref{dfguhijokljihugyfderctfvgbhnjmk}, we infer that there exists $C=C(N, \eta, M, \phi)>0$ such that
\begin{equation}\label{kmjnbhgvfcdxszasxdcfhnjm}
E_{R, 2}\le C
\begin{cases}
R^{N-\eta} &\quad\mbox{ if } 0<\eta<2,\\[2mm]
R^{N-2}(1+\log R) &\quad\mbox{ if } \eta=2,\\[2mm]
R^{N-2} &\quad\mbox{ if } \eta>2.
\end{cases}
\end{equation}

Now, we assume by contradiction that $\eta>2s$. From \eqref{OJS02irkt-3}, we have that
\[
\int_{\R^N} A(x)D^s u_R(x)\cdot D^s u_R(x)\,dx = E_{R, 1} + E_{R, 2}
.
\]
Combining the latter inequality with \eqref{swderftgyhujikopijhugyftdre}, we deduce that
\[
\lambda_A \|D^s\phi\|_{L^2(\R^N)}^2 R^{N-2s}\le E_{R, 1} + E_{R, 2}.
\]

We now distinguish all the possible values of $\eta$. If $2s<\eta<2$, then, by \eqref{jhgvcfvguhbijnlkòlàojpihugyguhpij} and \eqref{kmjnbhgvfcdxszasxdcfhnjm}, we have
\[
\lambda_A \|D^s\phi\|_{L^2(\R^N)}^2 R^{N-2s}\le C R^{N-2} + C R^{N-\eta}.
\]
Dividing by $R^{N-2s}$, we get
\[
0<\lambda_A \|D^s\phi\|_{L^2(\R^N)}^2\le C R^{2s-2} + C R^{2s-\eta}.
\] By sending~$R\to+\infty$ we obtain a contradiction.

If $\eta=2$, then, by \eqref{jhgvcfvguhbijnlkòlàojpihugyguhpij} and \eqref{kmjnbhgvfcdxszasxdcfhnjm}, and dividing by $R^{N-2s}$, we obtain that
\[ 0<
\lambda_A \|D^s\phi\|_{L^2(\R^N)}^2\le C R^{2s-2}(1+\log R),
\] which gives a contradiction by sending~$R\to+\infty$.

Finally, if $\eta>2$,  then, by \eqref{jhgvcfvguhbijnlkòlàojpihugyguhpij} and~\eqref{kmjnbhgvfcdxszasxdcfhnjm}, and dividing by $R^{N-2s}$, we get
\[
0<\lambda_A \|D^s\phi\|_{L^2(\R^N)}^2\le C R^{2s-2},
\] from which we obtain a contradiction by sending~$R\to+\infty$.

From these considerations, we deduce that necessarily~$\eta\le 2s$, as desired. \end{proof}

We now show the following counterpart of Theorem~\ref{sxdfcghvbjknlmò,à.,lmknjbhvgcfdxszxfcgvbhn}.

\begin{theorem}\label{NONNETH}
Let $s\in(0,1)$ and $A:\R^N\to\R^{N\times N}$ be bounded and measurable. Let  $k_A:\R^{2N}\setminus\Delta\to[0,+\infty)$ be measurable and satisfy~\eqref{OJS02irkt-3} and
\begin{equation}\label{DASOTTO}
k_A(x,y)\ge\frac{c}{|x-y|^{N+\gamma}}\quad\mbox{ for } |x-y|>1, 
\end{equation}
for some~$\gamma\in\R$ and $c>0$.

Then,  $\gamma\ge 2s$.
\end{theorem}

\begin{proof}
Let $x_0,y_0\in\R^N$ be such that $\frac{|x_0-y_0|}{4}=:r>1$.  We first observe that, if $(x,y)\in B_{r/2}(x_0)\times B_{r/2}(y_0)$,  then
\begin{equation}\label{LONTS}\begin{split}&
|x-y|=|(x_0-y_0)+(x-x_0)+(y-y_0)|\ge|x_0-y_0|-|x-x_0|-|y-y_0|\\&\qquad\qquad\qquad>
4r-\frac{r}2-\frac{r}2=3r>1\end{split}
\end{equation}
and
\begin{equation}\label{cnjaewrt7834tihuSDFGHJ54}
|x-y|\le |x_0-y_0|+|x-x_0|+|y-y_0|<5r.
\end{equation}

Let $u_0\in C^\infty_c(B_1)$ be such that $u_0=1$ in $B_{1/2}$. We define $u(x):=u_0\left(\frac{x-y_0}{r}\right)$.  We observe that $u\in C^\infty_c(B_r(y_0))$,  $u=1$ in $B_{r/2}(y_0)$, and $u=0$ in $B_{r/2}(x_0)$. Consequently, if $(x,y)\in B_{r/2}(x_0)\times B_{r/2}(y_0)$, then $(u(x)-u(y))^2=1$. Thus, by~\eqref{DASOTTO} and~\eqref{LONTS},
we obtain that
\begin{equation}\label{PKSMDxPJMDOJe3f13:S}
\begin{split}
&\iint_{\R^{2N}}(u(x)-u(y))^2 k_A(x,y)\,dx\,dy\ge 
\iint_{B_{r/2}(x_0)\times B_{r/2}(y_0)}(u(x)-u(y))^2 k_A(x,y)\,dx\,dy\\
&\qquad=
\iint_{B_{r/2}(x_0)\times B_{r/2}(y_0)}k_A(x,y)\,dx\,dy\ge c\iint_{B_{r/2}(x_0)\times B_{r/2}(y_0)}\frac{dx\,dy}{|x-y|^{N+\gamma}}.
\end{split}
\end{equation}
Therefore,
thanks to~\eqref{cnjaewrt7834tihuSDFGHJ54},
\begin{equation}\label{nhbgvfcdxsza<zsxdcfvgbhjnkml,.ò}
\iint_{\R^{2N}}(u(x)-u(y))^2 k_A(x,y) dx dy\ge c\, (5r)^{-N-\gamma}|B_{r/2}(x_0)| |B_{r/2}(y_0)| =\overline c r^{N-\gamma},
\end{equation}
where
$$\overline c:=\frac{c\,\omega^2_{N-1}}{2^{2N} 5^{N+\gamma}}.$$

Moreover,  since the fractional gradient is homogeneous of degree $s$,
\[
D^s u(x) = \frac1{r^s} D^s u_0\left(\frac{x-y_0}{r}\right).
\]
Accordingly,
\[
\begin{split}
&\int_{\R^N} A(x)D^s u(x)\cdot D^s u(x) dx\le\|A\|_{L^\infty(\R^N, \R^{N\times N})}\int_{\R^N}|D^s u(x)|^2\,dx\\
&=\frac{\|A\|_{L^\infty(\R^N, \R^{N\times N})}}{r^{2s}}\int_{\R^N}\left|D^s u_0\left(\frac{x-y_0}{r}\right)\right|^2 dx = C r^{N-2s},
\end{split}
\]
where
$$C:=\|A\|_{L^\infty(\R^N, \R^{N\times N})}\int_{\R^N}|D^s u_0(x)|^2 \,dx.$$
Notice in particular that $C>0$ is independent of $r$. 

Thus, combining the latter inequality with \eqref{OJS02irkt-3} and \eqref{nhbgvfcdxsza<zsxdcfvgbhjnkml,.ò}, we obtain that~$\frac{\overline c}{C}r^{2s-\gamma}\le1$. Since~$r>1$ can be taken arbitrarily large, we necessarily have~$\gamma\ge2s$, as desired.
\end{proof}

Combining Theorems \ref{sxdfcghvbjknlmò,à.,lmknjbhvgcfdxszxfcgvbhn} and \ref{NONNETH}, we have the following results:

\begin{corollary}\label{cor1tails}
Let $s\in(0,1)$ and let $A:\R^N\to\R^{N\times N}$ be bounded and measurable and satisfy~\eqref{uniformlyellipticdefn}.
Let $k_A:\R^{2N}\setminus\Delta\to[0,+\infty)$ be measurable and satisfy~\eqref{OJS02irkt-3}, \eqref{mlnjkbhgvcfdxszexdcfgvhbjnkml,}
and
\[
\frac{c}{|x-y|^{N+\gamma}}\le k_A(x,y)\le\frac{M}{|x-y|^{N+\eta}},\qquad |x-y|>1,
\] for some~$\eta>0$, $\gamma\in\R$ and $c, M>0$,

Then,  $\eta\le2s\le\gamma$.
\end{corollary}

\begin{corollary}\label{cor2tails}
Let $s\in(0,1)$ and let $A:\R^N\to\R^{N\times N}$ be bounded and measurable and satisfy~\eqref{uniformlyellipticdefn}.
Let $k_A:\R^{2N}\setminus\Delta\to[0,+\infty)$ be measurable and satisfy~\eqref{OJS02irkt-3}, \eqref{mlnjkbhgvcfdxszexdcfgvhbjnkml,}
and
\[
\frac{c}{|x-y|^{N+\beta}}\le k_A(x,y)\le\frac{M}{|x-y|^{N+\beta}},\qquad |x-y|>1,
\] for some~$\beta>0$ and $0<c\le M<+\infty$.

Then, $\beta=2s$.
\end{corollary}

\section{An explicit representation formula for the kernel}\label{kojhygftrdeswaq	QAWSEDRFTGYUHJ}
The aim of this section is to establish the kernel representation formula stated in Theorem~\ref{thmreg2}. To this end, we first present some preliminary results.

\subsection{An identification via bilinear forms}
To start with, we show that  whenever a symmetric kernel $k_A(x, y)$ represents the bilinear form associated with $A(x)$, the truncated kernels~$k_{A,\varepsilon}(x, y)$ introduced in \eqref{kerneltruncated} recover $k_A(x, y)$ in terms of the corresponding bilinear forms.
\begin{theorem}\label{thmA(x)identification}
Let $A:\R^N\to\R^{N\times N}$ be a bounded, measurable and symmetric matrix. 
Assume that there exists a measurable symmetric kernel $k_A:\R^{2N}\setminus\Delta\to\R$ such that \eqref{mnbhgvcfdxszxdfcgvhb} holds.

Then, for every $u,v\in C_c^\infty(\R^N)$,
\begin{eqnarray*}&&
\lim_{\varepsilon\searrow0}\iint_{\R^{2N}} (u(x)-u(y))(v(x)-v(y)) k_{A,\varepsilon}(x,y)\, dx\, dy\\&&\qquad
 = \iint_{\R^{2N}} (u(x)-u(y))(v(x)-v(y)) k_A(x,y) \,dx\, dy.
\end{eqnarray*}
\end{theorem}

To prove Theorem \ref{thmA(x)identification}, we need some intermediate steps.
For this, for every~$\varepsilon>0$, we define the ``truncated'' fractional gradient as
\begin{equation}\label{defnintegtruncation}
D_\varepsilon^s u(z):=d_{N,s}\int_{\{|z-x|>\varepsilon\}}(u(z)-u(x))\frac{z-x}{|z-x|^{N+s+1}}\, dx.
\end{equation}

\begin{lemma}\label{lemmaidtroncata}
Let $A:\R^N\to\R^{N\times N}$ be a bounded, measurable and symmetric matrix. For $\varepsilon>0$,  let $k_{A,\varepsilon}$ be defined as in \eqref{kerneltruncated}.

Then, for every $u,v\in C_c^\infty(\R^N)$,
\[
\int_{\R^N} A(z)D_\varepsilon^su(z)\cdot D_\varepsilon^sv(z)\, dz = \iint_{\R^{2N}} (u(x)-u(y))(v(x)-v(y))
k_{A,\varepsilon}(x,y) \,dx\, dy.
\]
\end{lemma}

\begin{proof}
For every given $z\in\R^N$, by odd symmetry of the integrand, we have that
\[
\int_{\{|z-x|>\varepsilon\}}\frac{z-x}{|z-x|^{N+s+1}} \,dx=0.
\]
Therefore, by \eqref{defnintegtruncation}, we can write
\begin{equation}\label{mjnkhbgvfcdxrcfvgbhn}
D_\varepsilon^s u(z) = -d_{N,s}\int_{\{|z-x|>\varepsilon\}} u(x)\frac{z-x}{|z-x|^{N+s+1}} \,dx.
\end{equation}
Hence, we have that
\[
A(z)D_\varepsilon^s u(z)\cdot D_\varepsilon^s v(z) = d_{N,s}^2\iint_{\substack{\{|z-x|>\varepsilon\}\\\{|z-y|>\varepsilon\}}} u(x)v(y)\frac{(z-x)\cdot A(z)(z-y)}{|z-x|^{N+s+1}|z-y|^{N+s+1}}\,dx\,dy.
\]

Now, we consider the function
\[
F_\varepsilon(z,x,y):=u(x)v(y)\chi_{\{|z-x|>\varepsilon\}}\chi_{\{|z-y|>\varepsilon\}}\frac{(z-x)\cdot A(z)(z-y)}{|z-x|^{N+s+1}|z-y|^{N+s+1}}
\] and we observe that
\begin{equation*}
A(z)D_\varepsilon^s u(z)\cdot D_\varepsilon^s v(z) = 
d_{N,s}^2\iint_{\R^{2N}} F_\varepsilon(z,x,y)\,dx\,dy
.\end{equation*}

We claim that
\begin{equation}\label{klmjhgftdrxsdcfvgbhn}
F_\varepsilon\in L^1(\R^{3N}).
\end{equation}
Indeed, since $A\in L^\infty(\R^N, \R^{N\times N})$, we infer that
\[
|F_\varepsilon(z,x,y)|\le\|A\|_{L^\infty(\R^N, \R^{N\times N})}|u(x)||v(y)| \chi_{\{|z-x|>\varepsilon\}} \chi_{\{|z-y|>\varepsilon\}}\frac{1}{|z-x|^{N+s}|z-y|^{N+s}}.
\]
We now choose $R>0$ such that $(\mbox{supp}(u)\cup \mbox{supp}(v))\subset B_R$. In the region $\{|z|\le 2R\}$, we have
\[
|F_\varepsilon(z,x,y)|\le\|A\|_{L^\infty(\R^N, \R^{N\times N})}\varepsilon^{-2N-2s}|u(x)||v(y)|.
\]
Hence, it follows that
\begin{equation}\label{456cjewkr7345iWERTYUIbdenwyjr7654}
\int_{B_{2R}}\int_{\R^N}\int_{\R^N} |F_\varepsilon(z,x,y)|\, dx\, dy\, dz\le\|A\|_{L^\infty(\R^N, \R^{N\times N})}\varepsilon^{-2N-2s}
|B_{2R}|\|u\|_{L^1(\R^N)}\|v\|_{L^1(\R^N)},
\end{equation} which is finite.

We now consider the region $\{|z|>2R\}$. If $x\in \mbox{supp}(u)\subset B_R$, then
\[
|z-x|\ge |z|-|x|\ge |z|-R >\frac{|z|}{2}.
\]
Similarly, for every $y\in\mbox{supp}(v)$, we have $|z-y|> |z|/2$. Accordingly, 
\[
|F_\varepsilon(z,x,y)|\le\|A\|_{L^\infty(\R^N, \R^{N\times N})}\frac{|u(x)||v(y)|}{|z-x|^{N+s}|z-y|^{N+s}}\le 2^{2N+2s} \|A\|_{L^\infty(\R^N)}|u(x)||v(y)||z|^{-2N-2s}.
\]
Therefore, we have
\begin{eqnarray*}&&
\int_{\R^N\setminus B_{2R}}\int_{\R^N}\int_{\R^N}|F_\varepsilon(z,x,y)|\, dx dy\, dz\\&&\qquad\le 2^{2N+2s}\|A\|_{L^\infty(\R^N, \R^{N\times N})}\|u\|_{L^1(\R^N)}\|v\|_{L^1(\R^N)}\int_{\{|z|>2R\}} |z|^{-2N-2s} \,dz,
\end{eqnarray*} which is finite.
This and~\eqref{456cjewkr7345iWERTYUIbdenwyjr7654}
entail~\eqref{klmjhgftdrxsdcfvgbhn}. 

In light of~\eqref{klmjhgftdrxsdcfvgbhn}, we can employ Fubini's theorem
and, recalling~\eqref{kerneltruncated}, we obtain that
\begin{equation}\label{hbgvfcdxsexrdctfvygbhjnkml,1}
\begin{split}
&\int_{\R^N} A(z)D_\varepsilon^s u(z)\cdot D_\varepsilon^s v(z) \,dz = d_{N,s}^2 \iiint_{\R^{3N}} F_\varepsilon(z,x,y) \,dx \,d, dz\\
&=d_{N,s}^2\iint_{\R^{2N}} u(x)v(y)\left(\;
\int_{\{\min\{|z-x|,|z-y|\}>\varepsilon\} }\frac{(z-x)\cdot A(z)(z-y)}{|z-x|^{N+s+1}|z-y|^{N+s+1}} dz \right) dx\,dy\\
&= -2\iint_{\R^{2N}} u(x)v(y)k_{A,\varepsilon}(x,y) \,d, dy.
\end{split}
\end{equation}

We observe that, since $A(z)$ is symmetric,
\[
(z-y)\cdot A(z)(z-x)=A(z)^T (z-y)\cdot (z-x)=A(z) (z-y)\cdot (z-x),
\]
therefore
\begin{equation}\label{jknbhgvcfdxszerdctfvgbhnj}
k_{A,\varepsilon}(x, y) = k_{A,\varepsilon}(y, x).
\end{equation}

Moreover, for every $x\in\R^N$, we have
\begin{eqnarray*}&&
\int_{\R^N} |k_{A,\varepsilon}(x,y)| \,dy \le \frac{d_{N,s}^2}{2}\|A\|_{L^\infty(\R^N, \R^{N\times N})}\int_{\R^N}\frac{\chi_{\{|z-x|>\varepsilon\}}}{|z-x|^{N+s}}\left(\;\int_{\R^N}\frac{\chi_{\{|z-y|>\varepsilon\}}}{|z-y|^{N+s}} \,dy\right)\,dz\\
&&\qquad=
\frac{d_{N,s}^2}{2}\|A\|_{L^\infty(\R^N, \R^{N\times N})}
\left( \; \int_{\{|h|>\varepsilon\}}\frac{dh}{|h|^{N+s}} \right)^2
=\frac{d_{N,s}^2}{2}\|A\|_{L^\infty(\R^N, \R^{N\times N})}
\left(\frac{\omega_{N-1}}{s\varepsilon^{s}}\right)^2
.\end{eqnarray*}
Accordingly, Fubini's theorem gives
\begin{equation}\label{l,òkmjnbhgvfytcdrsxeaw<<zesxrdtfgbhjnmk}
\int_{\R^N} k_{A,\varepsilon}(x,y)\,dy = -\frac{d_{N,s}^2}{2}\int_{\{|z-x|>\varepsilon\}}\frac{z-x}{|z-x|^{N+s+1}} A(z)\cdot\left(\;\int_{\{|z-y|>\varepsilon\}}\frac{z-y}{|z-y|^{N+s+1}} \,dy\right) \,dz =0.
\end{equation}

We now set
\[
I_\varepsilon:=\iint_{\R^{2N}}(u(x)-u(y))(v(x)-v(y)) k_{A,\varepsilon}(x,y)\, dx \,dy.
\]
For every $\varepsilon>0$, we have that $k_{A,\varepsilon}$ is integrable in each variable. Thus, employing also~\eqref{jknbhgvcfdxszerdctfvgbhnj},
we can write
\[
\begin{split}
I_\varepsilon &=\iint_{\R^{2N}} u(x)v(x)k_{A,\varepsilon}(x,y) \,dx\, dy + \iint_{\R^{2N}} u(y)v(y)k_{A,\varepsilon}(x,y)
\, dx\, dy\\
&\qquad-\iint_{\R^{2N}} u(x)v(y)k_{A,\varepsilon}(x,y) \,dx\, dy -\iint_{\R^{2N}} u(y)v(x)k_{A,\varepsilon}(x,y)\, dx \,dy\\
&=2\iint_{\R^{2N}} u(x)v(x)k_{A,\varepsilon}(x,y) \,dx\, dy-2\iint_{\R^{2N}} u(x)v(y)k_{A,\varepsilon}(x,y) \,dx\, dy
\end{split}
\]
and all the integrals are absolutely convergent. 

Now, by \eqref{l,òkmjnbhgvfytcdrsxeaw<<zesxrdtfgbhjnmk},
\[
\iint_{\R^{2N}} u(x)v(x)k_{A,\varepsilon}(x,y)\, d, dy =\int_{\R^N} u(x)v(x) \left(\;\int_{\R^N} k_{A,\varepsilon}(x,y)\, dy
\right) \,dx=0.
\]
Therefore, it follows that
\[
I_\varepsilon=-2\iint_{\R^{2N}}u(x)v(y)k_{A,\varepsilon}(x,y) \,dx\, dy.
\]
Recalling \eqref{hbgvfcdxsexrdctfvygbhjnkml,1}, we obtain the desired result.
\end{proof}

We now proceed by removing the truncation.

\begin{lemma}\label{lemmaDsL2}
Let $u\in C_c^\infty(\R^N)$. Then, 
\[
\lim_{\varepsilon\searrow 0}D^s_\varepsilon u = D^s u \quad\mbox{ in }L^2(\R^N, \R^N).
\]
\end{lemma}

\begin{proof}
By the fundamental theorem of calculus, we have that, for~$|h|\le\varepsilon$,
\[
u(z)-u(z-h)=\int_0^1 h\cdot\nabla u(z-th) \,dt.
\]
Therefore
\begin{eqnarray*}
|D^s u(z)-D^s_\varepsilon u(z)|& \le& 
d_{N,s}\int_{\{|h|\le\varepsilon\}}\frac{|u(z)-u(z-h)|}{|h|^{N+s}} \,dh\\
&\le&
d_{N,s}\int_{\{|h|\le\varepsilon\}}\left(\;
\int_0^1 |\nabla u(z-th) |\,dt\right)\frac{dh}{|h|^{N+s-1}}.
\end{eqnarray*}
Thus, by Minkowski's inequality, we infer that
\begin{eqnarray*}
\|D^s u-D^s_\varepsilon u\|_{L^2(\R^N)}&\le& d_{N,s}
\left(\;\int_{\R^N}
\left(\;\int_{\{|h|\le\varepsilon\}}\left(\;
\int_0^1 |\nabla u(z-th) |\,dt\right)\frac{dh}{|h|^{N+s-1}}\right)^2
\,dz\right)^{\frac12}
\\&\le&
d_{N,s}\int_{\{|h|\le\varepsilon\}} \int_0^1
\left(\;\int_{\R^N} |\nabla u(z-th) |^2\,dz\right)^{\frac12}
\frac{dh}{|h|^{N+s-1}}\,dt
\\&\le&
d_{N,s}\|\nabla u\|_{L^2(\R^N)}\int_{\{|h|\le\varepsilon\}}
\frac{dh}{|h|^{N+s-1}}\\
&=& \frac{d_{N,s}\,\omega_{N-1}\,\varepsilon^{1-s}}{(1-s)}
\|\nabla u\|_{L^2(\R^N)}
.
\end{eqnarray*}
Letting $\varepsilon\searrow 0$, we obtain the desired result.
\end{proof}

\begin{lemma}\label{lemmaADepstozero}
Let $A:\R^N\to\R^{N\times N}$ be a bounded, measurable and symmetric matrix. 

Then, for every $u,v\in C_c^\infty(\R^N)$, 
\[
\lim_{\varepsilon\searrow 0}\int_{\R^N} A(z)D_\varepsilon^s u(z) \cdot D_\varepsilon^s v(z) dz = \int_{\R^N} A(z)D^s u(z)\cdot D^s v(z) dz.
\]
\end{lemma}

\begin{proof}
For every~$z\in\R^N$, we write
\[
\begin{split}
&A(z) D_\varepsilon^s u(z)\cdot D_\varepsilon^s v(z) - A(z) D^s u(z)\cdot D^s v(z)\\
&\quad= A(z)(D_\varepsilon^s u(z)-D^s u(z))\cdot D_\varepsilon^s v(z) + A D^s u(z)\cdot(D_\varepsilon^s v(z)-D^s v(z)).
\end{split}
\]
Hence, we get
\[
\begin{split}
&\left|\,\int_{\R^N}A(z) D_\varepsilon^s u(z)\cdot D_\varepsilon^s v(z) - \int_{\R^N} A(z) D^s u(z)\cdot D^s v(z)\right|\\
&\quad\le\|A\|_{L^\infty(\R^N, \R^{N\times N})}\|D_\varepsilon^s u-D^s u\|_{L^2(\R^N)} \|D_\varepsilon^s v\|_{L^2(\R^N)}\\
&\qquad\quad+\|A\|_{L^\infty(\R^N, \R^{N\times N})}\|D^s u\|_{L^2(\R^N)}\|D_\varepsilon^s v-D^s v\|_{L^2(\R^N)}.
\end{split}
\]
This and Lemma \ref{lemmaDsL2} give the desired convergence.
\end{proof}

\begin{proof}[Proof of Theorem \ref{thmA(x)identification}]
The desired result plainly follows combining Lemmata~\ref{lemmaidtroncata} and~\ref{lemmaADepstozero} with the identity~\eqref{mnbhgvcfdxszxdfcgvhb}.
\end{proof}

\subsection{Proof of Theorem \ref{thmreg2}}
We stress that Theorem \ref{thmA(x)identification} identifies the kernel only at the level of the associated bilinear forms through a limit process in the principal value sense. To obtain a pointwise representation outside the diagonal, we need additional regularity on the matrix~$A(x)$.
This allows us to control the singular remainders near $z=x$ and $z=y$, to prove the existence of the pointwise limit of $k_{A,\varepsilon}(x,y)$, and then to identify it with $k_A(x, y)$ outside the diagonal.

\begin{lemma}\label{lemmareg1}
Let $s\in(0,1)$ and let $A\in C(\R^N, \R^{N\times N})$ be a bounded and symmetric matrix satisfying~\eqref{mkojihguyftdreswxrdcftvgybhjn}
for every compact set~$K\subset\R^N$.

Then, for every $x\ne y$, the limit in \eqref{INT2PVdefn} exists and is finite.

In addition, let
\begin{equation}\label{bxszasexrdtfyguhijokplèòl}
\widetilde k_A(x,y):= -\frac{d_{N,s}^2}{2}I_A(x,y).
\end{equation}
Then, for every compact set $E\subset \R^{2N}\setminus\Delta$,
\begin{equation}\label{desirtwqdvba68754olimit}
\lim_{\varepsilon\searrow0}\sup_{(x,y)\in E}\left|k_{A,\varepsilon}(x,y) - \widetilde k_A(x,y)\right|=0.
\end{equation}
\end{lemma}

\begin{proof}
For every~$x\in\R^N$, we use the notation
\[
H_x(z):=\frac{z-x}{|z-x|^{N+s+1}}.
\]
For~$x\neq y$, we take
\begin{equation}\label{5678vneti43u9tWERTYUr6873tgfbka987609876}
\rho:=\rho_{x,y}\in\left(0,\min\left\{1,\frac{|x-y|}{4}\right\}\right).
\end{equation}
With this choice, the balls $B_\rho(x)$ and $B_\rho(y)$ are disjoint. 

We claim that, for every~$x\neq y$, the function
\begin{equation}\label{bhgfdxszsxdcfvgb}
z\in B_\rho(x)\mapsto H_x(z)\cdot\Big(A(z)H_y(z)-A(x)H_y(x)\Big) \mbox{ is absolutely integrable in $B_\rho(x)$}.
\end{equation}
Indeed, thanks to~\eqref{5678vneti43u9tWERTYUr6873tgfbka987609876},
if~$z\in B_\rho(x)$,
\[
|z-y|\ge |x-y|-|z-x|\ge\frac34|x-y|\ge 3\rho.
\]
Therefore $H_y$ is smooth in a neighborhood of $\overline{B_\rho(x)}$. In particular, there exists $C_{x,y, \rho}>0$ such that, for every $z\in B_\rho(x)$,
\[
|H_y(z)|\le C_{x,y}\quad\mbox{ and }\quad |H_y(z)-H_y(x)| \le C_{x,y,\rho}|z-x|.
\]
Accordingly, recalling the notation in~\eqref{omegaAkdefn},
\[
\begin{split}
|A(z)H_y(z)-A(x)H_y(x)|&\le |A(z)-A(x)|\,|H_y(z)| + |A(x)|\,|H_y(z)-H_y(x)|\\
&\le C_{x,y, \rho}\,\omega_{A, K}(|z-x|) + C_{x,y, \rho} \|A\|_{L^\infty(\R^N, \R^{N\times N})}|z-x|,
\end{split}
\] for every compact set~$K\subset\R^N$ containing~$x$ and~$z$.

As a result,
\begin{equation}\label{9876yhgjkdfrgjhyuiWAEZSXRDCFTYGuhd3478y}
\Big|H_x(z)\cdot\left(A(z)H_y(z)-A(x)H_y(x)\right)\Big|\le C_{x,y, \rho}\frac{\omega_{A, K}(|z-x|)}{|z-x|^{N+s}} + C_{x,y, \rho} \frac{\|A\|_{L^\infty(\R^N, \R^{N\times N})}}{|z-x|^{N+s-1}}.
\end{equation}
Now, by using polar coordinates centered at $x$, we obtain that, up to renaming constants,
\begin{eqnarray*}&&
\int_{B_\rho(x)}\Big|H_x(z)\cdot\left(A(z)H_y(z)-A(x)H_y(x)\right)\Big| \,dz\\
&&\qquad\le C_{x,y, \rho} \int_0^\rho \frac{\omega_{A, K}(r)}{r^{1+s}}\,dr + C_{x,y, \rho} \|A\|_{L^\infty(\R^N, \R^{N\times N})} \int_0^\rho\frac{dr}{ r^{s}}.
\end{eqnarray*}
Since the first integral is finite thanks to \eqref{mkojihguyftdreswxrdcftvgybhjn}, from this we deduce the desired claim
in~\eqref{bhgfdxszsxdcfvgb}.

Moreover, we observe that~$H_x(z)$ is odd with respect to $x$. Therefore, for every $\varepsilon\in(0,\rho)$,
\[
\int_{B_\rho(x)\setminus B_\varepsilon(x)} H_x(z)\cdot A(x)H_y(x) \,dz =0
,\]
hence \[
\lim_{\varepsilon\searrow0}\int_{B_\rho(x)\setminus B_\varepsilon(x)} H_x(z)\cdot A(z)H_y(z) dz = \int_{B_\rho(x)}
H_x(z)\cdot\left( A(z)H_y(z)-A(x)H_y(x)\right) dz,
\]
and the integral at the right hand side is absolutely convergent thanks to \eqref{bhgfdxszsxdcfvgb}. This proves the existence of the limit near the singularity $z=x$.

We now address the singularity at $z=y$. Since $A(z)$ is symmetric, we have that
\[
H_x(z)\cdot A(z)H_y(z) = H_y(z)\cdot A(z)H_x(z).
\]
Therefore, by interchanging $x$ and $y$ and using the previous argument,  we infer that the limit
\[
\lim_{\varepsilon\searrow0}\int_{B_\rho(y)\setminus B_\varepsilon(y)} H_x(z)\cdot A(z)H_y(z) \,dz
\]
exists and is finite. 

We now consider the region $\R^N\setminus\left(B_\rho(x)\cup B_\rho(y)\right)$. Let~$R:=R_{x,y}>0$ be such that~$B_\rho(x)\cup B_\rho(y)\subset B_R$. Then, if~$z\in\R^N\setminus B_{2R}$,
$$ |z-x|\ge |z|-|x|\ge\frac{|z|}2 $$
and similarly~$|z-y|\ge |z|/2$.

Therefore
\[
\begin{split}&
\int_{\R^N\setminus\left(B_\rho(x)\cup B_\rho(y)\right)}\left|H_x(z)\cdot A(z)H_y(z)\right| \,dz \\&\qquad
\le \|A\|_{L^\infty(\R^N, \R^{N\times N})}\int_{\R^N\setminus\left(B_\rho(x)\cup B_\rho(y)\right)}\frac{dz}{|z-x|^{N+s}|z-y|^{N+s}}\\
&\qquad
\le C\|A\|_{L^\infty(\R^N, \R^{N\times N})}\left(\;\int_{B_{2R}\setminus\left(B_\rho(x)\cup B_\rho(y)\right)}\frac{dz}{|z-x|^{N+s}|z-y|^{N+s}}+
\int_{\R^N\setminus B_{2R}}\frac{dz}{|z|^{2N+2s} }\right)
\\
&\qquad <+\infty.
\end{split}
\]
Gathering these pieces of information,
we conclude that the limit in \eqref{INT2PVdefn} exists and is finite for every~$x\ne y$.

Let now~$E\subset\R^{2N}\setminus\Delta$ be compact. Thus, there exists~$d>0$ such that~$|x-y|\ge d$ for every~$(x,y)\in E$.
We take
\[
\rho\in\left(0,\min\left\{1,\frac d4\right\}\right).
\]
Since $E$ is compact, if~$(x,y)\in E$ then both~$x$ and~$y$ belong to a compact set of~$\R^N$. Hence, there exists a compact set $K\subset\R^N$ containing all the balls $B_\rho(x)$ and $B_\rho(y)$ for $(x,y)\in E$. 

Moreover, since $|x-y|\ge d$, all the estimates above can be chosen
uniformly for $(x,y)\in E$. In particular, 
we infer from~\eqref{9876yhgjkdfrgjhyuiWAEZSXRDCFTYGuhd3478y} that
there exists a constant~$C_E>0$ such that, for every~$(x,y)\in E$ and every~$z\in B_\rho(x)$,
\[
\Big|H_x(z)\cdot\left(A(z)H_y(z)-A(x)H_y(x) \right)\Big|\le C_E\frac{\omega_{A,K}(|z-x|)}{|z-x|^{N+s}} +\frac{C_E}{|z-x|^{N+s-1}}
.\]
We notice that the analogous estimate holds near $y$.

For $\varepsilon\in(0,\rho)$,  we set
\begin{equation}\label{kmjhbgvfcdxrcfvgb}
I_{A,\varepsilon}(x,y):= \int_{\R^N\setminus\left(B_\varepsilon(x)\cup B_\varepsilon(y)\right)}\frac{(z-x)\cdot A(z)(z-y)}{|z-x|^{N+s+1}|z-y|^{N+s+1}} \,dz.
\end{equation}
Since the constant terms cancel by odd symmetry, we obtain that
\[
\begin{split}
|I_A(x,y)-I_{A,\varepsilon}(x,y)|&
=\left| \;\int_{B_\varepsilon(x)\cup B_\varepsilon(y)} H_x(z) \cdot A(x)H_y(z)\,dz
\right|\\&\le 2
\int_{B_\varepsilon(x)} \Big|H_x(z)\cdot\left(A(z)H_y(z)-A(x)H_y(x) \right)\Big|\,dz\\&
\le C_E\left(\;\int_{B_\varepsilon(x)}\frac{\omega_{A,K}(|z-x|)}{|z-x|^{N+s}}\, dz +\int_{B_\varepsilon(x)} |z-x|^{1-N-s} \,dz\right)\\
&=C_{E,N}\left(\;\int_0^\varepsilon \frac{\omega_{A,K}(r)}{r^{1+s}} \,dr
+\int_0^\varepsilon r^{-s} \,dr \right)\\
&=C_{E,N,s}\left(\;\int_0^\varepsilon \frac{\omega_{A,K}(r)}{r^{1+s}} \,dr
+\varepsilon^{1-s}\right)
.
\end{split}
\]
Hence,
\[
\sup_{(x,y)\in E}|I_A(x,y)-I_{A,\varepsilon}(x,y)|\le C_{E,N,s}\left(\;\int_0^\varepsilon \frac{\omega_{A,K}(r)}{r^{1+s}} \,dr
+\varepsilon^{1-s}\right).
\]
By \eqref{mkojihguyftdreswxrdcftvgybhjn}, we thus find that
\[
\lim_{\varepsilon\searrow0}\sup_{(x,y)\in E}|I_A(x,y)-I_{A,\varepsilon}(x,y)|=0.
\]
Recalling \eqref{kerneltruncated}, \eqref{bxszasexrdtfyguhijokplèòl},
and~\eqref{kmjhbgvfcdxrcfvgb}, we obtain the desired limit in~\eqref{desirtwqdvba68754olimit}.
\end{proof}

We are now in the position to prove Theorem \ref{thmreg2}:

\begin{proof}[Proof of Theorem \ref{thmreg2}]
By Lemma \ref{lemmareg1}, we know that, for every compact set~$E\subset\R^{2N}\setminus\Delta$,
\begin{equation}\label{nbhvgcfdxsdrcftvgbhjnkm}
\lim_{\varepsilon\searrow0}\sup_{(x,y)\in E}\left|k_{A,\varepsilon}(x,y)-\widetilde k_A(x,y)\right|=0.
\end{equation}

Let now $u,v\in C_c^\infty(\R^N)$ with disjoint supports. Thus, $E:=\mbox{supp}(u)\times\mbox{supp}(v)$ is a compact subset of $\R^{2N}\setminus\Delta$.
Moreover,
\begin{eqnarray*}&&
\left| \iint_{\R^{2N}} u(x)v(y) \left( k_{A,\varepsilon}(x,y)-\widetilde k_A(x,y) \right)\, dx\, dy \right|\\&&\qquad\le \|u\|_{L^1(\R^N)} \|v\|_{L^1(\R^N)} \sup_{(x,y)\in E} \left| k_{A,\varepsilon}(x,y)-\widetilde k_A(x,y) \right|.
\end{eqnarray*}
Therefore, by \eqref{nbhvgcfdxsdrcftvgbhjnkm},
\begin{equation}\label{nbhgvfctdxrsrdctfvghbjn}
\lim_{\varepsilon\searrow0} \iint_{\R^{2N}} u(x)v(y)k_{A,\varepsilon}(x,y) \,dx\,dy = \iint_{\R^{2N}} u(x)v(y)\widetilde k_A(x,y) \,dx\, dy. 
\end{equation}

Now, since $u$ and $v$ have disjoint supports,  we have
\[
(u(x)-u(y))(v(x)-v(y)) = -u(x)v(y)-u(y)v(x).
\]
Therefore, using the symmetry of the kernel $k_{A,\varepsilon}$,
\[
\begin{split}
&\iint_{\R^{2N}}(u(x)-u(y))(v(x)-v(y)) k_{A,\varepsilon}(x,y) \,dx\, dy\\
&\quad= -\iint_{\R^{2N}}u(x)v(y)k_{A,\varepsilon}(x,y) \,dx\, dy -\iint_{\R^{2N}} u(y)v(x)k_{A,\varepsilon}(x,y)\, dx\, dy\\
&\quad= -2\iint_{\R^{2N}} u(x)v(y)k_{A,\varepsilon}(x,y) dx dy.
\end{split}
\]

The same argument, using the symmetry of the kernel $k_A$, yields
\[
\iint_{\R^{2N}} (u(x)-u(y))(v(x)-v(y)) k_A(x,y)\, dx\, dy = -2\iint_{\R^{2N}} u(x)v(y)k_A(x,y)\, dx\, dy.
\]
Hence, combining these identities with Theorem \ref{thmA(x)identification}, we get
\[
\lim_{\varepsilon\searrow0}\iint_{\R^{2N}} u(x)v(y)k_{A,\varepsilon}(x,y) \,dx \,dy = \iint_{\R^{2N}} u(x)v(y)k_A(x,y)\, dx\, dy.
\]
Comparing this expression with \eqref{nbhgvfctdxrsrdctfvghbjn}, we infer that
\[
\iint_{\R^{2N}} u(x)v(y)\widetilde k_A(x,y) \,dx\, dy = \iint_{\R^{2N}} u(x)v(y)k_A(x,y) \,dx\, dy.
\]

By \cite[Lemma 2.5]{DSVPART1}, we thus\footnote{Notice indeed that~\cite[Lemma 2.5]{DSVPART1} would be true even taking~$u,v\in C^\infty_c(\R^N)$ with disjoint supports in the assumptions.} have
\[
k_A(x,y)=\widetilde k_A(x,y) \quad\mbox{ a.e. in } \R^{2N}\setminus\Delta.
\]
Hence, the desired result follows by \eqref{bxszasexrdtfyguhijokplèòl}.
\end{proof}

We notice that a simpler regularity assumption on the matrix~$A$ implying \eqref{mkojihguyftdreswxrdcftvgybhjn} for every compact set $K\subset\R^N$ is
\[
A\in C_{\rm loc}^{0,\beta}(\R^N)\quad\mbox{ for some }\beta>s.
\]
Indeed, $\omega_{A,K}(r)\le C_Kr^\beta$,  and therefore
\[
\int_0^1\frac{\omega_{A,K}(r)}{r^{1+s}} dr\le C_K\int_0^1 r^{\beta-s-1} dr<+\infty.
\]
We stress that the threshold $\beta>s$ is sharp in the Hölder scale. Indeed, in general, $C^{0,s}$ regularity is not sufficient for the existence of the limit. We stress that the integrability assumption on the modulus of continuity is essentially sharp among conditions depending only on $\omega_{A,K}$, although additional cancellations may yield convergence for some particular matrices.

\subsection{A class of admissible matrices}
We now consider a class of nonconstant matrices for which \eqref{CHARKE} admits a more explicit formulation.

Let $M\ge 1$. For $j\in\{1, \dots, M\}$,  assume that
\begin{equation}\label{matrconstxi}
\begin{split}
&A_0,B_j,C_j\in\R^{N\times N} \mbox{ are constant and symmetric matrices, and } \xi_j\in\R^N\setminus\{0\}.
\end{split}
\end{equation}
We also suppose that~$A_0$ is positive definite and that
\begin{equation}\label{jihugyfctdxrsdctfvgbhn}
\sum_{j=1}^M\left(\|B_j\| + \|C_j\|\right) < \lambda_{\min}(A_0).
\end{equation}
We define 
\begin{equation}\label{Asincosdefn}
A(z):=A_0 + \sum_{j=1}^M\Big(B_j\cos(\xi_j\cdot z) + C_j\sin(\xi_j\cdot z)\Big).
\end{equation}

In this setting, we now prove that the matrix $A(z)$
satisfies all the assumptions required by Theorem \ref{thmreg2} for the pointwise identification of the kernel $k_A(x, y)$.

\begin{proposition}\label{jinhbugvyfctdxrsezaxrdctfvybhuijnkm}
The matrix $A(z)$ is bounded, symmetric, uniformly elliptic,
and belongs to~$C^{0,1}(\R^N,\R^{N\times N})$.
\end{proposition}

\begin{proof} Since~$A_0$, $B_j$, and~$C_j$ are symmetric matrices, we have that~$A(z)$ is symmetric as well.
Moreover, by~\eqref{jihugyfctdxrsdctfvgbhn}
\begin{eqnarray*}&&
\|A(z)\|\le\|A_0\| + \sum_{j=1}^M \|B_j\| |\cos(\xi_j\cdot z)| + \sum_{j=1}^M \|C_j\| |\sin(\xi_j\cdot z)|\\&&\qquad\quad\le
\|A_0\| + \sum_{j=1}^M \left(\|B_j\| + \|C_j\|\right)\le
\|A_0\|+\lambda_{\min}(A_0),
\end{eqnarray*}
which shows that $A(z)$ is bounded.

We prove next the uniform ellipticity. To this end, we set
\[
\lambda_0:=\lambda_{\min}(A_0)\quad\mbox{ and }\quad \lambda:=\lambda_0 -\sum_{j=1}^M\left(\|B_j\| + \|C_j\|\right).
\]
By \eqref{jihugyfctdxrsdctfvgbhn}, we have that $\lambda>0$. 

Since $A_0$ is positive definite, for every $\eta\in\R^N$,
\[
\eta\cdot A_0\eta\ge\lambda_0|\eta|^2.
\]
Furthermore, for every $j\in\{1,\dots, M\}$,
\[
|\eta\cdot B_j\eta|\le\|B_j\| |\eta|^2 \qquad\mbox{ and }\qquad |\eta\cdot C_j\eta|\le\|C_j\| |\eta|^2,
\]
Combining these pieces of information, we obtain that, for every $\eta\in\R^N$,
\[
\begin{split}&
\eta\cdot A(z)\eta=\eta\cdot A_0\eta + \sum_{j=1}^M\cos(\xi_j\cdot z)\eta\cdot B_j\eta + \sum_{j=1}^M
\sin(\xi_j\cdot z)\eta\cdot C_j\eta\\
&\qquad \ge\lambda_0|\eta|^2 - \sum_{j=1}^M \|B_j\| |\eta|^2 - \sum_{j=1}^M \|C_j\| |\eta|^2 = \lambda|\eta|^2.
\end{split}
\]

Moreover,
\[
\eta\cdot A(z)\eta\le\left(\|A_0\| + \sum_{j=1}^M \left(\|B_j\| + \|C_j\| \right)\right)|\eta|^2.
\]
Therefore, we have that $A(x)$ is uniformly elliptic.

Finally, we show that~$A(z)$ is globally Lipschitz. Indeed, we notice that
\[
|\cos a-\cos b|\le |a-b|\quad\mbox{ and }\quad |\sin a-\sin b|\le |a-b|.
\]
Thus, for every $z,w\in\R^N$,  we infer that
\[
\begin{split}
\|A(z)-A(w)\|&\le\sum_{j=1}^M \|B_j\| |\cos(\xi_j\cdot z)-\cos(\xi_j\cdot w)| + \sum_{j=1}^M \|C_j\| |\sin(\xi_j\cdot z)-\sin(\xi_j\cdot w)|\\
&\le\sum_{j=1}^M |\xi_j|\left(\|B_j\| + \|C_j\|\right)|z-w|,
\end{split}
\] which entails that~$A\in C^{0,1}\left(\R^N, \R^{N\times N}\right)$.
\end{proof}

We now determine the corresponding structure of the kernel. 
For this, for a constant and symmetric matrix~$M\in\R^{N\times N}$, we use the notation
\begin{equation}\label{kjnhbgvcfdxrctfvgbhjnk}
\mathbb K_{M,\xi}(h):=-\frac{d_{N,s}^2}{2}\lim_{\varepsilon\searrow0}\int_{\R^N\setminus\left(B_\varepsilon(h)\cup B_\varepsilon(-h)\right)}
\frac{w\cdot Mw-h\cdot Mh}{|w-h|^{N+s+1}|w+h|^{N+s+1}}\cos(\xi\cdot w) \,dw.
\end{equation}

We this notation, we can prove the following:

\begin{theorem}\label{theoremtrig}
For almost every $x\ne y$, we have
\begin{eqnarray*}
k_A(x,y)&=& k_{A_0}(x,y) + \sum_{j=1}^M \cos\left(\xi_j\cdot\frac{x+y}{2}\right)\mathbb K_{B_j,\xi_j}\left(\frac{x-y}{2}\right) \\&&\qquad+ \sum_{j=1}^M \sin\left(\xi_j\cdot\frac{x+y}{2}\right)\mathbb K_{C_j,\xi_j}\left(\frac{x-y}{2}\right).
\end{eqnarray*}
\end{theorem}

\begin{proof}
We first point out that, if~$M\in\R^{N\times N}$ is a 
constant and symmetric matrix and~$\xi\in\R^N$, then
the matrix
$$ A_\xi(w):=M\cos(\xi\cdot w) $$ is bounded, symmetric and smooth.
Accordingly,
Lemma~\ref{lemmareg1} can be employed, guaranteeing
that the limit in the definition of~$I_A(x,y)$ in~\eqref{INT2PVdefn} exists and is finite for every~$x\neq y$. In particular, for every~$h\neq0$, the limit
\begin{equation}\label{cnewtiy43782592gy678dghj}\lim_{\varepsilon\searrow0}\int_{\R^N\setminus\left(B_\varepsilon(h)\cup B_\varepsilon(-h)\right)}
\frac{(w-h)\cdot M(w+h)}{|w-h|^{N+s+1}|w+h|^{N+s+1}}\cos(\xi\cdot w) \,dw
\end{equation}
exists and is finite.

Moreover, since $M$ is symmetric, we have that
\begin{equation}\label{kjhbgvfcdxcfvghbjnm}
(w-h)\cdot M(w+h) = w\cdot Mw+w\cdot Mh-h\cdot Mw-h\cdot Mh = w\cdot Mw-h\cdot Mh.
\end{equation}
Therefore, in light of~\eqref{cnewtiy43782592gy678dghj}, the limit
in the definition~\eqref{kjnhbgvcfdxrctfvgbhjnk}
exists and is finite.

Let now~$x, y\in\R^N$ with $x\ne y$ and set
\begin{equation}\label{mhdefn}
m:=\frac{x+y}{2}\qquad\mbox{ and }\qquad h:=\frac{x-y}{2}.
\end{equation}
In this way, $x=m+h$ and $y=m-h$. 

By Theorem \ref{thmreg2} (which can be applied, in lieu of
Proposition~\ref{jinhbugvyfctdxrsezaxrdctfvybhuijnkm}), we have
\[
k_A(x,y) = -\frac{d_{N,s}^2}{2}\lim_{\varepsilon\searrow0}\int_{\R^N\setminus\left(B_\varepsilon(x)\cup B_\varepsilon(y)\right)}\frac{(z-x)\cdot A(z)(z-y)}{|z-x|^{N+s+1}|z-y|^{N+s+1}} \,dz.
\]
We use the change of variable~$z:=m+w$. In this way,
$z-x=w-h$ and $z-y=w+h$, hence
\[
k_A(x,y) = -\frac{d_{N,s}^2}{2} \lim_{\varepsilon\searrow 0}
\int_{D_{\varepsilon,h}}\frac{(w-h)\cdot A(m+w)(w+h)}{|w-h|^{N+s+1}|w+h|^{N+s+1}} \,dw,
\]
where 
$$ D_{\varepsilon,h}:=\R^N\setminus\left(B_\varepsilon(h)\cup B_\varepsilon(-h)\right).$$
We point out that 
\begin{equation}\label{cewt87934ytwqhirtSDFGH}
{\mbox{the set~$D_{\varepsilon,h}$ is symmetric with respect to~$0$.}}
\end{equation}

We thus set
\[
G_{M,h}(w):=\frac{w\cdot Mw-h\cdot Mh}{|w-h|^{N+s+1}|w+h|^{N+s+1}}
\]
and we notice that
\begin{equation}\label{cewt87934ytwqhirtSDFGH2}
G_{M,h}(-w)=G_{M,h}(w).
\end{equation}

With this notation, and recalling~\eqref{kjhbgvfcdxcfvghbjnm},
we have that
\begin{equation}\label{ewuty34789tyowi2345678gvchcvdjhewfPP}
\begin{split}
k_A(x,y)&= -\frac{d_{N,s}^2}{2} \lim_{\varepsilon\searrow 0}\Bigg( \;\int_{D_{\varepsilon,h}}
 G_{A_0,h}(w)\,dw\\
 &\; +\int_{D_{\varepsilon,h}} \sum_{j=1}^M\Big(
G_{B_j,h}(w)\cos(\xi_j\cdot (m+w)) + G_{C_j,h}(w)\sin(\xi_j\cdot (m+w))\Big)\,dw\Bigg).\end{split}\end{equation}

Now, since
\[
\cos(\xi_j\cdot(m+w)) = \cos(\xi_j\cdot m)\cos(\xi_j\cdot w) - \sin(\xi_j\cdot m)\sin(\xi_j\cdot w),
\]
we infer that
\[
\begin{split}&
\int_{D_{\varepsilon,h}}G_{B_j,h}(w)\cos(\xi_j\cdot(m+w)) \,dw\\&\quad= \cos(\xi_j\cdot m)\int_{D_{\varepsilon,h}} G_{B_j,h}(w)\cos(\xi_j\cdot w) \,dw- \sin(\xi_j\cdot m) \int_{D_{\varepsilon,h}} G_{B_j,h}(w)\sin(\xi_j\cdot w)  \,dw.
\end{split}
\]
Thus, recalling~\eqref{cewt87934ytwqhirtSDFGH} and~\eqref{cewt87934ytwqhirtSDFGH2}, we conclude that
\[
\int_{D_{\varepsilon,h}} G_{B_j,h}(w)\cos(\xi_j\cdot(m+w))\,dw= \cos(\xi_j\cdot m)\int_{D_{\varepsilon,h}} G_{B_j,h}(w)\cos(\xi_j\cdot w)\, dw.
\]

In a similar manner,
\[
\int_{D_{\varepsilon,h}} G_{C_j,h}(w) \sin(\xi_j\cdot(m+w))\,dw =\sin(\xi_j\cdot m) \int_{D_{\varepsilon,h}} G_{C_j,h}(w)\cos(\xi_j\cdot w) \,dw.
\]
Plugging the last two displays into~\eqref{ewuty34789tyowi2345678gvchcvdjhewfPP}, 
and recalling the notation in~\eqref{kjnhbgvcfdxrctfvgbhjnk}, we thus find that
\[
k_A(x,y)= k_{A_0}(x,y) + \sum_{j=1}^M \cos(\xi_j\cdot m)\mathbb K_{B_j,\xi_j}(h) + \sum_{j=1}^M \sin(\xi_j\cdot m)\mathbb K_{C_j,\xi_j}(h).
\] From this and \eqref{mhdefn}, we obtain the desired result.
\end{proof}

\section{Proof of Theorem~\ref{ESPONENTS}}\label{è+òplkojihugy87ft6r58ed4ws34aszxrdcft}

To establish Theorem~\ref{ESPONENTS}, we rely on an averaging result.
The gist is that the desired bound holds true for constant matrices, e.g. owing to the representation formula given in Corollary~\ref{COINCIDONOCOST},
and then the boundedness condition on a nonconstant matrix allows, in average, to preserve this bound. The regularity assumption~\eqref{mkojihguyftdreswxrdcftvgybhjn} guarantees
that the singular integrals defining~$k_A(x, y)$ are well defined and allows us to pass to the limit in the averaging procedure. The details go as follows.

\begin{lemma}
Let $A\in L^\infty(\R^N,\R^{N\times N})$ satisfy~\eqref{mkojihguyftdreswxrdcftvgybhjn}
for every compact set~$K\subset\R^N$.
Let~$k_A$ be given by~\eqref{CHARKE}.

Then, there exists $C>0$, depending only on $N$, $s$ and $\|A\|_{L^\infty(\R^N, \R^{N\times N})}$, such that, for every~$x_\star\in\R^N$, $h\in\R^N\setminus B_1$ and $L>0$,
\begin{equation}\label{ILle2s}
\left|\frac1{(2L)^N}\int_{x_\star+[-L,L]^N}k_A(x,x+h)\,dx\right|\le\frac{C}{|h|^{N+2s}}\left(1+\frac{|h|}{L}\right).
\end{equation}
\end{lemma}

\begin{proof}
We use the notation
\begin{equation}\label{5.1BIS}
A_L:=\frac1{(2L)^N}\int_{x_\star+[-L,L]^N}A(x) \,dx.
\end{equation}
Given $h\in\R^N\setminus B_1$, we define, for every~$i,j\in\{1,\dots,N\}$,
\[
\R^N\ni w\longmapsto\phi_{i,j}(w):=\frac{w_i(w-h)_j}{|w|^{N+s+1}|w-h|^{N+s+1}}
\]
and consider the inner product for matrices
\[
\langle A,\phi\rangle:=\sum_{i,j=1}^N A_{ij}\phi_{ij}.
\]
By~\eqref{CHARKE}, here used with $y:=x+h$, and changing variable~$w:=z-x$, we obtain
\begin{equation}\label{kernelshifted}
k_A(x,x+h) = -\frac{d_{N,s}^2}{2}\lim_{\varepsilon\searrow0}\int_{\R^N\setminus(B_\varepsilon\cup B_\varepsilon(h))}\langle A(x+w),\phi(w)\rangle\,dw.
\end{equation}

Now, given~$x_\star\in\R^N$, let $Q_L:=x_\star+[-L,L]^N$.  Let $\varepsilon>0$ and define
\[
F_\varepsilon(x):=\int_{\R^N\setminus(B_\varepsilon\cup B_\varepsilon(h))}\langle A(x+w),\phi(w)\rangle\,dw.
\]
We claim that 
\begin{equation}\label{5.2BIS}
{\mbox{$F_\varepsilon$ converges uniformly on $Q_L$ as $\varepsilon\searrow0$.}}\end{equation}
We notice that the integrand is absolutely integrable away from the singularities at~$w=0$ and~$w=h$. We set $\rho:=1/4$. Since $|h|>1$, we have that $B_\rho$ and $B_\rho(h)$ are disjoint. Moreover, set
\[
K:=\left(Q_L+\overline{B_\rho}\right)\cup\left(Q_L+h+\overline{B_\rho}\right)
\]
and we notice that~$K$ is compact. 

We first consider the singularity at $w=0$. For $x\in Q_L$ and $w\in B_\rho$, we have that~$|w-h|\ge\frac34 |h|$ in $B_\rho$, thus
\[
|\phi(w)|\le\frac{C}{|h|^{N+s}|w|^{N+s}}.
\]
Hence, by \eqref{mkojihguyftdreswxrdcftvgybhjn},
\begin{equation}\label{5.2TER}
\int_{B_\rho}\left|\left\langle A(x+w)-A(x),\phi(w)\right\rangle\right| dw\le\frac{C}{|h|^{N+s}}\int_0^\rho\frac{\omega_{A,K}(r)}{r^{1+s}} dr <+\infty
\end{equation}
uniformly with respect to $x\in Q_L$. 

Also, for every $i,j\in\{1,\dots, N\}$, we set
\[
q_j(w):=\frac{(w-h)_j}{|w-h|^{N+s+1}}.
\]
Then,
\[
\phi_{ij}(w)=\frac{w_i}{|w|^{N+s+1}}q_j(0) + \frac{w_i}{|w|^{N+s+1}}(q_j(w)-q_j(0)).
\]
Now, since $w_i|w|^{-N-s-1}$ is odd, the first term has vanishing integral on every annulus centered at zero. Moreover, since
\[
|q_j(w)-q_j(0)|\le\frac{C|w|}{|h|^{N+s+1}},\quad {\mbox{ for all }}
w\in B_\rho,
\]
we get
\[
\left|\frac{w_i}{|w|^{N+s+1}}(q_j(w)-q_j(0))\right|\le\frac{C}{|h|^{N+s+1}|w|^{N+s-1}},
\]
which is integrable near zero.
This entails that
$$ \int_{B_\rho}\left|\left\langle A(x),\phi(w)\right\rangle\right|\,dx
<\infty.
$$
From this and~\eqref{5.2TER} we obtain the uniform integrability near the singularity~$w=0$.

The singularity at $w=h$ is treated in exactly the same way, after setting $t:=w-h$,  and then subctracting $A(x+h)$. 
These estimates imply that, for $0<\varepsilon'<\varepsilon<\rho$,
\[
\sup_{x\in Q_L} |F_{\varepsilon'}(x)-F_\varepsilon(x)|\le\frac{C}{|h|^{N+s}}\int_0^\varepsilon\frac{\omega_{A,K}(r)}{r^{1+s}} dr + \frac{C\|A\|_{L^\infty(\R^N, \R^{N\times N})}}{|h|^{N+s+1}}\varepsilon^{1-s}.
\]
By \eqref{mkojihguyftdreswxrdcftvgybhjn},
the right hand side tends to $0$ as $\varepsilon\searrow 0$.
Therefore, $F_\varepsilon$ is a uniformly Cauchy sequence on~$Q_L$.
This entails the claim in~\eqref{5.2BIS}.

Accordingly, in lieu of~\eqref{5.2BIS}, the limit in~\eqref{kernelshifted} is uniform for~$x\in Q_L$,
hence we can  average with respect to $x$ and then pass to the limit as $\varepsilon\searrow0$.

We set
\[
\overline A_L(w):=\frac1{(2L)^N}\int_{x_\star+[-L,L]^N}A(w+x) \,dx.
\]
By Fubini's theorem, we have
\[
\frac1{(2L)^N}\int_{x_\star+[-L,L]^N}k_A(x,x+h)\,dx = -\frac{d_{N,s}^2}{2}\lim_{\varepsilon\searrow0}\int_{\R^N\setminus(B_\varepsilon\cup B_\varepsilon(h))}\langle\overline A_L(w),\phi(w)\rangle\, dw.
\]
Also, replacing $A$ with the constant matrix $A_L$ in~\eqref{5.1BIS},
\[
\frac1{(2L)^N}\int_{x_\star+[-L,L]^N} k_{  A_L}(x,x+h) \,dx = -\frac{d_{N,s}^2}{2}\lim_{\varepsilon\searrow0}\int_{\R^N\setminus(B_\varepsilon\cup B_\varepsilon(h))}\langle A_L,\phi(w)\rangle\, dw.
\]
Therefore, we obtain that
\begin{equation}\label{differenceaverage}
\begin{split}
&\left|\frac1{(2L)^N}\int_{x_\star+[-L,L]^N}k_A(x,x+h)\,dx - \frac1{(2L)^N}\int_{x_\star+[-L,L]^N} k_{A_L}(x,x+h)\, dx\right|\\
&\quad\le C \limsup_{\varepsilon\searrow0}\left|\;\int_{\R^N\setminus (B_\varepsilon\cup B_\varepsilon(h))}\left\langle\overline A_L(w)- A_L,\phi(w)
\right\rangle\, dw \right|.
\end{split}
\end{equation}

Furthermore, using the symbol ``$\triangle$'' to denote the symmetric difference of two sets and looking at the measure of the symmetric difference for shifted cubes, we have that, for all~$w$, $\widetilde w\in\R^N$,
\begin{equation}\label{SQsUHoBOS}\begin{split}
|\overline{A}_L(w)-\overline{A}_L(\widetilde w)|&=\frac1{(2L)^N}
\left| \;\int_{w+{x_\star+[-L,L]^N}}A(y)\,dy-\int_{\widetilde w+{x_\star+[-L,L]^N}}A(y)\,dy\right|\\&
\le\frac1{(2L)^N}\int_{(w+{x_\star+[-L,L]^N})\triangle(\widetilde w+{x_\star+[-L,L]^N})}|A(y)|\,dy\\&\le
\frac{C\,\|A\|_{L^\infty(\R^N, \R^{N\times N})}\big|(w+{x_\star+[-L,L]^N})\triangle(\widetilde w+{x_\star+[-L,L]^N})\big|}{(2L)^N}\\&\le
\frac{C\,|w-\widetilde w|}{L},
\end{split}
\end{equation}
up to renaming~$C$ at each step of the calculation.

We use this with $\widetilde w:=0$, observing that $A_L=\overline A_L(0)$, and conclude that
\begin{equation}\label{PKJSMDjpdmewyj054nbk3-03431b}
\begin{split}&
\left|\;\int_{B_{|h|/2}} \langle \big(\overline{A}_L(w)- {A}_L\big),\phi(w)\rangle\,dw\right|\le\frac{
C}L\int_{B_{|h|/2}}|w|\, |\phi(w)|\,dw
\\&\qquad\le \frac{C}{L} \int_{B_{|h|/2}}\frac{dw}{|w|^{N+s-1}|w-h|^{N+s}}\le
\frac{C}{L|h|^{N+s}} \int_{B_{|h|/2}}\frac{dw}{|w|^{N+s-1}}\le\frac{C}{L|h|^{N+2s-1}}.
\end{split}
\end{equation}
On a similar note, using~\eqref{SQsUHoBOS} with
$\widetilde w:=h$, we get\begin{equation}\label{PKJSMDjpdmewyj054nbk3-03431}
\begin{split}&
\left|\;\int_{B_{|h|/2}(h)} \langle \big(\overline{A}_L(w)-\overline{A}_L(h)\big),\phi(w)\rangle\,dw\right|\le\frac{
C}L\int_{B_{|h|/2}(h)} |w-h|\, |\phi(w)|\,dw\\&\qquad\le\frac{ C}L \int_{B_{|h|/2}(h)}\frac{\,dw}{|w|^{N+s}|w-h|^{N+s-1}}\le
\frac{C}{L|h|^{N+s}} \int_{B_{|h|/2}(h)}\frac{dw}{|w-h|^{N+s-1}}\le\frac{C}{L|h|^{N+2s-1}}.
\end{split}
\end{equation}

We now estimate the remaining constant contribution in $B_{|h|/2}(h)$. By using the change of variables $t:=\pm(w-h) $, we obtain that
\begin{eqnarray*}
&&2\int_{B_{|h|/2}(h)} \phi_{ij}(w)\,dw=
\int_{B_{|h|/2}} \phi_{ij}(h+t)\,dt+\int_{B_{|h|/2}} \phi_{ij}(h-t)\,dt\\&&\qquad=
\int_{B_{|h|/2}}\left(\frac{(h+t)_i \,t_j}{|h+t|^{N+s+1}|t|^{N+s+1}}
-\frac{(h-t)_i \,t_j}{|h-t|^{N+s+1}|t|^{N+s+1}}\right)\,dt\\&&\qquad=
\int_{B_{|h|/2}}\left(\frac{(h+t)_i \,t_j}{|h+t|^{N+s+1}|t|^{N+s+1}}
-\frac{(h+t)_i \,t_j}{|h-t|^{N+s+1}|t|^{N+s+1}}\right)\,dt\\
&&\qquad\qquad+
\int_{B_{|h|/2}}\left(\frac{(h+t)_i \,t_j}{|h-t|^{N+s+1}|t|^{N+s+1}}
-\frac{(h-t)_i \,t_j}{|h-t|^{N+s+1}|t|^{N+s+1}}\right)\,dt\\&&\qquad=
\int_{B_{|h|/2}} \frac{(h+t)_i\,t_j}{|t|^{N+s+1}}\left(\frac{1}{|h+t|^{N+s+1}}
-\frac{1}{|h-t|^{N+s+1}}\right)\,dt+2
\int_{B_{|h|/2}}\frac{t_i\,t_j}{|t|^{N+s+1}|h-t|^{N+s+1}}\,dt\\
&&\qquad=\int_{B_{|h|/2}}\left(\frac{(h+t)_i t_j}{|h+t|^{N+s+1}|t|^{N+s+1}} - \frac{(h-t)_i t_j}{|h-t|^{N+s+1}|t|^{N+s+1}}\right) dt.
\end{eqnarray*}
Thus, since, for all~$t\in B_{|h|/2}$,
\begin{eqnarray*}&&\left| |h+t|^{-(N+s+1)}-|h-t|^{-(N+s+1)}\right|=
|h|^{-(N+s+1)}\left| \left|\frac{h}{|h|}+\frac{t}{|h|}\right|^{-(N+s+1)}-\left|\frac{h}{|h|}-\frac{t}{|h|}\right|^{-(N+s+1)}\right|\\&&\qquad\le C|h|^{-(N+s+2)}|t|,
\end{eqnarray*}
we deduce that
\[
\left|\;\int_{B_{|h|/2}(h)} \phi_{ij}(w)\,dw\right|\le C\int_{B_{|h|/2}} \frac{(|h|+|t|)\,dt}{|t|^{N+s-1}|h|^{N+s+2}} + C\int_{B_{|h|/2}}\frac{dt}{|t|^{N+s-1}|h-t|^{N+s+1}}\le\frac{C}{|h|^{N+2s}}.
\]
On this account,
\begin{equation*}
\begin{split}&
\left|\;\int_{B_{|h|/2}(h)} \langle \big(\overline{A}_L(h)-{A}_L\big),\phi(w)\rangle\,dw\right|\le
\sum_{i,j=1}^N\left|\big(\overline{A}_L(h)-{A}_L\big)_{ij}\int_{B_{|h|/2}(h)} \phi_{ij}(w)\,dw\right|\\&\qquad\le
C\|A\|_{L^\infty(\R^N, \R^{N\times N})}\sum_{i,j=1}^N\left|\;\int_{B_{|h|/2}(h)} \phi_{ij}(w)\,dw\right|\le\frac{C}{|h|^{N+2s}}.
\end{split}
\end{equation*}
Combining the latter inequality with \eqref{PKJSMDjpdmewyj054nbk3-03431}, we arrive at
\[
\left|\;
\int_{B_{|h|/2}(h)}
\left\langle
\overline A_L(w)- A_L,\phi(w)
\right\rangle\,dw
\right|
\le
\frac{C}{|h|^{N+2s}}
+
\frac{C}{L|h|^{N+2s-1}}.
\]
Finally,
\[
\begin{split}
&\left|\;\int_{\R^N\setminus(B_{|h|/2}\cup B_{|h|/2}(h))}\left\langle\overline A_L(w)-A_L,\phi(w)\right\rangle\,dw\right|\\
&\quad\le C\|A\|_{L^\infty(\R^N, \R^{N\times N})}\int_{\R^N\setminus(B_{|h|/2}\cup B_{|h|/2}(h))}\frac{dw}{|w|^{N+s}|w-h|^{N+s}}\\
&\quad\le\frac{C}{|h|^{N+s}}\int_{\R^N\setminus B_{|h|/2}}\frac{dw}{|w|^{N+s}}\le\frac{C}{|h|^{N+2s}}.
\end{split}
\]
As a consequence, combining this with~\eqref{PKJSMDjpdmewyj054nbk3-03431b} and recalling \eqref{differenceaverage}, we get
\begin{equation}\label{OtJLhy8uVEfmECa34oBs}
\begin{split}
&\left|\frac1{(2L)^N}\int_{x_\star+[-L,L]^N}k_A(x,x+h) \,dx-\frac1{(2L)^N}\int_{x_\star+[-L,L]^N}k_{  A_L}(x,x+h) \,dx\right|\\
&\qquad\le\frac{C}{|h|^{N+2s}} + \frac{C}{L|h|^{N+2s-1}}.
\end{split}
\end{equation}

Now, in virtue of Corollary \ref{COINCIDONOCOST},
\[
\left|k_{  A_L}(x,x+h)\right| = \left|
\frac{c_{N,s}}{4s|h|^{N+2s}}\left((N+2s)\frac{h\cdot A_Lh}{|h|^2} - \tr(A_L)\right)\right| \le\frac{C}{|h|^{N+2s}},
\]
and accordingly
\[
\left|\frac1{(2L)^N}\int_{x_\star+[-L,L]^N} k_{ A_L}(x,x+h) dx\right|\le\frac{C}{|h|^{N+2s}}.
\]
Combining this with~\eqref{OtJLhy8uVEfmECa34oBs}, we conclude that
\[
\left|\frac1{(2L)^N}\int_{x_\star+[-L,L]^N}k_A(x,x+h) dx\right|\le\frac{C}{|h|^{N+2s}} + \frac{C}{L|h|^{N+2s-1}}=\frac{C}{|h|^{N+2s}}\left(1+\frac{|h|}{L}\right)
\]
This proves \eqref{ILle2s}, as desired.
\end{proof}

We are now in the position to prove Theorem~\ref{ESPONENTS}.
\begin{proof}[Proof of Theorem~\ref{ESPONENTS}]
The desired result follows by
sending $L\to+\infty$ in~\eqref{ILle2s}.
\end{proof}

\section{Proof of Theorem \ref{LOCALETRACE}}\label{èòkmnjibhuyctxdsza<azsxdcfv}

We provide the following preliminary statement.

\begin{proposition}\label{FREEZING}
Let $A\in L^\infty(\R^N,\R^{N\times N})$ be symmetric and satisfy~\eqref{mkojihguyftdreswxrdcftvgybhjn} for every compact set~$K\subset\R^N$. Let $k_A$ be given by~\eqref{CHARKE}.

Then, for every compact set $K\subset\R^N$,
\[
\lim_{r\searrow0}\sup_{\substack{x\in K\\e\in\mathbb S^{N-1}}}\left| r^{N+2s}k_A(x,x+re)-\frac{c_{N,s}}{4s}\left((N+2s)e\cdot A(x)e-\tr(A(x)\right)\right|=0.
\]
\end{proposition}
\begin{proof}
Let $K\subset\R^N$ be compact and set~$
K':=K+\overline{B_1}$.
Let $x\in K$, $e\in\mathbb S^{N-1}$ and $r\in(0,1)$. By~\eqref{CHARKE}, using the change of variables $w:=(z-x)/r$, we get
\[
r^{N+2s}k_A(x,x+re) = -\frac{d_{N,s}^2}{2}\lim_{\varepsilon\searrow0}\int_{\R^N\setminus\left(B_{\varepsilon/r}\cup B_{\varepsilon/r}(e)\right)}\frac{w\cdot A(x+rw)(w-e)}{|w|^{N+s+1}|w-e|^{N+s+1}} \,dw.
\]

Moreover, we point out that Corollary \ref{COINCIDONOCOST} can be applied
to the matrix $A(x)$ for every fixed~$x$. In this way, we have
\[
-\frac{d_{N,s}^2}{2}\lim_{\varepsilon\searrow0}\int_{\R^N\setminus\left(B_{\varepsilon/r}\cup B_{\varepsilon/r}(e)\right)}\frac{w\cdot A(x)(w-e)}{|w|^{N+s+1}|w-e|^{N+s+1}} \,dw = \frac{c_{N,s}}{4s}\Big((N+2s)e\cdot A(x)e-\tr(A(x)\Big).
\]
Thus, the desired result will follow once we prove that
\[
\lim_{r\searrow0}\sup_{\substack{x\in K\\ e\in\mathbb S^{N-1}}}\left|\lim_{\varepsilon\searrow0}\int_{\R^N\setminus\left(B_{\varepsilon/r}\cup B_{\varepsilon/r}(e)\right)}\frac{w\cdot\big(A(x+rw)-A(x)\big)(w-e)}{|w|^{N+s+1}|w-e|^{N+s+1}} \,dw\right|=0.
\]
To this end, we split the above integral into a neighborhood of $0$, a neighborhood of $e$, and the complement of them.

For $w\in B_{1/4}$ we have $|w-e|\ge|e|-|w|\ge 3/4$. Moreover,
since~$x\in K'$ and~$x+rw\in K'$, for $x\in K$, $r\in(0,1)$ and $w\in B_{1/4}$, we have that
\[
|A(x+rw)-A(x)|\le\omega_{A,K'}(r|w|).
\]
Combining these pieces of information, we get
\[
\begin{split}
&\left|\;\int_{B_{1/4}}\frac{w\cdot\big(A(x+rw)-A(x)\big)(w-e)}{|w|^{N+s+1}|w-e|^{N+s+1}} \,dw\right|\le \left(\frac43\right)^{N+s}\int_{B_{1/4}}\frac{\omega_{A,K'}(r|w|)}{|w|^{N+s}} \,dw\\
&\qquad= \left(\frac43\right)^{N+s} \omega_{N-1}\int_0^{1/4}\frac{\omega_{A,K'}(rt)}{t^{1+s}}\,dt = \left(\frac43\right)^{N+s} \omega_{N-1} r^s\int_0^{r/4}\frac{\omega_{A,K'}(\tau)}{\tau^{1+s}} \,d\tau.
\end{split}
\]
Hence, thanks to assumption \eqref{mkojihguyftdreswxrdcftvgybhjn}, we infer that
\begin{eqnarray*}&&
\lim_{r\searrow0}\sup_{\substack{x\in K\\e\in\mathbb S^{N-1}}}
\left|\;\int_{B_{1/4}}\frac{w\cdot\big(A(x+rw)-A(x)\big)(w-e)}{|w|^{N+s+1}|w-e|^{N+s+1}} \,dw\right|
\le\lim_{r\searrow0} r^s\int_0^{r/4}\frac{\omega_{A,K'}(\tau)}{\tau^{1+s}} \,d\tau=0.
\end{eqnarray*}

We now consider $B_{1/4}(e)$. Writing $u:=w-e$, we have $|u|<1/4$ and 
\[
A(x+r(e+u))-A(x) = A(x+r(e+u))-A(x+re) + A(x+re)-A(x).
\]
The difference $A(x+r(e+u))-A(x+re)$ gives the same estimate obtained above
with $u$ in place of $w$.

To deal with the contribution coming from~$A(x+re)-A(x)$,
we observe that, by odd symmetry,
\[
\int_{B_{1/4}\setminus B_{\varepsilon/r}}\frac{e\cdot\big(A(x+re)-A(x)\big)u}{|u|^{N+s+1}} \,du=0.
\]
Hence, 
\begin{equation}\label{star1}\begin{split}&
\lim_{\varepsilon\searrow0}\int_{B_{1/4}\setminus B_{\varepsilon/r}}\frac{(e+u)\cdot\big(A(x+re)-A(x)\big)u}{|e+u|^{N+s+1}|u|^{N+s+1}} \,du\\&\qquad = \int_{B_{1/4}}\left(\frac{e+u}{|e+u|^{N+s+1}}-e\right)\cdot\frac{(A(x+re)-A(x))u}{|u|^{N+s+1}}\,du.
\end{split}\end{equation}

Now,  since $|e|=1$, for every $u\in B_{1/4}$ and $t\in[0,1]$ we have $|e+tu|\ge |e|-t|u|\ge 3/4$.
Therefore, applying the mean value theorem to the map $\xi\mapsto\xi/|\xi|^{N+s+1}$ in $[e, e+u]$, we obtain that
\[
\left|\frac{e+u}{|e+u|^{N+s+1}}-e\right|\le C|u|,
\]
where $C>0$ depends only on $N$ and $s$.  Accordingly, we have
\[
\begin{split}
&\left|\;\int_{B_{1/4}}\left(\frac{e+u}{|e+u|^{N+s+1}}-e\right)\cdot\frac{(A(x+re)-A(x))u}{|u|^{N+s+1}} \,du\right|\le C|A(x+re)-A(x)|\int_{B_{1/4}}\frac{|u|^2}{|u|^{N+s+1}} \,du\\
&\qquad=C|A(x+re)-A(x)|\int_{B_{1/4}}\frac{du}{|u|^{N+s-1}}.
\end{split}
\]
Recalling \eqref{omegaAkdefn}, we have that $|A(x+re)-A(x)|\le\omega_{A,K'}(r)$. Hence, up to renaming~$C>0$,
\[
\left|\,\int_{B_{1/4}}\left(\frac{e+u}{|e+u|^{N+s+1}}-e\right)\cdot\frac{(A(x+re)-A(x))u}{|u|^{N+s+1}} du\right|\le C\omega_{A,K'}(r).
\]
Accordingly, recalling~\eqref{star1},
we conclude that
\[
\lim_{r\searrow0}\sup_{\substack{x\in K\\ e\in\mathbb S^{N-1}}}\left|\;\int_{B_{1/4}}\left(\frac{e+u}{|e+u|^{N+s+1}}-e\right)\cdot\frac{(A(x+re)-A(x))u}{|u|^{N+s+1}}\,du\right|=0.
\]

It only remains to estimate the contribution away from $0$ and $e$.  To this end, let $R>2$. If~$rR\le1$, then, for $w\in B_R\setminus (B_{1/4}\cup B_{1/4}(e))$,  we have $x+rw\in K'$ and
\[
|A(x+rw)-A(x)|\le\omega_{A,K'}(r|w|)\le\omega_{A,K'}(rR).
\]
Moreover, $|w|\ge 1/4$ and $|w-e|\ge 1/4$.  Therefore,
\[
\begin{split}
&\left|\;\int_{B_R\setminus(B_{1/4}\cup B_{1/4}(e))}\frac{w\cdot\big(A(x+rw)-A(x)\big)(w-e)}{|w|^{N+s+1}|w-e|^{N+s+1}}\, dw\right|\\
&\qquad\le\omega_{A,K'}(rR)\int_{B_R\setminus(B_{1/4}\cup B_{1/4}(e))}\frac{dw}{|w|^{N+s}|w-e|^{N+s}}\le C_R\,\omega_{A,K'}(rR),
\end{split}
\]
where $C_R>0$ is a constant depending only on $N$, $s$ and $R$. In particular, for every $R>2$, we get
\[
\lim_{r\searrow0}\sup_{\substack{x\in K\\e\in\mathbb S^{N-1}}}\left|\;\int_{B_R\setminus(B_{1/4}\cup B_{1/4}(e))}\frac{w\cdot\big(A(x+rw)-A(x)\big)(w-e)}{|w|^{N+s+1}|w-e|^{N+s+1}} \,dw\right|=0.
\]

Finally, we estimate the contribution from $\R^N\setminus B_R$.
Since $R>2$ and $e\in\mathbb S^{N-1}$, for every~$|w|>R$ we have $|w-e|\ge |w|-|e|\ge |w|-1\ge |w|/2$.  Moreover, $$|A(x+rw)-A(x)|\le 2\|A\|_{L^\infty(\R^N, \R^{N\times N})}.$$ Therefore, we have
\[
\begin{split}
&
\left|\;\int_{\R^N\setminus B_R}\frac{w\cdot\big(A(x+rw)-A(x)\big)(w-e)}{|w|^{N+s+1}|w-e|^{N+s+1}} \,dw\right|\\&\qquad
\le C\|A\|_{L^\infty(\R^N, \R^{N\times N})}\int_{\R^N\setminus B_R}\frac{dw}{|w|^{2N+2s}} = \frac{C\|A\|_{L^\infty(\R^N, \R^{N\times N})}}{R^{N+2s}},
\end{split}
\]
where $C>0$ depends only on $N$ and $s$. Consequently, we get
\[
\lim_{R\to+\infty}\sup_{\substack{x\in K,\ e\in\mathbb S^{N-1}\\ r\in(0,1)}}\left|\,\int_{\R^N\setminus B_R}\frac{w\cdot\big(A(x+rw)-A(x)\big)(w-e)}{|w|^{N+s+1}|w-e|^{N+s+1}} dw\right|=0.
\]
Combining all the previous estimates, we get the desired result.
\end{proof}

We are now in the position to prove Theorem \ref{LOCALETRACE}.

\begin{proof}[Proof of Theorem \ref{LOCALETRACE}]
Assume that $(i)$ holds and set
\[
\delta_K:=\inf_{x\in K}\left((N+2s)\lambda_{\min}(A(x))-\tr(A(x))\right)>0.
\]
By Proposition~\ref{FREEZING}, there exists $r_K>0$ such that,
for every $x\in K$, $e\in\mathbb S^{N-1}$ and $0<r\le r_K$,
\begin{equation}\label{mkjnhbgvfcdxszea<ezsrdtfyguhijo}
\left|r^{N+2s}k_A(x,x+re) - \frac{c_{N,s}}{4s}\left((N+2s)e\cdot A(x)e-\tr(A(x))\right)\right|\le\frac{c_{N,s}\delta_K}{8s}
.\end{equation}
Moreover, we know that $e\cdot A(x)e\ge\lambda_{\min}(A(x))$. Thus,
\[
(N+2s)e\cdot A(x)e-\tr(A(x))\ge\delta_K.
\]
Therefore,
\[
r^{N+2s}k_A(x,x+re)\ge\frac{c_{N,s}}{4s}\delta_K - \frac{c_{N,s}}{8s}\delta_K = \frac{c_{N,s}\delta_K}{8s}.
\]
Hence,  setting $\lambda_K:=(c_{N,s}\delta_K)/(8s)$, we obtain
\begin{equation}\label{chjwqkr783tyif345678}
k_A(x,x+re)\ge\frac{\lambda_K}{r^{N+2s}}.
\end{equation}

Moreover, since $A$ is bounded, we infer that, for every $x\in\R^N$ and $e\in\mathbb S^{N-1}$,
\[
\left|(N+2s)e\cdot A(x)e-\tr(A(x))\right|\le (N^2+2Ns +N) \|A\|_{L^\infty(\R^N, \R^{N\times N})} =:C_A.
\]
From this and \eqref{mkjnhbgvfcdxszea<ezsrdtfyguhijo}, we get
\[
r^{N+2s}|k_A(x,x+re)|\le\frac{c_{N,s}}{4s}\left|(N+2s)e\cdot A(x)e-\tr(A(x))\right| + \frac{c_{N,s}\delta_K}{8s}\le
\frac{c_{N,s}C_A}{4s} + \frac{c_{N,s}\delta_K}{8s}.
\]
Thus, defining
\[
\Lambda_K:=\frac{c_{N,s}C_A}{4s} + \frac{c_{N,s}\delta_K}{8s},
\]
we have that, for every $x\in K$, $e\in\mathbb S^{N-1}$ and $0<r\le r_K$,
\begin{equation}\label{chjwqkr783tyif3456782}
k_A(x,x+re)\le\frac{\Lambda_K}{r^{N+2s}}
\end{equation}
Finally, given $x\in K$ and $y\in\R^N$ with~$0<|x-y|\le r_K$, we take
\[
r:=|x-y|\quad\mbox{ and }\quad e:=\frac{y-x}{|y-x|},
\] and we obtain~$(ii)$ from~\eqref{chjwqkr783tyif345678}
and~\eqref{chjwqkr783tyif3456782}.

We now assume that $(ii)$ holds. Let $x\in K$ and $e\in\mathbb S^{N-1}$. For every $0<r\le r_K$, we have
\[
r^{N+2s}k_A(x,x+re)\ge\lambda_K.
\]
Letting $r\searrow0$ and applying Proposition~\ref{FREEZING}, we obtain that
\[
\frac{c_{N,s}}{4s}\left((N+2s)e\cdot A(x)e-\tr(A(x)\right)\ge\lambda_K.
\]
Since the latter inequality holds for every $e\in\mathbb S^{N-1}$, we get
\[
\frac{c_{N,s}}{4s}\left((N+2s)\lambda_{\min}(A(x))-\tr(A(x)\right)\ge\lambda_K.
\]
Hence, 
\[
\inf_{x\in K}\left((N+2s)\lambda_{\min}(A(x))-\tr(A(x))\right)\ge\frac{4s\lambda_K}{c_{N,s}}>0.
\]
This proves $(i)$ and cocludes the proof.
\end{proof}

We stress that Theorem~\ref{LOCALETRACE} is local with respect to the point $x\in K$. This is consistent with the fact that \eqref{mkojihguyftdreswxrdcftvgybhjn} holds for every compact set $K\subset\R^N$.
If the modulus of continuity satisfies~\eqref{mkojihguyftdreswxrdcftvgybhjn} uniformly on the whole space, then the argument in Proposition \ref{FREEZING} is uniform with respect to $x\in\R^N$. As a consequence, both the distance from the diagonal and the constants in the estimates from above and below can be chosen independently of~$x\in K$. 

\section{Proof of Theorem \ref{thmaltroopenproblem}}\label{poi98u76t5r4e3jhgfrde}
To start with, we identify the mechanism responsible for the failure of the truncation inequality. More precisely, we prove that if a symmetric kernel is strictly negative on the product of two disjoint open sets of $\R^N$, then we can construct a smooth and compactly supported function whose energy increases under truncation. The precise statement reads as follows:

\begin{lemma}\label{lemmaotheropenproblem}
Let $k_A \in L^1_{\mathrm{loc}}(\R^{2N}\setminus\Delta)$ be a symmetric kernel.

For every $u,v\in C_c^\infty(\R^N)$, let 
\[
B_A(u,v):= \iint_{\R^{2N}}(u(x)-u(y))(v(x)-v(y))k_A(x,y) \,dx \,dy.
\]
Assume that
\begin{equation}\label{starstar}
B_A(w,w)>0\quad\mbox{ for every }w\in C_c^\infty(\R^N)\setminus\{0\}.
\end{equation}
In addition, suppose that there exist two nonempty and disjoint open sets $U,V\subset\R^N$ such that
\begin{equation}\label{starstar2}
k_A(x,y)<0 \quad\mbox{ for a.e. }(x,y)\in U\times V.
\end{equation}
Let $T(t)$ be as in \eqref{TAGLIODEFN}.

Then, there exists $u\in C_c^\infty(\R^N)$ such that $T(u)\in C_c^\infty(\R^N)$, and 
\[
B_A(T(u),T(u))>B_A(u,u).
\]
\end{lemma}

\begin{proof}
Let $\varphi\in C_c^\infty(U,[0,1])$ and $\psi\in C_c^\infty(V,[0,+\infty))$. Since $\supp(\varphi)$ and $\supp(\psi)$ are disjoint, we have $(\varphi(x)-\varphi(y))(\psi(x)-\psi(y))=0$ unless either~$x\in\supp(\varphi)$ and $y\in\supp(\psi)$, or~$x\in\supp(\psi)$ and $y\in\supp(\varphi)$. Thus, since~$k_A$ is symmetric,
\[
B_A(\varphi,\psi)=-2\iint_{U\times V}\varphi(x)\psi(y)k_A(x,y) \,dy\, dx.
\]
By~\eqref{starstar2} it follows that $B_A(\varphi,\psi)>0$. 

Moreover, thanks to~\eqref{starstar}, we can define
\begin{equation}\label{cuorejr43otgbwu09876543}
\tau:=\frac{B_A(\varphi,\psi)}{B_A(\psi,\psi)}\in(0,+\infty).
\end{equation}
We also set~$u:=\varphi-\tau\psi$.
Since $\varphi$ and $\psi$ are nonnegative and have disjoint supports, we have
\[
u(x)=\begin{cases}
\varphi(x), & \mbox{ if } x\in\supp\varphi,\\[3pt]
-\tau\psi(x), & \mbox{ if } x\in\supp\psi,\\[3pt]
0, & \mbox{ if } x\notin\supp\varphi\cup\supp\psi.
\end{cases}
\]
Hence, recalling~\eqref{TAGLIODEFN}, we infer that
\[
T(u)=\varphi\quad\mbox{ in }\R^N,
\] and accordingly
\begin{equation}\label{cuorejr43otgbwu098765432}
B_A(T(u),T(u)) = B_A(\varphi,\varphi).
\end{equation}

Furthermore, by using the symmetry and the bilinearity of $B_A$, we get
\[
B_A(u,u) = B_A(\varphi,\varphi) -2\tau B_A(\varphi,\psi)+\tau^2  B_A(\psi,\psi).
\]
Thus, by~\eqref{cuorejr43otgbwu09876543} we obtain that
\[
B_A(u,u) = B_A(\varphi,\varphi)-\frac{B_A(\varphi, \psi)^2}{B_A(\psi, \psi)}.
\]
Combining this with~\eqref{cuorejr43otgbwu098765432},
we conclude that
\[
B_A(T(u),T(u))-B_A(u,u) = \frac{B_A(\varphi, \psi)^2}{B_A(\psi, \psi)}>0.
\]
This concludes the proof.
\end{proof}

We now combine Lemma \ref{lemmaotheropenproblem} with Proposition \ref{FREEZING} to obtain Theorem \ref{thmaltroopenproblem}.

\begin{proof}[Proof of Theorem \ref{thmaltroopenproblem}]
Since $A(x_0)$ is symmetric, there exists $e_0\in\mathbb S^{N-1}$ such that
\[
A(x_0)e_0 = \lambda_{\min}(A(x_0))e_0.
\]
Hence, thanks to the assumption \eqref{x0conditionbhgvfctdxrsdtfygh}, we obtain that
\[
(N+2s)e_0\cdot A(x_0)e_0-\tr(A(x_0))<0.
\]
We choose $\eta>0$ such that
\[
(N+2s)e_0\cdot A(x_0)e_0-\tr(A(x_0))<-3\eta.
\]
Since $A\in C(\R^N, \R^{N\times N})$, there exist $\rho>0$ and $\delta>0$ such that, if $x\in\overline{B_\rho(x_0)}$ and $e\in\mathbb S^{N-1}$ with $ |e-e_0|<\delta$, then
\begin{equation}\label{lpkjohiguytdreswaesrdtfgyh}
(N+2s)e\cdot A(x)e-\tr(A(x))\le -2\eta.
\end{equation}

We now use Proposition \ref{FREEZING} with $K:=\overline{B_\rho(x_0)}$.  Accordingly, there exists $r_0>0$ such that, for every~$x\in K$, $e\in\mathbb S^{N-1}$ and~$0<r\le r_0$, we have
\[
\left|r^{N+2s}k_A(x,x+re)-\frac{c_{N,s}}{4s}\left((N+2s)e\cdot A(x)e-\tr(A(x))\right)\right| < \frac{c_{N,s}\eta}{4s}.
\]
Now, if in addition $|e-e_0|<\delta$, then inequality \eqref{lpkjohiguytdreswaesrdtfgyh} gives
\[
r^{N+2s}k_A(x,x+re)\le -\frac{c_{N,s}}{2s}\eta + \frac{c_{N,s}\eta}{4s} = -\frac{c_{N,s}\eta}{4s}<0.
\]
Therefore, for every $x\in\overline{B_\rho(x_0)}$, $e\in\mathbb S^{N-1}$ with $ |e-e_0|<\delta$, and $0<r\le r_0$, we have that
\begin{equation}\label{mkjnhbgvfcdxszxrdcfvgbhnjmk}
k_A(x,x+re)<0.
\end{equation}

We now want to apply Lemma \ref{lemmaotheropenproblem}. To this end,
we construct two disjoint open sets of~$\R^N$ in which the kernel is
negative. For this, let
\begin{equation*} f(z):=\frac{1+z\cdot e_0}{|e_0+z|}, \quad{\mbox{ for all }}z\in B_{\frac12} ,\end{equation*}
and
$$ C_0:=\sup_{z\in B_{\frac12}}|D^2 f(z)|. $$
Let also
\[
\alpha\in\left(0,\min\left\{\frac14,\frac{\delta}{{2}\sqrt{C_0}}\right\}\right).
\]
and~$r>0$ be such that
\[
r<\rho\quad\mbox{ and }\quad (1+2\alpha)r<r_0.
\]
Moreover, we define
\[
U:=B_{\alpha r}(x_0)\quad\mbox{ and }\quad V:=B_{\alpha r}(x_0+re_0).
\]
Since $\alpha<1/2$ and
\[
|x_0-(x_0+re_0)|=r>2\alpha r,
\]
it follows that the sets $U$ and $V$ are disjoint. 
Moreover, $U\subset B_\rho(x_0)$.

Let $x\in U$ and $y\in V$. Then,
$$ |x-y|\le |x-x_0|+|y-x_0-re_0|+r\le (1+2\alpha)r<r_0.$$
Moreover, we can write
\[
y-x=r e_0+\xi \quad\mbox{ with }|\xi|<2\alpha r.
\]
Therefore, setting $ e:=(y-x)/|y-x|$, we obtain that $e\in\mathbb S^{N-1}$ and 
\begin{eqnarray*}
&&\left|e-e_0\right|^2 = \left|\frac{r e_0+\xi }{|r e_0+\xi |}- e_0 \right|^2
=2\left(1-\frac{r +\xi\cdot e_0}{|re_0+\xi|}\right)= 2\left(
1-\frac{ 1 +\frac{\xi}{r}\cdot e_0 }{\left|e_0+\frac{\xi}r\right| }\right)\\
&&\qquad = 2\left(
1-f\left(\frac{ \xi}{r}\right)\right)\le \frac{C_0|\xi|^2}{r^2}\le 4\alpha^2 C_0<\delta^2,
\end{eqnarray*}
that is, $|e-e_0|<\delta$.

Gathering these observations, 
we find that we are in the position of employing~\eqref{mkjnhbgvfcdxszxrdcfvgbhnjmk}, obtaining that
\begin{eqnarray*}
k_A(x,y)=k_A\left(x, x+ |y-x| e\right)<0.
\end{eqnarray*}

Furthermore,
\eqref{MATHCALBDEFN} and identity~\eqref{mnbhgvcfdxszxdfcgvhb} give that,
for $u,v\in C_c^\infty(\R^N)$,
\[
\mathcal B_A(u,v) = \iint_{\R^{2N}}(u(x)-u(y))(v(x)-v(y))k_A(x,y)\, dx\, dy
\]
and assumption~\eqref{uniformlyellipticdefn} entails that,
for every $w\in C_c^\infty(\R^N)\setminus\{0\}$,
\[
\mathcal B_A(w,w)\ge \lambda_A \|D^sw\|_{L^2(\R^N)}^2 >0.
\]
Accordingly, the desired result follows 
from Lemma \ref{lemmaotheropenproblem}. 
\end{proof}

\section{Proof of Theorem \ref{CARASS}}\label{èòplokijhugyftrdeswaqASDFG}

To prove Theorem~\ref{CARASS}, we first provide the following representation of the kernel $k_A(x, y)$.

\begin{proposition}\label{ASYPR}
Let $s\in(0,1)$ and let $A\in C(\R^N,\R^{N\times N})\cap L^\infty(\R^N,\R^{N\times N})$ be symmetric and satisfy~\eqref{mkojihguyftdreswxrdcftvgybhjn}
for every compact set $K\subset\R^N$.
Let $k_A$ be as in \eqref{CHARKE}.

Then there exists a constant $C>0$, depending only on $N$, $s$, and $\|A\|_{L^\infty(\R^N,\R^{N\times N})}$,
such that, for every $x,y\in\R^N$ with $x\ne y$,
\[
\left| k_A(x,y)+\frac{d_{N,s}^2}{2|x-y|^{N+s}}\PV\int_{\R^N}\frac{w\cdot\left(A(w+x)-A(w+y)\right)(x-y)}{|w|^{N+s+1}|x-y|}\, dw\right|\le\frac{C}{|x-y|^{N+2s}}.
\]
\end{proposition}
\begin{proof}
Let $x\ne y\in\R^N$. We use the short notation $R:=|x-y|$ and $e:=\frac{y-x}{R}$. We also consider three contributions to the integral defining $k_A(x, y)$, namely
\[
\begin{split}
&I_1:=\PV\int_{B_{R/2}(x)}\frac{(z-x)\cdot A(z)(z-y)}{|z-x|^{N+s+1}|z-y|^{N+s+1}}\, dz,\\
&I_2:=\PV\int_{B_{R/2}(y)}\frac{(z-x)\cdot A(z)(z-y)}{|z-x|^{N+s+1}|z-y|^{N+s+1}} \,dz,\\
\mbox{ and }\quad &I_3:=\int_{\R^N\setminus(B_{R/2}(x)\cup B_{R/2}(y))}\frac{(z-x)\cdot A(z)(z-y)}{|z-x|^{N+s+1}|z-y|^{N+s+1}} \,dz.
\end{split}
\]
We also set
\[
v_R(x):=\PV\int_{B_{R/2}}\frac{w\cdot A(w+x)}{|w|^{N+s+1}} \,dw.
\]
We point out that this principal value is well defined. Indeed, by odd symmetry,
\[
\left|\;\int_{B_{R/2}\setminus B_\varepsilon}\frac{w\cdot A(w+x)}{|w|^{N+s+1}}\,dw\right| = \left|\,\int_{B_{R/2}\setminus B_\varepsilon}\frac{w\cdot\big(A(w+x)-A(x)\big)}{|w|^{N+s+1}} \,dw\right|\le C\int_\varepsilon^{R/2} \frac{\omega_{A,K}(r)}{r^{1+s}}\, dr
\]
where $K$ is a compact set containing $x+\overline{B_{R/2}}$.  By \eqref{mkojihguyftdreswxrdcftvgybhjn}, the right hand side is finite as $\varepsilon\searrow0$. Therefore the limit defining $v_R(x)$ exists.

We claim that
\begin{equation}\label{èplokijhugyftdreswaqsdf}
\left|I_1+\frac{v_R(x)\cdot e}{R^{N+s}}\right|\le\frac{C}{R^{N+2s}}.
\end{equation}
Indeed, for~$|w|<\frac{R}{2}$,
\[
\left|\frac{Re-w}{|Re-w|^{N+s+1}}-\frac{e}{R^{N+s}}\right|=\frac{1}{R^{N+s}}\left|\frac{e-w/R}{|e-w/R|^{N+s+1}}-e\right|\le\frac{C|w|}{R^{N+s+1}}.
\]
On this account, setting $w:=z-x$,
up to renaming the constant $C>0$ from line to line, 
\[
\begin{split}
&\left|I_1+\frac{v_R(x)\cdot e}{R^{N+s}}\right|\le\int_{B_{R/2}}\frac{|w|\,\|A\|_{L^\infty(\R^N, \R^{N\times N})}}{|w|^{N+s+1}}\left|\frac{Re-w}{|Re-w|^{N+s+1}}-\frac{e}{R^{N+s}}\right| \,dw\\
&\qquad\le\frac{C\|A\|_{L^\infty(\R^N, \R^{N\times N})}}{R^{N+s+1}}\int_{B_{R/2}}\frac{dw}{|w|^{N+s-1}}\le\frac{C}{R^{N+2s}},
\end{split}
\]
leading to the desired result in \eqref{èplokijhugyftdreswaqsdf}.

Since $A$ is symmetric, the same argument after swapping $x$ and $y$ yields
\begin{equation}\label{èplokijhugyftdreswaqsdf2}
\left|
I_2-\frac{v_R(y)\cdot e}{R^{N+s}}
\right|
\le
\frac{C}{R^{N+2s}}.
\end{equation}

Furthermore, up to renaming the constant $C>0$ from line to line, we have
\[
\begin{split}
&I_3\le\|A\|_{L^\infty(\R^N, \R^{N\times N})}\int_{\R^N\setminus(B_{R/2}(x)\cup B_{R/2}(y))} \frac{dz}{|z-x|^{N+s}|z-y|^{N+s}}\\
&\qquad\le\frac{C}{R^{N+s}}\int_{\R^N\setminus B_{R/2}(x)} \frac{dz}{|z-x|^{N+s}} \le\frac{C}{R^{N+2s}}.
\end{split}
\]
This estimate, combined with \eqref{èplokijhugyftdreswaqsdf}, \eqref{èplokijhugyftdreswaqsdf2} and \eqref{CHARKE}, yields
\[
\left|k_A(x,y) + \frac{d_{N,s}^2}{2R^{N+s}}(v_R(y)-v_R(x))\cdot e\right|\le\frac{C}{R^{N+2s}}.
\]
Since
\[
(v_R(y)-v_R(x))\cdot e =\PV\int_{B_{R/2}}\frac{w\cdot\big(A(w+x)-A(w+y)\big)(x-y)}{|w|^{N+s+1}|x-y|}\, dw,
\]
it only remains to consider the integral over $\R^N\setminus B_{R/2}$. For this purpose, we notice that
\begin{eqnarray*}&&
\frac{1}{R^{N+s}}\left|\;\int_{\R^N\setminus B_{R/2}}\frac{w\cdot\big(A(w+x)-A(w+y)\big)(x-y)}{|w|^{N+s+1}|x-y|} \,dw\right|
\\&&\qquad\quad \le \frac{2\|A\|_{L^\infty(\R^N, \R^{N\times N})}}{R^{N+s}}\int_{\R^N\setminus B_{R/2}}\frac{dw}{|w|^{N+s}}\le\frac{C}{R^{N+2s}}.
\end{eqnarray*}
Therefore, combining the previous two estimates, we obtain the desired result.
\end{proof}

With this, we can complete the proof of Theorem~\ref{CARASS}.

\begin{proof}[Proof of Theorem~\ref{CARASS}]
We first prove that $(i)$ implies $(ii)$. For $x\in\R^N$, we define
\begin{equation}\label{lkjohuygftdreswaesrdftgyh}
V(x):=\PV\int_{\R^N}\frac{w\cdot A(w+x)}{|w|^{N+s+1}}\, dw.
\end{equation}
We point out that
the principal value is well defined, thanks to odd symmetry and~\eqref{mkojihguyftdreswxrdcftvgybhjn}.

We also notice that, in terms of the function $V(x)$, Proposition \ref{ASYPR} gives
\begin{equation}\label{ASYPRV}
\left|k_A(x,y)+\frac{d_{N,s}^2}{2|x-y|^{N+s}}(V(x)-V(y))\cdot\frac{x-y}{|x-y|}\right|\le\frac{C}{|x-y|^{N+2s}}.
\end{equation}

Let $y\in\R^N$ and $e\in\mathbb S^{N-1}$. Let $R>0$ and set $y_R:=y+Re$. Now, since
\[
\lim_{R\to +\infty}\frac{|y_R|}{R} = 1,
\]
assumption~\eqref{ mnbvcxdszaxdcfvbnm,.poijuhgf} implies that
\[
\lim_{R\to+\infty}R^{N+s}k_A(y_R, y)=0.
\]
Therefore, multiplying \eqref{ASYPRV} by $R^{N+s}$ and sending $R\to+\infty$, we obtain that
\begin{equation}\label{,lòkmjnhugyftdreswaqsedrftvgybhnjm}
\lim_{R\to+\infty} (V(y_R)-V(y))\cdot e=0.
\end{equation}

For any~$t\in\R$, the limit in~\eqref{,lòkmjnhugyftdreswaqsedrftvgybhnjm} 
gives that
\[
\lim_{R\to+\infty} V(y+(t+R)e)\cdot e = V(y)\cdot e.
\]
Moreover,
applying~\eqref{,lòkmjnhugyftdreswaqsedrftvgybhnjm} with $y+te$ in place of $y$, we get
\[
\lim_{R\to+\infty} V(y+(t+R)e)\cdot e = V(y+te)\cdot e.
\]
Combining the last two displays, we find that
\[
(V(y+te)-V(y))\cdot e=0 \quad\mbox{ for every~$y\in\R^N$, $e\in\mathbb S^{N-1}$, and~$ t\in\R$.}
\]

Accordingly, we have
\[
\partial_e(V\cdot e)=0
\]
in the sense of distributions, for every $e\in\mathbb S^{N-1}$. 

Taking first $e=e_i$ and then $e=(e_i+e_j)/\sqrt2$, we infer that
\[
\partial_iV_i=0\quad\mbox{ and }\quad \partial_iV_j+\partial_jV_i=0\quad\mbox{for all }i,j=1,\ldots,N.
\]
Differentiating these identities, we obtain
\[
2\partial_i\partial_jV_k = \partial_i(\partial_jV_k+\partial_kV_j) + \partial_j(\partial_iV_k+\partial_kV_i) - \partial_k(\partial_iV_j+\partial_jV_i) = 0.
\]
Therefore, $V$ is affine.

We now show that $A$ is divergence free. To this end, let $A_j$ denote the $j$th column of $A$. Taking the Fourier transform of \eqref{lkjohuygftdreswaesrdftgyh} gives that, for every~$\xi\in\R^N\setminus\{0\}$,
\[
\widehat{V_j}(\xi) = C_{N,s}|\xi|^{s-1}\widehat{\mbox{div}(A_j)}(\xi),
\]
where $C_{N,s}\ne0$ depends only on $N$ and $s$. Since $V$ is affine, its Fourier transform is supported at the origin. Since $|\xi|^{s-1}$ is smooth and strictly positive on $\R^N\setminus\{0\}$, we infer that $\mbox{supp}(\widehat{\mbox{div}(A_j)})\subseteq\{0\}$. Thus, $\mbox{div}A_j$ is a polynomial.

We claim that this polynomial must vanish identically. Indeed, since $A_j$ is bounded, for every~$\varphi\in C_c^\infty(\R^N)$ and~$R>1$, we have
\begin{equation}\label{upperbhvbskegy458t9439}
\begin{split}
\left|\left\langle \mbox{div}(A_j),\varphi(\cdot/R)\right\rangle\right|=\left|\;\int_{\R^N}A_j(x)\cdot\nabla\big(\varphi(x/R)\big) \,dx\right|\le C\|A\|_{L^\infty(\R^N, \R^{N\times N})}R^{N-1}.
\end{split}
\end{equation}
Suppose by contradiction that $\mbox{div}(A_j)$ is not identically zero, and let $m\ge0$ be its degree. Writing
\[
\mbox{div}(A_j)=\mbox{div}(A_j)_{0}+\cdots+\mbox{div}(A_j)_m,
\]
where $\mbox{div}(A_j)_\ell$ is homogeneous of degree $\ell$ and $\mbox{div}(A_j)_m\not\equiv0$, we choose $\varphi\in C_c^\infty(\R^N)$ such that
\[
\int_{\R^N}\mbox{div}(A_j)_m(y)\varphi(y)\,dy\ne0.
\]
Then, by using the change of variables $y:=x/R$, we get
\[
\left\langle \mbox{div}(A_j),\varphi(\cdot/R)\right\rangle = \sum_{\ell=0}^m R^{N+\ell} \int_{\R^N}\mbox{div}(A_j)_\ell(y)\varphi(y) \,dy.
\]
Since the coefficient of $R^{N+m}$ is nonzero, there exist $c>0$ and $R_0>1$ such that
\[
\left|\left\langle \mbox{div}(A_j),\varphi(\cdot/R)\right\rangle\right|\ge cR^{N+m}\quad\mbox{for every }R\ge R_0.
\]
This, together with~\eqref{upperbhvbskegy458t9439}, gives the desired contradiction if~$R$ is large. Therefore, $\mbox{div}(A_j)=0$ in the sense of distributions, as desired.

We now prove that $(ii)$ implies $(i)$. Let $A_j(x)$ denote the $j$th column of the matrix $A(x)$. We claim that, for every $x\in\R^N$,
\begin{equation}\label{kpoijuhgvgt7frdswe5drfuipk}
\PV\int_{\R^N}\frac{w\cdot A_j(w+x)}{|w|^{N+s+1}}\, dw=0
.\end{equation}
Indeed, fixed $x\in\R^N$, we define~$
F(w):=A_j(w+x)$.
Since $\mbox{div}(A_j)=0$ in the sense of distributions, we also have \begin{equation}\label{c87ir54u6itgkjvewj43r32t0987654}
{\mbox{$\mbox{div}(F)=0$ in the sense of distributions.}}\end{equation} 

Moreover, for every $r>0$,
\begin{equation}\label{fluxzero00}
\int_{\partial B_r}F(w)\cdot\frac{w}{|w|}\,d\mathcal H^{N-1}(w)=0.
\end{equation}
This follows from an approximation argument. Namely,
let $F_\delta:=F\ast\rho_\delta$, where $\rho_\delta$ is a mollifier. In lieu of~\eqref{c87ir54u6itgkjvewj43r32t0987654}, we get
\[
\mbox{div}(F_\delta) = \left(\mbox{div}(F)\right)\ast\rho_\delta = 0.
\]
Moreover, since $F_\delta$ is smooth, by the divergence theorem, we get
\begin{equation}\label{jnhbgvfctdxrctfvygbhunj}
\int_{\partial B_r} F_\delta(w)\cdot\frac{w}{|w|}\, d\mathcal H^{N-1}(w) = \int_{B_r}\mbox{div}(F_\delta)(w)\, dw = 0.
\end{equation}
Moreover,
\[
|F_\delta(w)-F(w)|\le\int_{\R^N}\rho_\delta(z) |F(w-z)-F(w)| dz.
\]
Since $F$ is continuous, we get
\[
\lim_{\delta\searrow0}\sup_{w\in\partial B_r}|F_\delta(w)-F(w)|=0.
\]
Consequently, we have
\[
0\le\lim_{\delta\searrow0}\left|\;\int_{\partial B_r}\left(F_\delta(w)-F(w)\right)\cdot\frac{w}{|w|} \,d\mathcal H^{N-1}(w)\right|\le\mathcal H^{N-1}(\partial B_r)\lim_{\delta\searrow0}\sup_{w\in\partial B_r}|F_\delta(w)-F(w)|=0.
\] This and \eqref{jnhbgvfctdxrctfvygbhunj} entail~\eqref{fluxzero00}.

Let $0<\varepsilon<R$. By passing to polar coordinates and using \eqref{fluxzero00}, we obtain that
\[
\int_{B_R\setminus B_\varepsilon}\frac{w\cdot F(w)}{|w|^{N+s+1}} \,dw = \int_\varepsilon^R \frac{1}{r^{N+s}}\left(\;
\int_{\partial B_r} F(w)\cdot\frac{w}{|w|} \,d\mathcal H^{N-1}(w)\right)dr=0.
\]
Namely,\begin{equation}\label{anellozero00}
\int_{B_R\setminus B_\varepsilon}\frac{w\cdot A_j(w+x)}{|w|^{N+s+1}}\, dw=0.
\end{equation}

Furthermore, since $A(x)$ is bounded, we have
\[
\int_{\R^N\setminus B_R}\frac{|w\cdot A_j(w+x)|}{|w|^{N+s+1}} \,dw\le\|A\|_{L^\infty(\R^N, \R^{N\times N})}\int_{\R^N\setminus B_R}\frac{dw}{|w|^{N+s}} = \frac{C \|A\|_{L^\infty(\R^N, \R^{N\times N})}}{R^{s}},
\]
which tends to zero as $R\to+\infty$. Hence, using \eqref{anellozero00} and letting $R\to+\infty$, we obtain, for every $\varepsilon>0$,
\[
\int_{\R^N\setminus B_\varepsilon}\frac{w\cdot A_j(w+x)}{|w|^{N+s+1}}\, dw = \lim_{R\to+\infty}\int_{B_R\setminus B_\varepsilon}\frac{w\cdot A_j(w+x)}{|w|^{N+s+1}}\, dw = 0.
\]
Letting $\varepsilon\searrow0$ gives~\eqref{kpoijuhgvgt7frdswe5drfuipk}.

Thus, since \eqref{kpoijuhgvgt7frdswe5drfuipk} holds for every $j=1,\ldots,N$, we obtain that
\[
V(x)=0\quad\mbox{for every }x\in\R^N.
\]
Consequently, by \eqref{ASYPRV}, we infer that, for every $x\ne y$,
\[
|k_A(x,y)|\le\frac{C}{|x-y|^{N+2s}}
.\]

Finally, let $y\in\R^N$. If $|x|\ge2|y|$,  then $|x-y|\ge |x|-|y|\ge |x|/2$. Accordingly, we have
\[
|x|^{N+s}|k_A(x,y)|\le C\frac{|x|^{N+s}}{|x-y|^{N+2s}}\le\frac{C}{|x|^s}.
\]
Letting $|x|\to+\infty$, we obtain that
\[
\lim_{|x|\to+\infty}|x|^{N+s}k_A(x,y)=0,
\]
which proves $(i)$ and completes the equivalence.

In particular, if $N=1$, condition $(ii)$ reduces to $\partial_1 A_{11}=0$ in $\mathcal D'(\R)$. Since $A$ is continuous, it follows that $A$ is constant. This concludes the proof.
\end{proof}

\section{Proof of Theorem~\ref{SORP}}\label{l,kmojnbhugvycfdxsza<zsexdcv}

For smooth, matrix-valued functions whose derivatives of every order are bounded, we consider the distance
\[
d(M,\widetilde M):=\sum_{k=0}^{+\infty}\frac{p_k(M-\widetilde M)}{2^k\big(1+p_k(M-\widetilde M)\big)},
\]
where
\[
p_k(M):=\max_{{{\alpha\in\N^N}\atop{|\alpha|\le k}}\atop{i,j\in\{1,\dots,N\}}}\|\partial^\alpha M_{i,j}\|_{L^\infty(\R^N, \R^{N\times N})}.
\]

The main technical ingredient towards the proof of Theorem~\ref{SORP} is given by the following observation.

\begin{proposition}\label{PROgsz}
We have that
\[
\lim_{x\to+\infty} x^{1+s}\PV\int_{\R}\frac{(z-x)\,z\,\tanh z}{|z-x|^{2+s}|z|^{2+s}} \,dz = -\int_\R\frac{z\,\tanh z}{|z|^{2+s}} \,dz.
\]
\end{proposition}

The proof of Proposition~\ref{PROgsz} is somewhat delicate, due to the singularity at~$z=x$. For this reason, we separate the argument into three independent results (namely Lemmata~\ref{PROgsz1}, \ref{PROgsz2} and~\ref{PROgsz3}).

\begin{lemma}\label{PROgsz1}
We have that
\[
\lim_{x\to+\infty} x^{1+s}\int_{-x/2}^{x/2}\frac{(z-x)\,z\,\tanh z}{|z-x|^{2+s}|z|^{2+s}} \,dz = -\int_\R\frac{z\,\tanh z}{|z|^{2+s}} \,dz.
\]
\end{lemma}

\begin{proof}
Let $x\in(0, +\infty)$ and set
\[
g(z):=\frac{z\,\tanh z}{|z|^{2+s}}.
\]
We first observe that $g\in L^1(\R)$.  Since $z-x<0$ for $z\in[-x/2,x/2]$, we have
\[
x^{1+s}\int_{-x/2}^{x/2}\frac{(z-x)\,z\,\tanh z}{|z-x|^{2+s}|z|^{2+s}} \,dz = -\int_{\R}\left(1-\frac{z}{x}\right)^{-1-s} g(z)\,\chi_{[-x/2,x/2]}(z) \,dz.
\]
We also recall that the function $(0,+\infty)\ni r\mapsto r^{-1-s}$ is decreasing and consequently
$$ \max_{z\in[-x/2,x/2]}\left(1-\frac{z}{x}\right)^{-1-s}=\max_{t\in[-1/2,1/2]}(1-t)^{-1-s} =\max_{r\in[1/2,3/2]}r^{-1-s}=2^{1+s}.$$
Moreover, for every $z\in\R$,
\[
\lim_{x\to+\infty}\left(1-\frac{z}{x}\right)^{-1-s}\chi_{[-x/2,x/2]}(z)=1.
\]
Hence, by the dominated convergence theorem, we conclude that
\[
\lim_{x\to+\infty} x^{1+s}\int_{-x/2}^{x/2}\frac{(z-x)\,z\,\tanh z}{|z-x|^{2+s}|z|^{2+s}}\, dz = -\int_\R g(z)\, dz,
\]
as desired.
\end{proof}

\begin{lemma}\label{PROgsz2}
We have that
\[
\lim_{x\to+\infty} x^{1+s} \PV\int_{x/2}^{3x/2}\frac{(z-x)\,z\,\tanh z}{|z-x|^{2+s}|z|^{2+s}} \,dz=0.
\]
\end{lemma}

\begin{proof}
Let $z>0$ and define
\[
q(z):=\frac{\tanh z}{z^{1+s}}.
\]
Since
\[
\PV\int_{x/2}^{3x/2}\frac{z-x}{|z-x|^{2+s}}\,dz=0,
\]
we obtain that
\[
\PV\int_{x/2}^{3x/2}\frac{(z-x)\,z\,\tanh z}{|z-x|^{2+s}|z|^{2+s}} \,dz = \int_{x/2}^{3x/2}\frac{z-x}{|z-x|^{2+s}}\left(q(z)-q(x)\right) \,dz.
\]

We now take $x\ge2$. For $z\in[x/2,3x/2]$, we have
\[
q'(z) = \frac{\sech^2 z}{z^{1+s}} - (1+s)\frac{\tanh z}{z^{2+s}}.
\]
Since the function $z\sech^2z$ is bounded for $z\ge1$, there exists $C>0$ such that
\[
|q'(z)|\le\frac{C}{x^{2+s}}\quad\mbox{ for every }z\in[x/2,3x/2].
\]
Therefore, by the mean value theorem,
\[
|q(z)-q(x)|\le\frac{C|z-x|}{x^{2+s}}.
\]
Consequently, we have
\[
\left|\PV\int_{x/2}^{3x/2}\frac{(z-x)\,z\,\tanh z}{|z-x|^{2+s}|z|^{2+s}}\, dz\right|\le\frac{C}{x^{2+s}}\int_{x/2}^{3x/2}\frac{dz}{|z-x|^s}\le\frac{C}{x^{1+2s}}.
\]
Hence,
\[
0\le\lim_{x\to+\infty} x^{1+s}\left|\PV\int_{x/2}^{3x/2}\frac{(z-x)\,z\,\tanh z}{|z-x|^{2+s}|z|^{2+s}}\, dz\right|\le C\lim_{x\to+\infty}x^{-s}=0,
\]
which proves the desired result.
\end{proof}

\begin{lemma}\label{PROgsz3}
We have that
\[
\lim_{x\to+\infty} x^{1+s}\int_{(-\infty,-x/2)\cup(3x/2,+\infty)}\frac{(z-x)\,z\,\tanh z}{|z-x|^{2+s}|z|^{2+s}}\, dz=0.
\]
\end{lemma}

\begin{proof}
Since $|\tanh z|\le1$, and setting $z:=xw$, we infer that
\begin{eqnarray*}&&
\left|\;\int_{(-\infty,-x/2)\cup (3x/2,+\infty)} \frac{(z-x)\,z\,\tanh z}{|z-x|^{2+s}\,|z|^{2+s}}\,dz\right|
\le \int_{(-\infty,-x/2)\cup (3x/2,+\infty)} \frac{dz}{|z-x|^{1+s}\,|z|^{1+s}}\\
&&\qquad=\frac{1}{x^{1+2s}}\int_{(-\infty,-1/2)\cup (3/2,+\infty)} \frac{ dw}{ |w-1|^{1+s}\,|w|^{1+s}}=
\frac{C}{x^{1+2s}},
\end{eqnarray*}for some positive constant~$C$. It follows that
\[
0\le\lim_{x\to+\infty} x^{1+s}\left|\,\int_{(-\infty,-x/2)\cup(3x/2,+\infty)}\frac{(z-x)\,z\,\tanh z}{|z-x|^{2+s}|z|^{2+s}}\, dz\right|\le C\lim_{x\to+\infty}x^{-s}=0,
\]
as desired.
\end{proof}

\begin{proof}[Proof of Proposition~\ref{PROgsz}]
The desired result plainly follows from Lemmata~\ref{PROgsz1}, \ref{PROgsz2} and~\ref{PROgsz3}.
\end{proof}

With this,  we can construct the counterexample advertised in Theorem~\ref{SORP}.
\begin{proof}[Proof of Theorem~\ref{SORP}]
Let $N=1$ and take $A(x):=1$.  By \cite[Theorem 1.2]{DSVPART1}, Open Problem~\ref{opeSPJMDPKM} has an affirmative answer for $A(x)=1$. For $n\ge2$,  we define
\[
A_n(x):=1-\frac{\tanh x}{n}.
\]
Since $|\tanh x|<1$, we have that, for every $x\in\R$,
\[
1-\frac1n<A_n(x)<1+\frac1n
.\]
In particular, for every $n\ge2$,
\[
\frac12\le A_n(x)\le\frac32
\]
Hence, all the coefficients $A_n$ are bounded and uniformly elliptic.

We also observe that
\begin{equation}\label{cjewr796598732WSDFGHvdt3478y9o}
\lim_{n\to+\infty}d(A_n,A)=0.
\end{equation}
Indeed, for every $k\in\N$,
\[
p_k(A_n-A) = \frac{p_k(\tanh)}{n}.
\]
Therefore,
\[
\lim_{n\to+\infty}\frac{p_k(A_n-A)}{2^k\big(1+p_k(A_n-A)\big)}=0.
\]
Moreover,
\[
0\le\frac{p_k(A_n-A)}{2^k (1+p_k(A_n-A))}\le\frac1{2^k}.
\]
Thus, the dominated convergence theorem for series yields~\eqref{cjewr796598732WSDFGHvdt3478y9o}.

We now assume by contradiction that that Open Problem~\ref{opeSPJMDPKM} has an affirmative answer for some $A_n$. Thus, by \eqref{CHARKE}, the corresponding kernel is given by
\[
k_{A_n}(x,y)=-\frac{d_{1,s}^2}{2}\PV\int_{\R}\frac{(z-x)(z-y)\left(1-\frac{\tanh z}{n}\right)}{|z-x|^{2+s}|z-y|^{2+s}} \,dz.
\]
In particular, for $x\in (0, +\infty)$,
\[
\frac{2k_{A_n}(x,0)}{d_{1,s}^2} = -\PV\int_{\R}\frac{(z-x)z}{|z-x|^{2+s}|z|^{2+s}} \,dz + \frac1n\PV\int_{\R}\frac{(z-x)z\tanh z}{|z-x|^{2+s}|z|^{2+s}} \,dz.
\]
By the change of variable~$w:=z/x$, we obtain that
\[ \lim_{x\to+\infty}
x^{1+s}\PV\int_{\R}\frac{(z-x)z}{|z-x|^{2+s}|z|^{2+s}} \,dz =
\lim_{x\to+\infty} x^{-s}\PV\int_{\R}\frac{(w-1)w}{|w-1|^{2+s}|w|^{2+s}} \,dw=0.
\]
Moreover, Proposition~\ref{PROgsz} gives
\[
\lim_{x\to+\infty} x^{1+s}\PV\int_{\R}\frac{(z-x)z\tanh z}{|z-x|^{2+s}|z|^{2+s}} \,dz = -\int_\R\frac{z\tanh z}{|z|^{2+s}} \,dz.
\]
Therefore,
\[
\lim_{x\to+\infty}\frac{2x^{1+s}k_{A_n}(x,0)}{d_{1,s}^2} = -\frac1n\int_\R\frac{z\tanh z}{|z|^{2+s}} \,dz.
\]

Since
\[
\frac{z\tanh z}{|z|^{2+s}}>0 \quad\mbox{ for every }z\ne0,
\]
and the latter function belongs to $L^1(\R)$, we have
\[
\int_\R\frac{z\tanh z}{|z|^{2+s}}\,dz>0.
\]
Accordingly, we have
\[
\lim_{x\to+\infty} x^{1+s}k_{A_n}(x,0)<0.
\]
Hence there exists $R_n>0$ such that
\[
k_{A_n}(x,0)<0\quad\mbox{for every }x\ge R_n.
\]
This contradicts the nonnegativity required in Open Problem~\ref{opeSPJMDPKM}. Therefore Open Problem~\ref{opeSPJMDPKM} has a negative answer for every $A_n$, with $n\ge2$. This concludes the proof.
\end{proof}

\appendix

\section{Proof of Corollary \ref{COINCIDONOCOST}}\label{mkjnbhgvfcdxrsezaesxrdctfvyguhoplòè'0o98765}

Throughout this appendix, $A\in\R^{N\times N}$ is a constant and symmetric matrix. Moreover, we introduce the following notations. For $0<\alpha<N$,  we set
\begin{equation}\label{lkojihugvyfctdxsza<zsxc}
\gamma_{N,\alpha}:=\frac{\displaystyle\Gamma\left(\frac{N-\alpha}{2}\right)}{\displaystyle 2^\alpha\pi^{N/2}\Gamma\left(\frac{\alpha}{2}\right)}.
\end{equation}
Since
\[
\Gamma\left(\frac{N+s+1}{2}\right) = \frac{N+s-1}{2}\Gamma\left(\frac{N+s-1}{2}\right),
\]
by \eqref{lkojihugvyfctdxsza<zsxc} and~\eqref{DNSDEFN} we deduce that
\begin{equation}\label{ojnhbgvfcdxszexrdctfvygbhn}
(N+s-1)\gamma_{N,1-s} = \frac{2^{s-1}(N+s-1)\Gamma\left(\frac{N+s-1}{2}\right)}{\pi^{N/2}\Gamma\left(\frac{1-s}{2}\right)} = \frac{2^s\Gamma\left(\frac{N+s+1}{2}\right)}{\pi^{N/2}\Gamma\left(\frac{1-s}{2}\right)} = d_{N,s}.
\end{equation}

We also recall that if~$0<\alpha, \beta<N$ with $\alpha+\beta<N$, then 
\begin{equation}\label{STEININTEGRAL}
\gamma_{N,\alpha}\gamma_{N,\beta}\int_{\R^N}|z-x|^{\alpha-N}|z-y|^{\beta-N} dz = \gamma_{N,\alpha+\beta} |x-y|^{\alpha+\beta-N},
\end{equation}
see e.g.~\cite[page 118, formula~(8)]{MR290095}.

We are now in the position to prove Corollary \ref{COINCIDONOCOST}.

\begin{proof}[Proof of Corollary \ref{COINCIDONOCOST}]
We distinguish three cases.
\smallskip

\noindent {\bf{Case 1}}: We first assume that $N+2s>2$ and we set
\[
F(x,y):=\int_{\R^N}|z-x|^{1-s-N}|z-y|^{1-s-N} dz.
\]
Formula \eqref{STEININTEGRAL}, here used with $\alpha=\beta=1-s$, gives that
\begin{equation}\label{,lòmkjnbhgvyfctdxrszexrdcfvgb}
F(x,y)=\frac{\gamma_{N,2-2s}}{\gamma_{N,1-s}^2}|x-y|^{2-N-2s}.
\end{equation}

Let now $x\ne y$ and take $\rho:=\rho_{x,y}\in(0,|x-y|/4)$. We claim that, if $\varphi$ is continuously differentiable in $B_\rho(x)$, then
\begin{equation}\label{,òlkmjnbhgvfctdxrszea<zesxdrcfvghbjnm}
\lim_{\varepsilon\searrow0}\int_{B_\rho(x)\setminus B_\varepsilon(x)}\frac{(z-x)_j\varphi(z)}{|z-x|^{N+s+1}}\, dz = \int_{B_\rho(x)}\frac{(z-x)_j (\varphi(z)-\varphi(x))}{|z-x|^{N+s+1}} \,dz.
\end{equation}
Indeed, by odd symmetry,
\begin{eqnarray*}
\int_{B_\rho(x)\setminus B_\varepsilon(x)}\frac{(z-x)_j\varphi(z)}{|z-x|^{N+s+1}}\, dz =\int_{B_\rho(x)\setminus B_\varepsilon(x)}\frac{(z-x)_j(\varphi(z)-\varphi(x))} {|z-x|^{N+s+1}} \,dz.
\end{eqnarray*}
Moreover, since $\varphi$ is continuously differentiable in~$B_\rho(x)$, there exists $C>0$ such that
\[
\frac{|(z-x)_j|\,|\varphi(z)-\varphi(x)|}{|z-x|^{N+s+1}}\le\frac{C}{|z-x|^{N+s-1}},
\]
which belongs to~$L^1(B_\rho(x))$. Hence~\eqref{,òlkmjnbhgvfctdxrszea<zesxdrcfvghbjnm} follows by the dominated convergence theorem. 

Using \eqref{,òlkmjnbhgvfctdxrszea<zesxdrcfvghbjnm} with $\varphi(z):=|z-y|^{1-s-N}$, we infer that, for $x\ne y$, 
\begin{eqnarray*}
\frac{\partial F}{\partial x_j}(x,y)&=&
(N+s-1)\lim_{\varepsilon\searrow0}\int_{\R^N\setminus B_\varepsilon(x)}\frac{(z-x)_j}{|z-x|^{N+s+1}}|z-y|^{1-s-N}\, dz.
\end{eqnarray*}
Differentiating with respect to $y_k$, we get
\[
\frac{\partial^2F}{\partial x_j\partial y_k}(x,y) = (N+s-1)^2\lim_{\delta\searrow0}\lim_{\varepsilon\searrow0}\int_{\R^N\setminus\left(B_\varepsilon(x)\cup B_\delta(y)\right)}\frac{(z-x)_j(z-y)_k}{|z-x|^{N+s+1}|z-y|^{N+s+1}} \,dz.
\]

We point out that the two singularities can be addressed independently.  Indeed, on $B_\rho(x)$ the map
\[
z\mapsto\frac{(z-y)_k}{|z-y|^{N+s+1}}
\]
is continuously differentiable. Moreover, by symmetry,
\begin{eqnarray*}&&
\int_{B_\rho(x)\setminus B_\varepsilon(x)}\frac{(z-x)_j(z-y)_k}{|z-x|^{N+s+1}|z-y|^{N+s+1}} \,dz \\&&\quad= \int_{B_\rho(x)\setminus B_\varepsilon(x)}\frac{(z-x)_j}{|z-x|^{N+s+1}}\left(\frac{(z-y)_k}{|z-y|^{N+s+1}}-\frac{(x-y)_k}{|x-y|^{N+s+1}}\right)\,dz,
\end{eqnarray*}
In addition,  there exists $C>0$ such that
\[
\left|\frac{(z-x)_j}{|z-x|^{N+s+1}}\right|\left|\frac{(z-y)_k}{|z-y|^{N+s+1}}-\frac{(x-y)_k}{|x-y|^{N+s+1}}\right|\le \frac{C}{|z-x|^{N+s-1}},
\]
which is integrable in a neighborhood of $x$. Thus, by the dominated convergence theorem, we get
\begin{eqnarray*}&&
\lim_{\varepsilon\searrow0}\int_{B_\rho(x)\setminus B_\varepsilon(x)}\frac{(z-x)_j(z-y)_k}{|z-x|^{N+s+1}|z-y|^{N+s+1}} \,dz \\&&\quad= \int_{B_\rho(x)}\frac{(z-x)_j}{|z-x|^{N+s+1}}\left(\frac{(z-y)_k}{|z-y|^{N+s+1}}-\frac{(x-y)_k}{|x-y|^{N+s+1}}\right) \,dz.
\end{eqnarray*}
The same holds in $B_\rho(y)$. Also, away from $B_\rho(x)\cup B_\rho(y)$ the integrand is absolutely integrable.
Accordingly, we can write
\[
\frac{\partial^2F}{\partial x_j\partial y_k}(x,y) = (N+s-1)^2\lim_{\varepsilon\searrow0}\int_{\R^N\setminus\left(B_\varepsilon(x)\cup B_\varepsilon(y)\right)}
\frac{(z-x)_j(z-y)_k}{|z-x|^{N+s+1}|z-y|^{N+s+1}}\, dz.
\]

Thus, recalling \eqref{INT2PVdefn} and~\eqref{,lòmkjnbhgvyfctdxrszexrdcfvgb}, we obtain that
\[
\begin{split}
&\sum_{j,k=1}^N A_{jk}\frac{\partial^2}{\partial x_j\partial y_k}\left(\frac{\gamma_{N,2-2s}}{\gamma_{N,1-s}^2}|x-y|^{2-N-2s}\right)\\
&\quad=
\sum_{j,k=1}^N A_{jk}\frac{\partial^2F}{\partial x_j\partial y_k}(x,y)\\ &\quad= (N+s-1)^2 \lim_{\varepsilon\searrow0}\int_{\R^N\setminus\left(B_\varepsilon(x)\cup B_\varepsilon(y)\right)}\frac{(z-x)\cdot A(z-y)}{|z-x|^{N+s+1}|z-y|^{N+s+1}} dz\\
&\quad=(N+s-1)^2I_A(x,y).
\end{split}\]

Now, we observe that
\[
\frac{\partial^2}{\partial x_j\partial y_k}|x-y|^{2-N-2s} = (N+2s-2)|x-y|^{-N-2s}\left(\delta_{jk} -(N+2s)\frac{(x_j-y_j)(x_k-y_k)}{|x-y|^2}\right),
\]
hence
\[
(N+s-1)^2I_A(x,y) = \frac{\gamma_{N,2-2s}(N+2s-2)}{\gamma_{N,1-s}^2}\frac{1}{|x-y|^{N+2s}}\left(\tr(A)-(N+2s)\frac{(x-y)\cdot A(x-y)}{|x-y|^2}\right).
\]
Moreover, since 
\[
\Gamma\left(\frac{N+2s}{2}\right) = \frac{N+2s-2}{2}\Gamma\left(\frac{N+2s-2}{2}\right),
\]
recalling the definition in \eqref{lkojihugvyfctdxsza<zsxc}, we get
\[
\gamma_{N,2-2s}(N+2s-2) = \frac{2^{2s-1}\Gamma\left(\frac{N+2s}{2}\right)}{\pi^{N/2}\Gamma(1-s)}.
\]
Combining the previous displays, using first \eqref{ojnhbgvfcdxszexrdctfvygbhn},  and then \eqref{,mknjbvcxzdxcfvbn}, we conclude that
\[
\begin{split}
I_A(x,y) &= \frac{2^{2s-1}\Gamma\left(\frac{N+2s}{2}\right)}{\pi^{N/2}\Gamma(1-s)d_{N,s}^2}\frac{1}{|x-y|^{N+2s}}\left(\tr(A)-(N+2s)\frac{(x-y)\cdot A(x-y)}{|x-y|^2}\right)\\
&=\frac{c_{N,s}}{2s\,d_{N,s}^2|x-y|^{N+2s}}\left(\tr(A)-(N+2s)\frac{(x-y)\cdot A(x-y)}{|x-y|^2}\right).
\end{split}
\]
Thus,  \eqref{PARTE1COR1.5} follows from \eqref{INT2PVdefn} and Theorem~\ref{thmreg2}.
\smallskip

\noindent {\bf{Case 2}}: We now address the case $N=1$ and $0<s<1/2$. We notice that in dimension one the argument used in the first case for $N+2s>2$ cannot be directly applied since
\[
F(x, y)=\int_{\R}\frac{dz}{|z-h|^s|z|^s}
\]
does not converge at infinity.

To overcome this difficulty, for $h\ne0$, we define
\[
\Phi_s(h):=\int_{\R}\left(\frac{1}{|z-h|^s|z|^s}-\frac{1}{|z|^{2s}}\right)\,dz.
\]
Indeed, for $|z|\ge2|h|$, there exists a constant $C>0$ depending only on $s$, such that
\[
\left|\frac{1}{|z-h|^s|z|^s} - \frac{1}{|z|^{2s}}\right|\le C\frac{|h|}{|z|^{1+2s}},
\] which is
now integrable at infinity.

Moreover, the change of variable $t:=z/|h|$ gives
\begin{equation}\label{567vhewiyt43683ertywdfeghwjRTYU}
\Phi_s(h)=|h|^{1-2s}\Phi_s(1)\quad\mbox{ for every }h\ne0.
\end{equation}
Hence we now focus on computing $\Phi_s(1)$.
We split the integral over the three intervals $(-R,0)$, $(0,1)$ and $(1,R)$. By using the changes of variables $t:=-z$ in $(-R,0)$ and
$t:=z-1$ in $(1,R)$, respectively, we obtain that
\[
\Phi_s(1) = \int_0^1\frac{dt}{t^s(1-t)^s} + \lim_{R\to+\infty}\left(\;\int_0^R\frac{dt}{t^s(1+t)^s} + \int_0^{R-1}\frac{dt}{t^s(1+t)^s} -2\int_0^R\frac{dt}{t^{2s}}
\right).
\]
Since
\[
\lim_{R\to+\infty}\int_{R-1}^R\frac{dt}{t^s(1+t)^s}=0,
\]
we can write
\[
\Phi_s(1) = \int_0^1\frac{dt}{t^s(1-t)^s} + 2\lim_{R\to+\infty}\left(\int_0^R\frac{dt}{t^s(1+t)^s}-\int_0^R\frac{dt}{t^{2s}}\right).
\]

Now, by \cite[8.380, n.1, page 917]{MR3307944}, we know that
\[
\int_0^1\frac{dt}{t^s(1-t)^s}=\frac{\Gamma(1-s)^2}{\Gamma(2-2s)}.
\]
Moreover,
\[
\lim_{R\to+\infty}\left(\;\int_0^R\frac{dt}{t^s(1+t)^s} -
\int_0^R\frac{dt}{t^{2s}}\right) = \int_0^{+\infty}\left(\frac{1}{t^s(1+t)^s}-\frac{1}{t^{2s}}\right)\,dt = \int_0^{+\infty} \left(\left(\frac{t}{1+t}\right)^s-1\right)\frac{dt}{t^{2s}}.
\]
Thus, integrating by parts and using \cite[8.380, n.1, page 917]{MR3307944}, we obtain that
\[
\int_0^{+\infty} \left(\left(\frac{t}{1+t}\right)^s-1\right) \frac{dt}{t^{2s}} = -\frac{s}{1-2s}\int_0^{+\infty} \frac{t^{-s}}{(1+t)^{s+1}} dt = -\frac{\Gamma(1-s)\Gamma(2s)}{(1-2s)\Gamma(s)}.
\]
Combining the previous displays, we get
\[
\Phi_s(1) = \frac{\Gamma(1-s)^2}{\Gamma(2-2s)}-\frac{2\Gamma(1-s)\Gamma(2s)}{(1-2s)\Gamma(s)}.
\]
By using \eqref{DNSDEFN}, \eqref{,mknjbvcxzdxcfvbn}, and the properties of the  Gamma function, we have
\[
\Phi_s(1)=-\frac{s\,c_{1,s}}{2(1-2s)d_{1,s}^2}.
\]
Thus, recalling~\eqref{567vhewiyt43683ertywdfeghwjRTYU}, we conclude that
\begin{equation}\label{lèkpugyftdresdxtfvgbhnj}
\Phi_s(h) = -\frac{s\,c_{1,s}}{2(1-2s)d_{1,s}^2}|h|^{1-2s}.
\end{equation}

We now return to \eqref{INT2PVdefn}. Since $N=1$, the matrix $A$ consists of only one entry, which we still denote by $A$. Moreover, $\tr(A)=A$ and $\frac{(x-y)\cdot A(x-y)}{|x-y|^2}=A$. Thus, in particular,
\begin{equation}\label{,kmjnhbgvfctdxrsza<zsxdcvbnm,mnhgfds}
\tr(A) - (1+2s)\frac{(x-y)\cdot A(x-y)}{|x-y|^2} = -2sA.
\end{equation}

For $R>|h|+1$ and $\varepsilon>0$,  we set
\[
I_{R,\varepsilon}(h):=\int_{[-R,R]\setminus((-\varepsilon,\varepsilon)\cup (h-\varepsilon,h+\varepsilon) )}
\frac{dz}{|z-h|^s|z|^s}.
\]
Differentiating with respect to $h$, we get
\begin{equation}\label{7378thewyit3846234rtdvtweyvft}
\frac{d}{dh}I_{R,\varepsilon}(h) = s\int_{[-R,R]\setminus((-\varepsilon,\varepsilon)\cup (h-\varepsilon,h+\varepsilon) )}\frac{z-h}{|z-h|^{s+2}|z|^s}\, dz + \frac{1}{\varepsilon^s}\left(\frac{1}{|h-\varepsilon|^s}-\frac{1}{|h+\varepsilon|^s}\right).
\end{equation}
Since $h\ne 0$, the last term in the right hand side converges to $0$ as $\varepsilon\searrow0$. 

We now consider the integral in the right hand side near the two singularities. Near $z=0$, the function
\[
z\mapsto\frac{z-h}{|z-h|^{s+2}}
\]
is bounded in a neighborhood of $0$. Hence, there exists $C>0$ such that
\[
\left|\frac{z-h}{|z-h|^{s+2}|z|^s}\right|\le\frac{C}{|z|^s},
\]
which is integrable near $0$. 

Near $z=h$, we write
\[
\int_{(h-\rho,h+\rho)\setminus (h-\varepsilon,h+\varepsilon)}\frac{z-h}{|z-h|^{s+2}|z|^s} \,dz = \int_{(h-\rho,h+\rho)\setminus (h-\varepsilon,h+\varepsilon)}\frac{z-h}{|z-h|^{s+2}}\left(\frac{1}{|z|^s}-\frac{1}{|h|^s}\right) \,dz.
\]
Moreover, since $z\mapsto |z|^{-s}$ is continuously differentiable near $h\ne 0$, we have
\[
\left|\frac{1}{|z|^s}-\frac{1}{|h|^s}\right|\le C|z-h|,
\]
and therefore
\[
\left|\frac{z-h}{|z-h|^{s+2}}\left(\frac{1}{|z|^s}-\frac{1}{|h|^s}\right)\right|\le\frac{C}{|z-h|^s},
\]
which is integrable near $h$. 

Finally, for $|z|\ge 2|h|$, we have $|z-h|\ge\ |z|/2$, and hence
\[
\left|\frac{z-h}{|z-h|^{s+2}|z|^s}\right|\le\frac{C}{|z|^{2s+1}},
\]
which is integrable at infinity. 

Gathering these observations, and letting first
$\varepsilon\searrow0$ and then $R\to+\infty$ in~\eqref{7378thewyit3846234rtdvtweyvft}, we obtain that
\[
\frac{d}{dh}\Phi_s(h)=\lim_{\varepsilon\searrow0}\int_{\R\setminus\left((-\varepsilon,\varepsilon)\cup (h-\varepsilon,h+\varepsilon)\right)}\frac{\partial}{\partial h}|z-h|^{-s}\,|z|^{-s}\, dz.
\]
Since
\[
\frac{\partial}{\partial h}|z-h|^{-s} = -\frac{\partial}{\partial z}|z-h|^{-s},
\]
we now integrate by parts with respect to $z$.  We first assume that $h>0$. In this case, let~$R>2h$ and $0<\varepsilon<h/2$. We integrate separately on $(-R,-\varepsilon)$, $(\varepsilon,h-\varepsilon)$, and $(h+\varepsilon,R)$. We notice that the boundary contributions at the two singular points cancel out, as
\[
\varepsilon^{-s}\left(|h-\varepsilon|^{-s}-|h+\varepsilon|^{-s}\right) + \varepsilon^{-s}\left(|h+\varepsilon|^{-s}-|h-\varepsilon|^{-s}\right)=0.
\]
Moreover, the contribution at $\pm R$ is given by
\[
R^{-s}\left((R+h)^{-s}-(R-h)^{-s}\right),
\]
and tends to zero as $R\to+\infty$.  Therefore, we have
\[
\frac{d}{dh}\Phi_s(h) = \lim_{\varepsilon\searrow0} \int_{\R\setminus\left((-\varepsilon,\varepsilon)\cup (h-\varepsilon,h+\varepsilon)\right)} |z-h|^{-s}\frac{d}{dz}|z|^{-s} dz.
\]
The same identity holds for $h<0$.

We now differentiate once more with respect to $h$. Let $h>0$ and $0<\varepsilon<h/2$.  Thus, 
\[
\begin{split}
&\frac{d}{dh}\int_{\R\setminus\left((-\varepsilon,\varepsilon)\cup (h-\varepsilon,h+\varepsilon)\right)}|z-h|^{-s}\frac{d}{dz}|z|^{-s} \,dz\\
&\quad=\int_{\R\setminus\left((-\varepsilon,\varepsilon)\cup (h-\varepsilon,h+\varepsilon)\right)}\frac{\partial}{\partial h}|z-h|^{-s}\frac{d}{dz}|z|^{-s} \,dz + \varepsilon^{-s}\left(\frac{d}{dz}|z|^{-s}\Bigg|_{z=h-\varepsilon} - \frac{d}{dz}|z|^{-s}\Bigg|_{z=h+\varepsilon}\right).
\end{split}
\]
Since $z\mapsto \frac{d}{dz}|z|^{-s}$ is continuously differentiable in a neighborhood of $h\ne 0$,  there exists $C_h>0$ such that
\[
\left|\frac{d}{dz}|z|^{-s}\bigg|_{z=h-\varepsilon} - \frac{d}{dz}|z|^{-s}\bigg|_{z=h+\varepsilon}\right|\le C_h\varepsilon.
\]
Thus, the boundary term is bounded by $C_h\varepsilon^{1-s}$ and tends to zero as $\varepsilon\searrow0$.

It remains to justify the convergence of the integral near $z=h$. Since $h\ne 0$, the function~$z\mapsto \frac{d}{dz}|z|^{-s}$ is continuously differentiable in a neighborhood of $h$. Hence, we have
\[
\frac{d}{dz}|z|^{-s} = \left(\frac{d}{dz}|z|^{-s} - \left.\frac{d}{dz}|z|^{-s}\right|_{z=h}\right) + \left.\frac{d}{dz}|z|^{-s}\right|_{z=h}.
\]
Therefore,
\[
\begin{split}
&\int_{B_\rho(h)\setminus B_\varepsilon(h)}\frac{\partial}{\partial h}|z-h|^{-s}\frac{d}{dz}|z|^{-s} \,dz\\
=\;&\int_{B_\rho(h)\setminus B_\varepsilon(h)}\frac{\partial}{\partial h}|z-h|^{-s}\left(\frac{d}{dz}|z|^{-s}-\left.\frac{d}{dz}|z|^{-s}\right|_{z=h}\right)dz + \left.\frac{d}{dz}|z|^{-s}\right|_{z=h}\int_{B_\rho(h)\setminus B_\varepsilon(h)}\frac{\partial}{\partial h}|z-h|^{-s} \,dz.
\end{split}
\]
We notice that the last integral vanishes by symmetry. Moreover, there exists $C_h>0$ such that
\[
\left|\frac{d}{dz}|z|^{-s}-\left.\frac{d}{dz}|z|^{-s}\right|_{z=h}\right|\le C_h|z-h|,
\]
and therefore
\[
\left|\frac{\partial}{\partial h}|z-h|^{-s}\left(\frac{d}{dz}|z|^{-s}-\left.\frac{d}{dz}|z|^{-s}\right|_{z=h}\right)\right|\le \frac{C_h}{|z-h|^s},
\]
and the latter function is integrable near $h$. Similarly, near $z=0$, the function $z\mapsto\frac{\partial}{\partial h}|z-h|^{-s}$ is continuously differentiable, and the singular odd part of $\frac{d}{dz}|z|^{-s}$ cancels out by symmetry. More precisely, we get
\[
\left|\frac{\partial}{\partial h}|z-h|^{-s}-\frac{\partial}{\partial h}|-h|^{-s}\right|\le C_h|z|,
\]
so the remaining integrand is bounded by $C_h/|z|^s$, which is integrable near $0$. Finally, for~$|z|\ge 2|h|$, we obtain that
\[
\left|\frac{\partial}{\partial h}|z-h|^{-s}\frac{d}{dz}|z|^{-s}\right|\le\frac{C}{|z|^{2s+1}},
\]
which is integrable at infinity. Therefore,
\[
\frac{d^2}{dh^2}\Phi_s(h)=\lim_{\varepsilon\searrow0}\int_{\mathbb{R}\setminus(B_\varepsilon(0)\cup B_\varepsilon(h))}\frac{\partial}{\partial h}|z-h|^{-s}\frac{d}{dz}|z|^{-s}\, dz.
\]
Since
\[
\frac{\partial}{\partial h}|z-h|^{-s} = s\frac{z-h}{|z-h|^{s+2}}\quad\mbox{ and }\quad\frac{d}{dz}|z|^{-s}=-s\frac{z}{|z|^{s+2}},
\]
we conclude that
\[
\frac{d^2}{dh^2}\Phi_s(h) = -s^2\lim_{\varepsilon\searrow0}\int_{\R\setminus\left((-\varepsilon,\varepsilon)\cup (h-\varepsilon,h+\varepsilon)\right)}\frac{(z-h)z}{|z-h|^{s+2}|z|^{s+2}}\, dz.
\]
Since we set $h=x-y$, we get
\[
\frac{\partial^2}{\partial x\partial y}\Phi_s(x-y) = -\frac{d^2}{dh^2}\Phi_s(h),
\]
and hence, after translating $z\mapsto z-y$, we obtain that
\begin{equation}\label{lkjohgfdtrxsd}
A\frac{\partial^2}{\partial x\partial y}\Phi_s(x-y) = s^2 \lim_{\varepsilon\searrow0}A\int_{\R\setminus\left((x-\varepsilon,x+\varepsilon)\cup (y-\varepsilon,y+\varepsilon)\right)}\frac{(z-x)(z-y)}{|z-x|^{s+2}|z-y|^{s+2}} \,dz = s^2I_A(x,y).
\end{equation}
Now, by \eqref{lèkpugyftdresdxtfvgbhnj}, and since
\[
\frac{\partial^2}{\partial x\partial y}|x-y|^{1-2s} = 2s(1-2s)|x-y|^{-1-2s},
\]
it follows that
\begin{equation}\label{lkjohgfdtrxsd2}
\frac{\partial^2}{\partial x\partial y}\Phi_s(x-y) = -\frac{s\,c_{1,s}}{2(1-2s)d_{1,s}^2}\,2s(1-2s)|x-y|^{-1-2s} = -\frac{s^2c_{1,s}}{d_{1,s}^2}|x-y|^{-1-2s}.
\end{equation}
Thus, combining \eqref{lkjohgfdtrxsd}, \eqref{lkjohgfdtrxsd2} and \eqref{,kmjnhbgvfctdxrsza<zsxdcvbnm,mnhgfds}, we get
\[
I_A(x,y) = -\frac{c_{1,s}}{d_{1,s}^2}\frac{A}{|x-y|^{1+2s}} =\frac{c_{1,s}}{2s\,d_{1,s}^2|x-y|^{1+2s}}\left(\tr(A)-(1+2s)\frac{(x-y)\cdot A(x-y)}{|x-y|^2}\right).
\]
Thus,  \eqref{PARTE1COR1.5} follows by \eqref{INT2PVdefn}, and Theorem \ref{thmreg2}.
\smallskip

\noindent {\bf{Case 3}}:
We finally consider the case $N=1$ and $s=1/2$.
We set $h:=x-y$. By using the change of variable~$z:=y+ht$, 
we get
\begin{equation}\label{lòkmjnbhgvcfdxszxdcfvgb}
\begin{split}
I_A(x,y)&= A\lim_{\varepsilon\searrow0}\int_{\R\setminus\left(
(x-\varepsilon, x+\varepsilon)\cup (y-\varepsilon, y+\varepsilon)\right)}\frac{(z-x)(z-y)}
{|z-x|^{5/2}|z-y|^{5/2}}\, dz\\
&= \frac{A}{|h|^2}\lim_{\varepsilon\searrow0}\int_{\R\setminus\left(
(-\varepsilon/|h|, \varepsilon/|h|)\cup (1- \varepsilon/|h|,1+\varepsilon/|h|)
\right)}\frac{t(t-1)}{|t|^{5/2}|t-1|^{5/2}} \,dt.
\end{split}
\end{equation}
Since $h\ne0$, by setting $\delta:=\varepsilon/|h|$, we get
\[
I_A(x,y)=\frac{A}{|x-y|^2}\lim_{\delta\searrow0}\int_{\R\setminus\left((-\delta, \delta) \cup (1-\delta, 1+\delta)\right)}\frac{t(t-1)}{|t|^{5/2}|t-1|^{5/2}}\, dt.
\]
For $\delta\in(0,1/2)$, we split the integral over the three intervals $(-\infty,-\delta)$, $(\delta,1-\delta)$ and $(1+\delta,+\infty)$. This gives
\[
\int_{\R\setminus\left((-\delta, \delta) \cup (1-\delta, 1+\delta)
\right)}\frac{t(t-1)}{|t|^{5/2}|t-1|^{5/2}}\, dt = \int_{-\infty}^{-\delta}\frac{dt}{|t|^{3/2}|t-1|^{3/2}}-\int_\delta^{1-\delta}\frac{dt}{t^{3/2}(1-t)^{3/2}} + \int_{1+\delta}^{+\infty}\frac{dt}{t^{3/2}(t-1)^{3/2}}.
\]
We notice that the changes of variables $t:=-u$ in the first integral and $t:=1+u$ in the third one, respectively, make the two outer integrals coincide. Hence,
\[
\int_{\R\setminus\left((-\delta, \delta) \cup (1-\delta, 1+\delta)\right)}\frac{t(t-1)}{|t|^{5/2}|t-1|^{5/2}}\, dt = 2\int_\delta^{+\infty}\frac{du}{u^{3/2}(1+u)^{3/2}} -\int_\delta^{1-\delta}\frac{du}{u^{3/2}(1-u)^{3/2}}.
\]
Furthermore,  the two integrals at the right hand side be computed explicitly. Indeed, we have
\[
\int_\delta^{+\infty}\frac{du}{u^{3/2}(1+u)^{3/2}} = -4+ \frac{2(2\delta+1)}{\sqrt{\delta(1+\delta)}}\quad\mbox{ and }\quad\int_\delta^{1-\delta}
\frac{du}{u^{3/2}(1-u)^{3/2}} = \frac{4(1-2\delta)}{\sqrt{\delta(1-\delta)}}.
\]
Thus, it follows that
\[
\int_{\R\setminus\left((-\delta, \delta) \cup (1-\delta, 1+\delta)\right)}\frac{t(t-1)}{|t|^{5/2}|t-1|^{5/2}} \,dt = -8+\frac{4}{\sqrt{\delta}}\left(\frac{1+2\delta}{\sqrt{1+\delta}}-\frac{1-2\delta}{\sqrt{1-\delta}}\right).
\]
Accordingly, we have
\[
\lim_{\delta\searrow0}\int_{\R\setminus\left((-\delta, \delta) \cup (1-\delta, 1+\delta)\right)}\frac{t(t-1)}{|t|^{5/2}|t-1|^{5/2}} \,dt=-8.
\]
By using this information in \eqref{lòkmjnbhgvcfdxszxdcfvgb}, we conclude that
\[
I_A(x,y)
=
-\frac{8A}{|x-y|^2}.
\]
Now, from \eqref{DNSDEFN} and \eqref{,mknjbvcxzdxcfvbn}, we get
\[
c_{1,1/2}=-\frac{2\Gamma(1)}{\sqrt{\pi}\Gamma(-1/2)}=\frac1\pi\quad\mbox{ and }\quad  d_{1,1/2}=\frac{2^{1/2}\Gamma(5/4)}{\sqrt{\pi}\Gamma(1/4)} = \frac{1}{2\sqrt{2\pi}}.
\]
Combining these information with \eqref{,kmjnhbgvfctdxrsza<zsxdcvbnm,mnhgfds}, we obtain that
\[
I_A(x,y) = \frac{c_{1,1/2}}{d_{1,1/2}^2|x-y|^2}\left(\tr(A)-2\frac{(x-y)\cdot A(x-y)}{|x-y|^2}\right) = -\frac{8A}{|x-y|^2}.
\]
Thus,  \eqref{PARTE1COR1.5} follows by \eqref{INT2PVdefn}, and Theorem \ref{thmreg2}. This concludes the proof.
\end{proof}

\section*{Acknowledgments}
The authors thank Daniel E. Spector for the helpful discussions and insights, and for bringing reference \cite{LoRodBOH} to their attention.

The authors are members of the Australian Mathematical Society (AustMS). CS
is member of INdAM-GNAMPA.

This work has been supported by the Australian Laureate Fellowship FL190100081 and
the Australian Future Fellowship FT230100333.

CS acknowledges the support of the Juan de la Cierva Fellowship (grant number JDC2023-050365-I).

CS gratefully acknowledges the support of the Australian Laureate Fellowship FL190100081 for her research stay at UWA, where part of this research was carried out.

{\small \def\cdprime{$''$}
}

\end{document}